\documentclass[a4paper,12pt]{article}

\usepackage[utf8]{inputenc}
\usepackage[english]{babel}
\usepackage{amsmath, amssymb, amsthm}
\usepackage{enumerate}
\usepackage[colorlinks=true,linkcolor=blue,citecolor=blue]{hyperref}

\usepackage{tikz}
\usepackage{mdframed}
\usepackage{pgfplots}
\pgfplotsset{width=7cm,compat=1.10}
\pgfplotsset{
  /pgfplots/colormap={pink}{%
    color(0cm) = (blue);
    color(1cm) = (cyan!50!blue);
    color(2cm) = (cyan!50);
    color(3cm) = (cyan) }
}
\usetikzlibrary{calc,arrows,decorations.pathreplacing,fadings,3d,positioning}

\title{Geometry of Diophantine exponents
       \thanks{This research has been financed by the Russian Science Foundation grant 22-21-00079, https://rscf.ru/project/22-21-00079/}}
\author{Oleg\,N.\,German}
\date{}

\theoremstyle{definition}
\newtheorem{definition}{Definition}

\newtheorem*{notation*}{Notation}

\theoremstyle{remark}

\newtheorem*{remark*}{Remark}

\theoremstyle{plain}
\newtheorem{theorem}{Theorem}
\newtheorem{lemma}{Lemma}

\newtheorem{proposition}{Proposition}
\newtheorem{corollary}{Corollary}
\newtheorem*{statement*}{Statement}
\newtheorem*{corollary*}{Corollary}

\newtheorem*{littlewood}{Littlewood's Conjecture}
\newtheorem*{cassels}{Cassels--Swinnerton-Dyer's Conjecture}

\DeclareMathOperator{\vol}{vol}

\DeclareMathOperator{\conv}{conv}

\renewcommand{\phi}{\varphi}

\renewcommand{\vec}[1]{\mathbf{#1}}
\renewcommand{\geq}{\geqslant}
\renewcommand{\leq}{\leqslant}

\newcommand{\e}{\varepsilon}

\newcommand{\R}{\mathbb{R}}
\newcommand{\Z}{\mathbb{Z}}
\newcommand{\Q}{\mathbb{Q}}
\newcommand{\N}{\mathbb{N}}

\newcommand{\La}{\Lambda}

\newcommand{\bpsi}{\underline{\phi}}
\newcommand{\apsi}{\overline{\phi}}
\newcommand{\bPsi}{\underline{\Phi}}
\newcommand{\aPsi}{\overline{\Phi}}

\newcommand{\cB}{\mathcal{B}}
\newcommand{\cC}{\mathcal{C}}

\newcommand{\cF}{\mathcal{F}}
\newcommand{\cG}{\mathcal{G}}
\newcommand{\cH}{\mathcal{H}}

\newcommand{\cK}{\mathcal{K}}
\newcommand{\cL}{\mathcal{L}}
\newcommand{\cM}{\mathcal{M}}
\newcommand{\cO}{\mathcal{O}}
\newcommand{\cP}{\mathcal{P}}
\newcommand{\cQ}{\mathcal{Q}}
\newcommand{\cS}{\mathcal{S}}
\newcommand{\cT}{\mathcal{T}}
\newcommand{\cV}{\mathcal{V}}
\newcommand{\cW}{\mathcal{W}}

\newcommand{\gT}{\mathfrak{T}}

\newcommand{\tr}[1]{{#1}^\intercal}

\newcommand{\Glin}{G_{\textup{lin}}}
\newcommand{\Gsim}{G_{\textup{sim}}}

\newcommand{\starv}{\textup{St}_{\vec v}}
\newcommand{\starw}{\textup{St}_{\vec w}}

\begin{document}

\maketitle


\begin{abstract}
  Diophantine exponents are ones of the simplest quantitative characteristics responsible for the approximation properties of linear subspaces of a Euclidean space. This survey is aimed at describing the current state of the area of Diophantine approximation which studies Diophantine exponents and relations they satisfy. We discuss classical Diophantine exponents arising in the problem of approximating zero with the set of the values of several linear forms at integer points, their analogues in Diophantine approximation with weights, multiplicative Diophantine exponents, and Diophantine exponents of lattices. We pay special attention to the transference principle.
\end{abstract}


\tableofcontents

\section{Introduction}


Given a square with side length 1, is the length of its diagonal equal to a ratio of two integers? Can one construct a circle with the same area as that square using straightedge and compass? These questions date back to the times of Ancient Greece -- about two and a half thousand years ago. The Greeks answered the first question immediately, whereas the second one remained open until only little more than a century ago.


The answer to the first question proved the existence of irrational numbers. However, explicit examples of such numbers, as a rule, were algebraic, i.e. they were roots of polynomials with rational coefficients. And until the 19-th century, it had been unclear whether there exist non-algebraic, i.e. transcendental, numbers. It was Liouville who managed to answer this question\footnote{Of course, it is obvious nowadays that transcendental numbers do exist, as the set of algebraic numbers is countable and therefore it is a set of zero measure. But the thing is that cardinality theory was developed by Cantor only in the seventies of the 19-th century. Before him, such concepts had been unavailable. Thus, the result by Liouville, which he obtained in 1844, was truly outstanding.} by showing that algebraic numbers cannot be approximated by rationals ``too well'', after which he easily constructed an example of a transcendental number. This is how the theory of \emph{Diophantine approximation} was born. It was this theory that provided a proof of the transcendence of $\pi$ and, therefore, an answer to the second of the two questions mentioned above -- that it is impossible to construct a circle with the same area as a given square using straightedge and compass.


Hence came the understanding that real numbers can be ranked by their ``irrationality measure'' -- the better an irrational number can be approximated by rationals, the more irrational it is considered. The simplest quantitative characteristic of how well an irrational number $\theta$ can be approximated by rationals, is its \emph{Diophantine exponent} -- the supremum of real $\gamma$ such that the inequality
\[ 
  \bigg|\theta-\frac pq\bigg|\leqslant\frac{1}{q^{1+\gamma}}
\]
admits infinitely many solutions in integer $p$, $q$.


The problem of approximation of a real number with rationals has a very natural geometric interpretation, which helps working with more general problems -- when we need to find so called \emph{simultaneous approximations}, i.e. when several numbers are to be approximated by rationals with \emph{the same} denominator. In this context there also appear \emph{Diophantine exponents} as the simplest measure of deviation from ``rationality''.


It is worth mentioning that the problem of simultaneous approximations to the powers of a single number $\theta$, i.e. to the numbers $1,\,\theta,\,\theta^2,\,\theta^3,\ldots,\,\theta^n$, provides one of the most important tools for studying algebraic numbers. For instance, in his proof of transcendence of $e$, Hermite actually constructed ``good'' simultaneous approximations to the powers of $e$.


For the past several years, the theory of Diophantine approximation experienced a significant progress: the number of new theorems exceeded the number of those proved before. In this survey we tried to gather as many results as possible concerning various Diophantine exponents -- and there are more than a dozen types of them nowadays. We have no claim for exhaustiveness and we provide only the proofs that are concise enough and essential for understanding the methods of working with the objects under study. Detailed proofs can be found in original papers. The multidimensional setting will be of most interest to us, as multidimensionality delivers a variety of natural ways to define Diophantine exponents. We pay special attention to the so called \emph{transference principle}, which connects ``dual'' problems. We note also that we almost do not mention \emph{intermediate} Diophantine exponents, since a proper account of what is currently known about them would double the size of the paper.

\section{Approximation of a real number by rationals}

\subsection{Diophantine exponent and irrationality measure}

In Introduction we actually gave the following

\begin{definition} \label{def:beta_1_dim}
  Let $\theta$ be a real number. Its \emph{Diophantine exponent} $\omega(\theta)$ is defined as the supremum of real $\gamma$ such that the inequality
  \begin{equation} \label{eq:beta_1_dim}
    \bigg|\theta-\frac pq\bigg|\leqslant\frac{1}{q^{1+\gamma}}
  \end{equation}
  admits infinitely many solutions in integer $p$, $q$.
\end{definition}

In the case of a single number, it is more traditional, however, to talk about the \emph{irrationality measure} of $\theta$.

\begin{definition} \label{def:irrationality_measure}
  Let $\theta$ be a real number. Its \emph{irrationality measure} $\mu(\theta)$ is defined as the supremum of real $\gamma$ such that the inequality
  \begin{equation} \label{eq:irrationality_measure}
    \bigg|\theta-\frac pq\bigg|\leqslant\frac{1}{q^\gamma}
  \end{equation}
  admits infinitely many solutions in integer coprime $p$, $q$.
\end{definition}

Definition \ref{def:irrationality_measure} differs from \ref{def:beta_1_dim} only by the absence of 1 in the exponent in \eqref{eq:irrationality_measure} and by the condition that $p$ and $q$ should be coprime. Thus, for an irrational $\theta$ we have
\[ 
  \mu(\theta)=\omega(\theta)+1.
\]
Whereas, if $\theta\in\Q$, then $\omega(\theta)=\infty$, $\mu(\theta)=1$ (see Proposition \ref{prop:beta_for_rationals} below).

Most of results describing the measure of deviation of a number from rationality are usually stated in terms of $\mu(\theta)$. However, the most reasonable way to define Diophantine exponents in the multidimensional setting gives exactly $\omega(\theta)$ in the particular one-dimensional case -- not $\mu(\theta)$. Therefore, we prefer to formulate all our statements in terms of $\omega(\theta)$.

\subsection{Dirichlet's theorem}

As a rule, talks on the elements of Diophantine approximation start with mentioning the respective Dirichlet theorem. Let us keep to this tradition. Being rather elementary, Dirichlet's theorem on the approximation of real numbers by rationals is as well a most fundamental statement contained, in this form or another, in most of the existing theorems on Diophantine approximation.

\begin{theorem}[G. Lejeune Dirichlet, 1842] \label{t:dirichlet}
  Let $\theta$ and $t$ be real numbers, $t>1$. Then there are integers $p$ and $q$ such that $0<q<t$ and
  \[ 
    \bigg|\theta-\frac pq\bigg|\leq\frac{1}{qt}\,.
  \]
\end{theorem}

The original statement of the theorem proved in Dirichlet's paper \cite{dirichlet} of the year 1842 is actually different. Ir much more resembles the statement of Corollary \ref{cor:dirichlet_irr} below and it was given for a linear form in arbitrarily many variables (see also Theorems \ref{t:dirichlet_linear_form}, \ref{t:dirichlet_simultaneous}). However, Theorem \ref{t:dirichlet} immediately follows from Dirichlet's original argument, and it so happened that this way of formulating Dirichlet's theorem is considered to be classical.

There are two most known proofs of this theorem -- an arithmetic one, involving the pigeon-hole principle, and a geometric one, involving Minkowski's convex body theorem. The latter approach is discussed in more detail in Section \ref{subsec:geometry_of_dirichlet}.

For irrational $\theta$ Theorem \ref{t:dirichlet} provides an estimate, to some extent trivial, for the Diophantine exponent of an irrational number.

\begin{corollary} \label{cor:dirichlet_irr}
  Let $\theta$ be an irrational real number. Then there are infinitely many pairs of coprime integers $p$, $q$ such that
  \[ 
    \bigg|\theta-\frac pq\bigg|<\frac{1}{q^2}\,.
  \]
\end{corollary}

\begin{corollary} \label{cor:beta_for_irr}
  If $\theta$ is irrational, then $\omega(\theta)\geq1$, $\mu(\theta)\geq2$.
\end{corollary}

The irrationality assumption cannot be omitted, as is shown by the following

\begin{proposition} \label{prop:beta_for_rationals}
  If $\theta$ is rational, then $\omega(\theta)=\infty$, $\mu(\theta)=1$.
\end{proposition}

\begin{proof}
  Let $\theta=a/b$, $a,b\in\Z$, $b>0$.

  Then for every $k\in\N$ and every $\gamma\in\R$
  \[ 
    \bigg|\theta-\frac{ka}{kb}\bigg|=0<\frac{1}{q^\gamma}\,, 
  \]
  whence it immediately follows that $\omega(\theta)=\infty$.

  Furthermore, for every rational $p/q$ distinct from $\theta$ we have
  \[ 
    \bigg|\theta-\frac pq\bigg|=\frac{|aq-bp|}{bq}\geq\frac{1/b}{q}\,,
  \]
  whence it is clear that for every positive $\gamma$ inequality \eqref{eq:beta_1_dim} admits only a finite number of solutions. Thus, $\mu(\theta)\leq1$. Finally, we come to $\mu(\theta)=1$ by noticing that the linear Diophantine equation
  \[ 
    ax-by=1 
  \]
  with coprime $a$ and $b$ admits infinitely many solutions.
\end{proof}

\subsection{Diophantine and Liouville numbers}

In 1844, two years after Dirichlet's paper was published, Liouville \cite{liouville_44} constructed the first examples of transcendental numbers. The main idea of his proof led to the following statement, which was published in 1851 in \cite{liouville_51}.

\begin{theorem}[J.\,Liouville, 1844--1851] \label{t:liouville}
  Let $\theta$ be an algebraic number of degree $n$. Then there is a positive constant $c$ depending only on $\theta$ such that for every integer $p$ and every positive integer $q$ we have
  \[ 
    \bigg|\theta-\frac pq\bigg|>\frac{c}{q^n}\,. 
  \]
\end{theorem}

Theorem \ref{t:liouville} immediately implies an estimate for the irrationality measure (as well as for the Diophantine exponent) of an algebraic number, which generalises the estimate provided by Proposition \ref{prop:beta_for_rationals}.

\begin{corollary} \label{cor:beta_for_alg_by_liouville}
  If $\theta$ is an algebraic number of degree $n$, then $\mu(\theta)\leq n$. Respectively, for $n\geq 2$, we have $\omega(\theta)\leq n-1$.
\end{corollary}

Hence we conclude that if $\mu(\theta)=\infty$ (which, for irrational $\theta$, is equivalent to $\omega(\theta)=\infty$), then $\theta$ cannot be algebraic, therefore, it is transcendental.

\begin{definition}
   Let $\theta$ be an irrational number. If $\omega(\theta)=\infty$, then $\theta$ is called a \emph{Liouville number}. If $\omega(\theta)<\infty$, then $\theta$ is called a \emph{Diophantine number}.
\end{definition}

Thus, algebraic numbers are Diophantine. But what are the possible values of their Diophantine exponents? Corollaries \ref{cor:beta_for_irr} and \ref{cor:beta_for_alg_by_liouville} give us an analogue of Proposition \ref{prop:beta_for_rationals} for quadratic irrationalities.

\begin{proposition} \label{prop:beta_for_quadratic_irr}
  If $\theta$ is a quadratic irrationality, i.e. an algebraic number of degree $2$, then $\omega(\theta)=1$, $\mu(\theta)=2$.
\end{proposition}


For algebraic numbers of higher degrees Liouville's theorem had been improved upon subsequently by Thue \cite{thue}, Siegel \cite{siegel}, Dyson \cite{dyson_roth}, Gelfond \cite{gelfond}, until Roth \cite{roth} obtained in 1955 a result, for which he was awarded the Fields Medal in 1958.

\begin{theorem}[K.\,Roth, 1955] \label{t:roth}
  If $\theta$ is an irrational algebraic number, then for every positive $\e$ the inequality
  \[ 
    \bigg|\theta-\frac pq\bigg|<\frac{1}{q^{2+\e}} 
  \]
  has no more than a finite number of solutions in integer $p$, $q$.
\end{theorem}

\begin{corollary} \label{cor:roth}
  If $\theta$ is an irrational algebraic number, then $\omega(\theta)=1$, $\mu(\theta)=2$.
\end{corollary}

Thus, irrational algebraic numbers have the smallest possible value of the Diophantine exponent. Which means that algebraic numbers are badly approximable by rationals. However, the classical notion of \emph{badly approximable} numbers is slightly more subtle. It assumes that the strongest possible way to improve on Corollary \ref{cor:dirichlet_irr} is by decreasing the constant.

\begin{definition} \label{def:BA}
    An irrational number $\theta$ is called \emph{badly approximable (by rationals)}, if there is a constant $c>0$ such that for every rational $p/q$ we have
  \[ 
    \bigg|\theta-\frac pq\bigg|\geq\frac{c}{q^2}\,. 
  \]
\end{definition}

It follows from Liouville's theorem that algebraic numbers of degree $2$ are badly approximable. It is an open question whether algebraic numbers of higher degrees are badly approximable as well. By now, even Lang's conjecture (1945) is still open. It claims that for every irrational algebraic number $\theta$ there exists $\delta>1$ such that the inequality
\[ 
  \bigg|\theta-\frac pq\bigg|<\frac{1}{q^2(\log q)^\delta} 
\]
has no more than a finite number of solutions in integer $p$, $q$.

\subsection{Geometric interpretation of Dirichlet's theorem} \label{subsec:geometry_of_dirichlet}

Let us consider the line $\cL(\theta)$ in $\R^2$ passing through the origin and the point $(1,\theta)$. The set of points $(x,y)$ satisfying the inequality
\[ 
  |x(\theta x-y)|<1
\]
forms a ``hyperbolic cross'' $\cH(\theta)$ with the line $\cL(\theta)$ and the ordinate axis as asymptotes. As for the set determined by the system
\[ 
  \begin{cases}
    |x|\leq t \\
    |\theta x-y|\leq 1/t
  \end{cases}, 
\]
it is a parallelogram $\cP(\theta,t)$ with sides parallel to those two lines (see Fig.\ref{fig:dirichlet}).

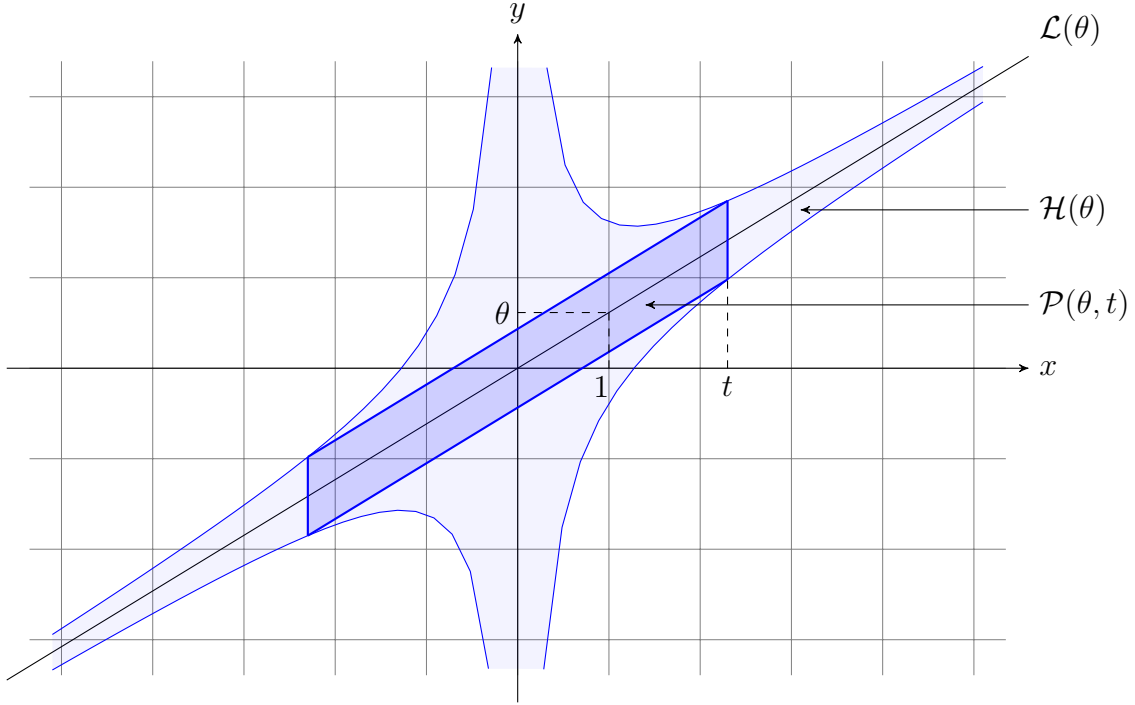
\begin{figure}[h]
\centering
\begin{tikzpicture}[domain=-5.1:5.1,scale=1.2]
    \draw[very thin,color=gray,scale=1] (-5.35,-6*8/13+0.3) grid (5.35,6*8/13-0.3);

    \draw[->,>=stealth'] (-5.6,0) -- (5.6,0) node[right] {$x$};
    \draw[->,>=stealth'] (0,-6*8/13) -- (0,6*8/13) node[above] {$y$};

    \draw[color=black] plot[domain=-5.6:5.6] (\x, {8*\x/13}) node[above right]{$\cL(\theta)$};

    \draw[color=blue] plot[domain=0.32:5.1]  (\x, {1/(\x)+8*\x/13});
    \draw[color=blue] plot[domain=0.2855:5.1]  (\x, {-1/(\x)+8*\x/13});
    \draw[color=blue] plot[domain=-5.1:-0.32]  (\x, {1/(\x)+8*\x/13});
    \draw[color=blue] plot[domain=-5.1:-0.2855]  (\x, {-1/(\x)+8*\x/13});

    \fill[fill=blue,opacity=0.05]
        plot [domain=0.32:5.1] (\x,{1/(\x)+8*\x/13}) --
        plot [domain=5.1:0.2855] (\x, {-1/(\x)+8*\x/13}) --
        plot [domain=-0.32:-5.1] (\x,{1/(\x)+8*\x/13}) --
        plot [domain=-5.1:-0.2855] (\x, {-1/(\x)+8*\x/13}) --
        cycle;

    \draw[thick,color=blue] plot[domain=-2.3:2.3] (\x, {8*\x/13+1/2.3});
    \draw[thick,color=blue] plot[domain=-2.3:2.3] (\x, {8*\x/13-1/2.3});
    \draw[thick,color=blue] plot[domain=-1/2.3+8*2.3/13:1/2.3+8*2.3/13] ({2.3}, \x);
    \draw[thick,color=blue] plot[domain=-1/2.3-8*2.3/13:1/2.3-8*2.3/13] ({-2.3}, \x);

    \fill[fill=blue,opacity=0.15]
        plot[domain=-2.3:2.3] (\x, {8*\x/13+1/2.3}) --
        plot[domain=2.3:-2.3] (\x, {8*\x/13-1/2.3}) --
        cycle;

    \draw[dashed] (2.3,0) node[below,yshift=0.03cm]{$t$} -- (2.3,-1/2.3+8*2.3/13);
    \draw[dashed] (1,0) node[below,xshift=-0.1cm,yshift=0.03cm]{$1$} -- (1,8/13);
    \draw[dashed] (0,8/13) node[left,xshift=0.05cm]{$\theta$} -- (1,8/13);

    \draw[->,>=stealth',color=black,thin] (5.6,1.75) node[right]{$\cH(\theta)$} -- (3.1,1.75);
    \draw[->,>=stealth',color=black,thin] (5.6,0.7) node[right]{$\cP(\theta,t)$} -- (1.4,0.7);

\end{tikzpicture}
\caption{For Dirichlet's theorem} \label{fig:dirichlet}
\end{figure}

Clearly, Corollary \ref{cor:dirichlet_irr} to Dirichlet's theorem states that $\cH(\theta)$ contains infinitely many points $(x,y)\in\Z^2$ with nonzero $x$, and Dirichlet's theorem itself (Theorem \ref{t:dirichlet}) states that $\cP(\theta,t)$ contains a nonzero point of $\Z^2$ for \emph{every} $t>1$. Thus formulated, Dirichlet's theorem becomes a particular case of the following classical \emph{Minkowski convex body theorem} published in \cite{minkowski} (see also \cite{cassels_GN}, \cite{schmidt_DA}).

\begin{theorem}[H.\,Minkowski, 1896] \label{t:minkowski_convex}
  Let $\cM$ be a convex centrally symmetric closed body in $\R^d$ centered at the origin. Suppose that the volume of $\cM$ is not less than $2^d$. Then $\cM$ contains a nonzero point of $\Z^d$.
\end{theorem}

\begin{corollary}
\label{cor:minkowski_linear_forms}
  Let $L_1,\ldots,L_d$ be a $d$-tuple of homogeneous linear forms in $\R^d$ with determinant $D\neq0$ and let $\delta_1,\ldots,\delta_n$ be positive numbers whose product equals $|D|$. Then there is a point $\vec v\in\Z^d\backslash\{\vec 0\}$ such that
  \[ 
    |L_1(\vec v)|\leq\delta_i,\quad|L_i(\vec v)|<\delta_i,\quad i=2,\ldots,d. 
  \]
\end{corollary}

It is worth mentioning that often, instead of Minkowski's theorem, Corollary \ref{cor:minkowski_linear_forms} is applied. This Corollary bears the name of \emph{Minkowski's theorem for linear forms}. Clearly, Dirichlet's theorem requires only Corollary \ref{cor:minkowski_linear_forms}.


\subsection{Relation to continued fractions}\label{subsec:continued_fractions}


Let us remind the algorithm of continued fraction expansion of a real number $\theta$. Denoting by $[\,\cdot\,]$ the integer part, let us define the sequences $(\alpha_k)$, $(a_k)$, $(p_k)$, $(q_k)$ by
\begin{equation} \label{eq:CF_attributes}
  \begin{aligned}
    \alpha_0=\theta,\quad\alpha_{k+1}=(\alpha_k-[\alpha_k])^{-1},\quad a_k=[\alpha_k], \\ 
    \begin{aligned}
      p_{-2} & =0, \quad\,\ p_{-1}=1, \\
      q_{-2} & =1, \quad\,\ q_{-1}=0,
    \end{aligned}
    \quad\,\
    \begin{aligned}
      p_k & =a_kp_{k-1}+p_{k-2}, \\
      q_k & =a_kq_{k-1}+q_{k-2},
    \end{aligned}
  \end{aligned}
  \qquad k=0,1,2,\ldots
\end{equation}
These sequences are infinite, provided $\theta$ is irrational. If $\theta$ is rational, they stop as soon as $\alpha_k$ becomes integer.

Then (see \cite{khintchine_CF}, \cite{schmidt_DA}), for every $k=0,1,2,\ldots$ such that $\alpha_k$ is correctly defined, we have
\[ 
  \theta=[a_0;a_1,\ldots,a_{k-1},\alpha_k]=
  a_0+\cfrac{1}{a_1+\cfrac{1}{\stackrel{\ddots}{\phantom{|}}+\,\cfrac{1}{a_{k-1}+\cfrac{1}{\alpha_k}}}}\,,\qquad
  \frac{p_k}{q_k}=[a_0;a_1,\ldots,a_k]. 
\]
The numbers $a_k$ are called \emph{partial quotients} of $\theta$ and $p_k/q_k$ are called \emph{convergents} of $\theta$.

\subsubsection{Diophantine exponent and growth of partial quotients}\label{subsubsec:exponent_vs_growth}

The convergents of $\theta$ satisfy (see \cite{lang}, \cite{khintchine_CF}, \cite{schmidt_DA}) the relation
\[ 
  \bigg|\theta-\frac pq\bigg|\leq\frac{1}{q^2}\,. 
\]
On the other hand, if a rational $p/q$ satisfies this inequality, then by Fatou's theorem (see \cite{fatou}, \cite{grace}, \cite{lang}) $p/q$ coincides either with a convergent of $\theta$, or with a mediant neighbouring a convergent. Moreover, by Legendre's theorem (see \cite{lang}, \cite{khintchine_CF}, \cite{schmidt_DA}) any rational number $p/q$ satisfying the inequality
\[ 
  \bigg|\theta-\frac pq\bigg|\leq\frac{1}{2q^2} 
\]
coincides with a convergent of $\theta$.

Thus, the order of approximation of $\theta$ by rationals is determined by how fast the differences between $\theta$ and its convergents decrease. For these differences there exist classical estimates (see \cite{lang}, \cite{khintchine_CF}, \cite{schmidt_DA}).

\begin{proposition} \label{prop:CF_gap_estimates}
  Let $\theta=[a_0;a_1,a_2,\ldots]$ be the continued fraction expansion of $\theta$ and let $p_k/q_k=[a_0;a_1,\ldots,a_k]$ be its convergents. Then
  \[ 
    \text{a) }\ \frac{1}{q_k(q_k+q_{k+1})}<\bigg|\theta-\frac{p_k}{q_k}\bigg|\leq\frac{1}{q_kq_{k+1}}\,, 
  \]
  \[ 
    \text{\,b) }\ \ \frac{1}{q_k^2(a_{k+1}+2)}<\bigg|\theta-\frac{p_k}{q_k}\bigg|\leq\frac{1}{q_k^2a_{k+1}}\,. 
  \]
\end{proposition}

\begin{corollary}\label{cor:BA_vs_boundedness_of_partial_quotients}
  An irrational number is badly approximable if and only if its partial quotients are bounded.
\end{corollary}

\begin{corollary}\label{cor:omega_vs_growth}
  Given an irrational number $\theta$, within the notation of Proposition \ref{prop:CF_gap_estimates}, we have
  \begin{equation}\label{eq:omega_vs_growth}
    \omega(\theta)=\limsup_{k\to+\infty}\frac{q_{k+1}}{q_k}=1+\limsup_{k\to+\infty}\frac{a_{k+1}}{q_k}\,.
  \end{equation}
\end{corollary}

Thus, the order of approximation of $\theta$ by rationals is determined by the growth of the partial quotients of $\theta$.

\subsubsection{Geometric algorithm}\label{subsubsec:geometric_algorithm}

The approach described in Section \ref{subsec:geometry_of_dirichlet} enables interpreting a continued fraction as a geometric object (see also \cite{delone}, \cite{arnold_mccme}, \cite{erdos_gruber_hammer}, \cite{karpenkov_book}, \cite{german_tlyustangelov_MJCNT_2016}).
Let us define sequences of real numbers $(\beta_k)$, $(b_k)$ and a sequence of lattice points $(\vec v_k)$. Set
\[ 
  \vec v_{-2}=(1,0),\qquad\vec v_{-1}=(0,1). 
\]
Given $\vec v_{k-2}$ and $\vec v_{k-1}$, define $\beta_k$, $b_k$, and $\vec v_k$ by
\[ 
  \vec v_{k-2}+\beta_k\vec v_{k-1}\in\cL(\theta),\qquad b_k=[\beta_k],\qquad\vec v_k=\vec v_{k-2}+b_k\vec v_{k-1}. 
\]
In other words, $\vec v_k$ is the ultimate point of $\Z^2$ before crossing the line $\cL(\theta)$ on the way from $\vec v_{k-2}$ in the direction determined by\footnote{Here and henceforth we do not distinguish between the concepts of a point and of its radius vector} $\vec v_{k-1}$.
\begin{figure}[h]
\centering
\begin{tikzpicture}

    \draw[color=black] plot[domain=-0.5:8.2] (\x, {8*\x/13}) node[above right]{$\cL(\theta)$};

    \draw[color=blue,dashed] (1-0.3*2.6,3-0.3*0.6) -- (1+2.75*2.6,3+2.75*0.6);

    \draw[color=blue,thick] (0,0) -- (2.6,0.6);
    \draw[color=blue,thick] (0,0) -- (1,3);

    \node[fill=blue,circle,inner sep=1.8pt] at (0,0) {};
    \node[fill=blue,circle,inner sep=1.8pt] at (1,3) {};
    \node[fill=blue,circle,inner sep=1.8pt] at (2.6,0.6) {};
    \node[fill=blue,circle,inner sep=1.8pt] at (1+1*2.6,3+1*0.6) {};
    \node[fill=blue,circle,inner sep=1.8pt] at (1+2*2.6,3+2*0.6) {};

    \node[fill=white,draw=blue,thick,circle,inner sep=1.5pt] at (1+2.382*2.6,3+2.382*0.6) {};

    \draw (-0.1,0) node[below right]{$\vec 0$};
    \draw (0.8+2*2.6,3.04+2*0.6) node[above]{$\vec v_k$};
    \draw (2.6,0.6) node[below right] {$\vec v_{k-1}$};
    \draw (1,3) node[below right] {$\vec v_{k-2}$};
   \draw (1+2.382*2.6,3+2.382*0.6) node[below right]{$\vec v_{k-2}+\beta_k\vec v_{k-1}$};
\end{tikzpicture}
\caption{Construction of $\vec v_k$} \label{fig:next_point}
\end{figure}
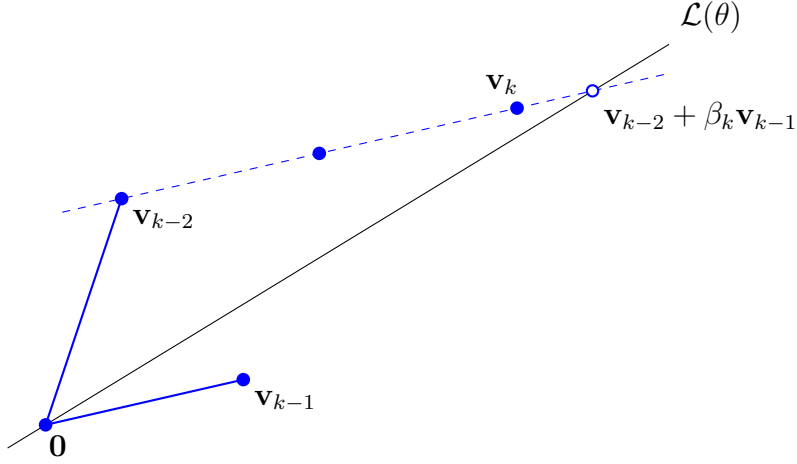

It is impossible to construct the point $\vec v_k$ if and only if $\vec v_{k-1}$ is on $\cL(\theta)$ -- in this case the three sequences stop. If there are no nonzero points of $\Z^2$ on $\cL(\theta)$, the sequences are infinite.

\begin{theorem} \label{t:CF_geometry}
  For each $k$ we have
  \[ 
    \begin{aligned}
      \textup{a) } & \det(\vec v_{k-1},\vec v_k)=(-1)^{k-1} \\
      \textup{b) } & \det(\vec v_{k-2},\vec v_k)=(-1)^kb_k \\
      \textup{c) } & \beta_k=\alpha_k,\ b_k=a_k,\ \vec v_k=(q_k,p_k).
    \end{aligned} 
  \]
\end{theorem}

\begin{proof}
  
  Statements (a) and (b) follow from the fact that both determinant and the relation $\vec v_k=\vec v_{k-2}+b_k\vec v_{k-1}$ are linear.
  
  In order to prove (c), let us express $\beta_{k+1}$ in terms of $\beta_k$. Vectors $\vec v_{k-2}+\beta_k\vec v_{k-1}$ and $\vec v_{k-1}+\beta_{k+1}\vec v_k$ are collinear, whence by (a) and (b) we get
  \[ 
    0=\det(\vec v_{k-2}+\beta_k\vec v_{k-1},\vec v_{k-1}+\beta_{k+1}\vec v_k)=(-1)^k+(-1)^{k-1}\beta_k\beta_{k+1}+(-1)^kb_k\beta_{k+1}. 
  \]
  Thus, $1-\beta_{k+1}(\beta_k-b_k)=0$, i.e.
  \[ 
    \beta_{k+1}=(\beta_k-[\beta_k])^{-1}, 
  \]
  which coincides with the respective formula for $\alpha_k$, $\alpha_{k+1}$ and gives us the induction step. Finally, we note that
  \[ 
    \beta_0=\theta=\alpha_0,\quad\vec v_{-2}=(1,0)=(q_{-2},p_{-2}),\quad\vec v_{-1}=(0,1)=(q_{-1},p_{-1})
  \]
  and make use of \eqref{eq:CF_attributes}. The argument is complete.
\end{proof}

Statements (a) and (b) of Theorem \ref{t:CF_geometry} provide a geometric interpretation of the well-known relations
\[
  \begin{aligned}
    p_kq_{k-1}-p_{k-1}q_k & =(-1)^{k-1}, \\
    p_kq_{k-2}-p_{k-2}q_k & =(-1)^ka_k.
  \end{aligned} 
\]
Statement (a) means that for each $k$ the vectors $\vec v_{k-1}$, $\vec v_k$ form a basis of $\Z^2$, respectively oriented. Statement (b) means that the integer length of the segment with endpoints at $\vec v_{k-2}$ and $\vec v_k$ equals $a_k$.

\begin{definition}
  The \emph{integer length} of a segment with endpoints in $\Z^2$ is defined as the number of minimal (w.r.t. inclusion) segments with endpoints in $\Z^2$ contained in it.
\end{definition}

It easily follows from the rule of constructing $\vec v_k$ from the two preceding points that for every 
$k\geq0$ the points $\vec v_{k-1}$ and $\vec v_k$ are separated by $\cL(\theta)$. Furthermore, the points with even nonnegative indices lie below $\cL(\theta)$, and the ones with odd indices lie above $\cL(\theta)$. This corresponds to the fact that all the convergents with even indices are smaller than $\theta$ and those with odd indices are greater than $\theta$.

\subsubsection{Geometric proofs}

Many statements concerning continued fractions can be proved geometrically. As an example, let us present an argument that proves inequalities (b) from Proposition \ref{prop:CF_gap_estimates}. Let us rewrite them as
\begin{equation} \label{eq:pre_two_parallelograms}
  \frac{1}{a_{k+1}+2}<q_k|\theta q_k-p_k|\leq\frac{1}{a_{k+1}}
\end{equation}
and notice that $q_k|\theta q_k-p_k|$ is nothing else but the area of the parallelogram $\cQ_1$ with vertices $\vec a$, $\vec v_k$, $\vec b$, $\vec 0$ (see Fig.\ref{fig:two_parallelograms}). Consider the parallelogram $\cQ_2$ with vertices $\vec 0$, $\vec c$, $\vec d$, $\vec e$ (see that same Fig.\ref{fig:two_parallelograms}).

\begin{figure}[h]
\centering
\begin{tikzpicture}
    \draw[->,>=stealth'] (-2.7,0) -- (10,0) node[right] {$x$};
    \draw[->,>=stealth'] (0,-1.2) -- (0,6) node[above] {$y$};

    \draw[color=black] plot[domain=-1.2:10.2] (\x, {8*\x/13}) node[above right]{$\cL(\theta)$};

    \draw[color=blue,dashed] (0,0) -- (1+2*2.6,3+2*0.6);
    \draw[color=blue,dashed] (0,0) -- (1,3);
    \draw[color=blue,dashed] (1-1.2*2.6,3-1.2*0.6) -- (1+3.2*2.6,3+3.2*0.6);
    \draw[color=blue,dashed] (1-1.2*2.6,0.6/2.6-1.2*0.6) -- (1+3.2*2.6,0.6/2.6+3.2*0.6);

    \draw[color=blue,very thick] (1,3) -- (1+2*2.6,3+2*0.6);
    \draw[decorate,decoration={brace,raise=3.5pt,amplitude=6pt}] (1,3) -- (1+2*2.6,3+2*0.6);

    \fill[blue,opacity=0.1]
        plot[domain=1-1.2*2.6:1+3.2*2.6] (\x, {0.6*\x/2.6}) --
        plot[domain=1+3.2*2.6:1-1.2*2.6] (\x, {0.6*\x/2.6+3-0.6/2.6}) --
        cycle;


    \draw[very thick] (2.6,0.6) -- (2.6,1.6);
    \draw[very thick] (0,-1) -- (2.6,0.6);
    \draw[dashed] (2.6,0) node[below]{$q_k$} -- (2.6,0.6);

    \draw[very thick] (0,3-0.6/2.6) -- (1+2.382*2.6,3-0.6/2.6+3+2.382*0.6);
    \draw[very thick] (1+2.382*2.6,3+2.382*0.6) -- (1+2.382*2.6,3-0.6/2.6+3+2.382*0.6);

    \draw[very thick] (0,-1) -- (0,3-0.6/2.6);
    \draw[very thick] (0,0) -- (1+2.382*2.6,3+2.382*0.6);

    \draw (2.6,0.6) node[below right]{$\vec v_k$};
    \draw (1.1+2*2.6,3.04+2*0.6) node[above]{$\vec v_{k+1}$};
    \draw (1,3) node[below right]{$\vec v_{k-1}$};
    \draw (1+2.6,3.25+0.6) node[above]{$a_{k+1}$};
    \draw (0,0) node[above left]{$\vec 0$};
    \draw (0,-1.1) node[left]{$\vec a$};
    \draw (2.6,1.4) node[right]{$\vec b$};
    \draw (1+2.382*2.6,3+2.382*0.6) node[below right]{$\vec c$};
    \draw (1+2.382*2.6,3-0.6/2.6+3+2.382*0.6) node[right]{$\vec d$};
    \draw (0,3-0.6/2.6) node[above left]{$\vec e$};
    \draw (1-2.6,3-0.6) node[above left]{$\vec v_{k-1}-\vec v_k$};
    \draw (1+3*2.6,3+3*0.6) node[below right]{$\vec v_{k+1}+\vec v_k$};

    \node[fill=blue,circle,inner sep=1.8pt] at (0,0) {};
    \node[fill=blue,circle,inner sep=1.8pt] at (2.6,0.6) {};
    \node[fill=blue,circle,inner sep=1.8pt] at (2*2.6,2*0.6) {};
    \node[fill=blue,circle,inner sep=1.8pt] at (3*2.6,3*0.6) {};
    \node[fill=blue,circle,inner sep=1.8pt] at (1,3) {};
    \node[fill=blue,circle,inner sep=1.8pt] at (1+1*2.6,3+1*0.6) {};
    \node[fill=blue,circle,inner sep=1.8pt] at (1+2*2.6,3+2*0.6) {};
    \node[fill=blue,circle,inner sep=1.8pt] at (1+3*2.6,3+3*0.6) {};
    \node[fill=blue,circle,inner sep=1.8pt] at (1-2.6,3-0.6) {};

    \node[fill=white,draw=blue,thick,circle,inner sep=1.5pt] at (1+2.382*2.6,3+2.382*0.6) {};
    \node[fill=white,draw=blue,thick,circle,inner sep=1.5pt] at (0,3-0.6/2.6) {};
    \node[fill=white,draw=blue,thick,circle,inner sep=1.5pt] at (1+2.382*2.6,3-0.6/2.6+3+2.382*0.6) {};
    \node[fill=white,draw=blue,thick,circle,inner sep=1.5pt] at (0,-1) {};
    \node[fill=white,draw=blue,thick,circle,inner sep=1.5pt] at (2.6,1.6) {};

    \draw[->,>=stealth',color=black,thin] (5,-0.6) node[right]{$\cQ_1$} -- (0.25,-0.6);
    \draw[->,>=stealth',color=black,thin] (1.5,6) node[left]{$\cQ_2$} -- (6.25,6);

\end{tikzpicture}
\caption{Parallelograms $\cQ_1$, $\cQ_2$ and the empty stripe} \label{fig:two_parallelograms}
\end{figure}
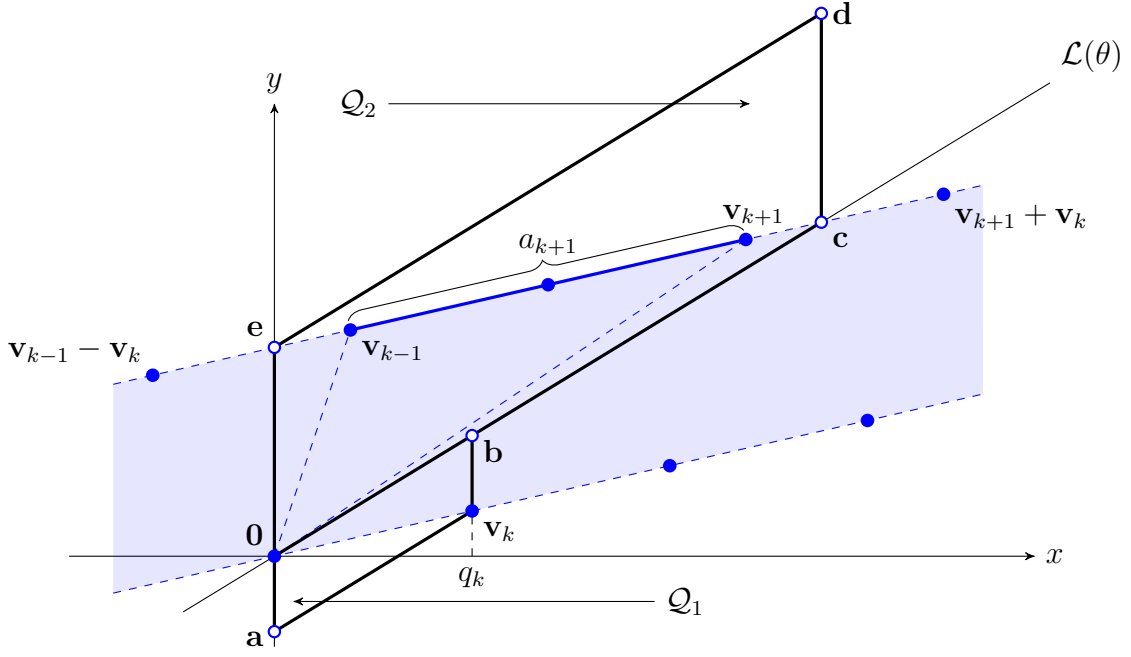

\begin{lemma} \label{l:two_parallelograms}
  The product of the areas of $\cQ_1$ and $\cQ_2$ equals $1$.
\end{lemma}

\begin{proof}
  Since $\vec v_k$, $\vec v_{k-1}$ form a basis of $\Z^2$, the width $H$ of the stripe in Fig.\ref{fig:two_parallelograms} equals $|\vec v_k|^{-1}$. Denote by $h$ the distance from $\vec b$ to the line spanned by $\vec v_k$. Then it follows from the similarity of the respective triangles that
  \[ 
    \frac{|\vec v_k|}{h}=\frac{|\vec c-\vec e|}{H}\,, 
  \]
  whence
  \[ 
    \vol\cQ_1\cdot\vol\cQ_2=|\vec v_k|h\cdot|\vec c-\vec e|H=(|\vec v_k|H)^2=1. 
  \]
\end{proof}

In view of Lemma \ref{l:two_parallelograms} inequalities \eqref{eq:pre_two_parallelograms} can be rewritten as
\[ 
  a_{k+1}\leq\vol Q_2<a_{k+1}+2, 
\]
i.e. as
\[ 
  |\det(\vec v_{k-1},\vec v_{k+1})|\leq|\det(\vec e,\vec c)|<|\det(\vec v_{k-1},\vec v_{k+1})|+2. 
\]
These inequalities immediately follow from the fact that the segment $[\vec e,\vec c]$ contains the segment $[\vec v_{k-1},\vec v_{k+1}]$ and is contained in the segment $[\vec v_{k-1}-\vec v_k,\vec v_{k+1}+\vec v_k]$.

This is how statement (b) of Proposition \ref{prop:CF_gap_estimates} is proved geometrically. Statement (a) can be proved similarly.

\subsubsection{Klein polygons}\label{subsubsec:klein_polygons}

The following theorem establishes the relation between continued fractions and a construction which goes back to Klein \cite{klein} and bears the name of \emph{Klein polygons}. Denote by $\cO_+$ the closure of the positive quadrant, i.e. the set of points in $\R^2$ with nonnegative coordinates.

\begin{theorem} \label{t:klein_polygons}
  Let $\theta$ be positive. Then the points $\vec v_k$ with even indices are the vertices of the convex hull of nonzero integer points lying in $\cO_+$ not above $\cL(\theta)$. Similarly, the points $\vec v_k$ with odd indices are the vertices of the convex hull of nonzero integer points lying in $\cO_+$ not below $\cL(\theta)$.
\end{theorem}

\begin{proof}
  Since $\theta>0$, the coordinates of every $\vec v_k$ are nonnegative. Besides that, we know that if $k$ is even, then $\vec v_k$ lies below $\cL(\theta)$ and if $k$ is odd, then $\vec v_k$ lies above $\cL(\theta)$. Therefore, it suffices to show that there are no nonzero points of $\Z^2$ in the intersection of the stripe shown in Fig.\ref{fig:two_parallelograms} with the angle between $\cL(\theta)$ and the respective coordinate axis, apart from those lying on the segment $[\vec v_{k-1},\vec v_{k+1}]$.
  
  For $k\geq0$ the representation
  \[ 
    \vec v_{k-1}-\vec v_k=\vec v_{k-1}-(a_k\vec v_{k-1}+\vec v_{k-2})=(1-a_k)\vec v_{k-1}-\vec v_{k-2} 
  \]
  implies that at least one of the coordinates of the point $\vec v_{k-1}-\vec v_k$ is negative. For $k=-1$ this is even more obvious.
  
  Thus, for every $k$ the point $\vec v_{k-1}-\vec v_k$ does not belong to $\cO_+$.
  
  As for the point $\vec v_{k+1}+\vec v_k$, by construction it does not belong to the closure of the angle which contains the segment $[\vec v_{k-1},\vec v_{k+1}]$.
  
  Taking into account that there are no points of $\Z^2$ in the interior of the stripe shown in Fig.\ref{fig:two_parallelograms}, we complete the proof.
\end{proof}

\begin{definition}\label{def:klein_polygons}
  The convex hulls discussed in Theorem \ref{t:klein_polygons} are called \emph{Klein polygons}.
\end{definition}

Thus, the continued fraction of $\theta$ is ``written'' on the boundaries of Klein polygons: the vertices have coordinates equal to the denominators and numerators of the convergents and the integer lengths of the edges are equal to the partial quotients.

We note that we may confine ourselves to one of the two broken lines. The reason is that each angle formed by the segments $[\vec v_{k-2},\vec v_k]$ and $[\vec v_k,\vec v_{k+2}]$ can be equipped (see \cite{german_tlyustangelov_MJCNT_2016}, \cite{korkina_2dim}, \cite{karpenkov_trig}) with its ``integer value'' by defining the latter as the  absolute value of the determinant of the shortest integer vectors parallel to those segments. In our case those are the vectors $\vec v_{k-1}$ and $\vec v_{k+1}$. As we know (see statement (b) of Theorem \ref{t:CF_geometry}), the absolute value of their determinant is equal to the partial quotient $a_{k+1}$. Thus, \emph{each} of the two broken lines contains all the information concerning the continued fraction of $\theta$.

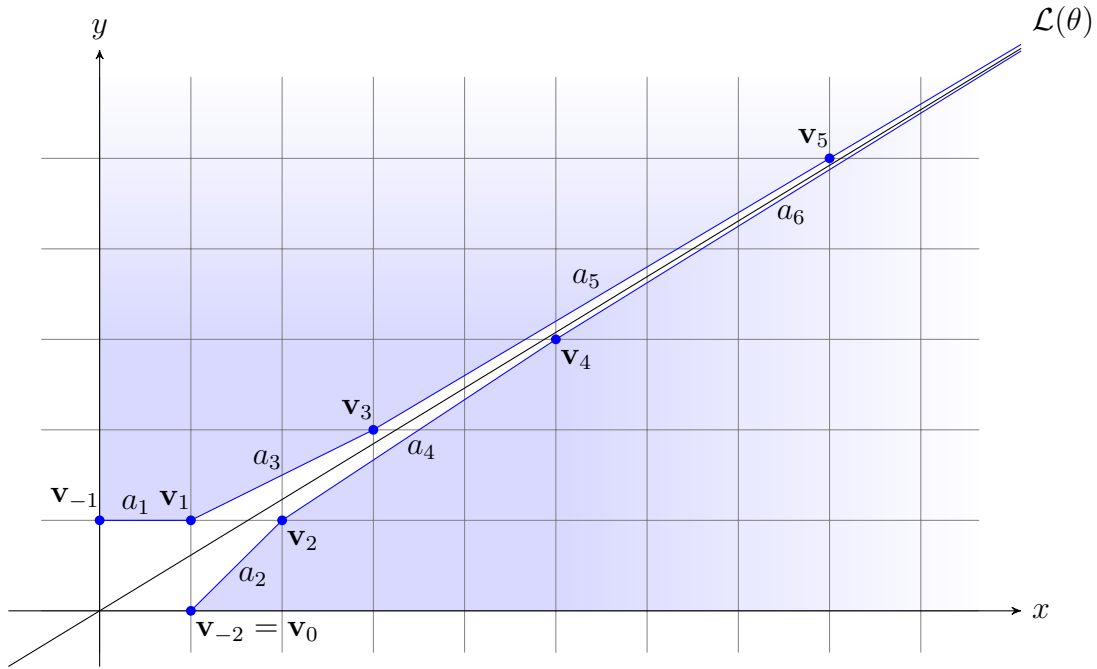
\begin{figure}[h]
\centering
\begin{tikzpicture}[scale=1.2]
%
    \fill[blue!15!]
        (0,3) -- (0,1) -- (1,1) -- (3,2) -- (3+5/3,3) -- cycle;
    \fill[blue!15!]
        (1,0) -- (2,1) -- (5,3) -- (5,0) -- cycle;
    \fill[blue!15!,path fading=north]
        (0,5.9) -- (0,3) -- (3+5/3,3) -- (8+0.9/3*5,5.9) -- cycle;
    \fill[blue!15!,path fading=east]
        (5,0) -- (5,3) -- (5+2.9*8/5,5.9) -- (5+2.9*8/5,0) -- cycle;

    \draw[very thin,color=gray,scale=1] (4-2.9/5*8,-6/13) grid (5+2.9/5*8,5.9);

    \draw[->,>=stealth'] (-1,0) -- (10.1,0) node[right] {$x$};
    \draw[->,>=stealth'] (0,-8/13) -- (0,6.2) node[above] {$y$};

    \draw[color=black] plot[domain=-1:10.1] (\x, {8*\x/13}) node[above right]{$\cL(\theta)$};

    \draw[color=blue] (0,1) -- (1,1);
    \draw[color=blue] (1,1) -- (3,2);
    \draw[color=blue] (3,2) -- (8,5);
    \draw[color=blue] (8,5) -- (8+0.4199*5,5+0.4199*3);

    \draw[color=blue] (1,0) -- (2,1);
    \draw[color=blue] (2,1) -- (5,3);
    \draw[color=blue] (5,3) -- (5+0.6375*8,3+0.6375*5);

    \node[fill=blue,circle,inner sep=1.3pt] at (0,1) {};
    \node[fill=blue,circle,inner sep=1.3pt] at (1,1) {};
    \node[fill=blue,circle,inner sep=1.3pt] at (3,2) {};
    \node[fill=blue,circle,inner sep=1.3pt] at (8,5) {};

    \node[fill=blue,circle,inner sep=1.3pt] at (1,0) {};
    \node[fill=blue,circle,inner sep=1.3pt] at (2,1) {};
    \node[fill=blue,circle,inner sep=1.3pt] at (5,3) {};

    \draw (0.12,1) node[above left]{$\vec v_{-1}$};
    \draw (1.12,1) node[above left]{$\vec v_1$};
    \draw (3.12,2) node[above left]{$\vec v_3$};
    \draw (8.12,5) node[above left]{$\vec v_5$};

    \draw (1-0.07,0) node[below right]{$\vec v_{-2}=\vec v_0$};
    \draw (2-0.07,1) node[below right]{$\vec v_2$};
    \draw (5-0.07,3) node[below right]{$\vec v_4$};

    \draw (0.4,1-0.05) node[above]{$a_1$};
    \draw (2+0.12,1.5-0.05) node[above left]{$a_3$};
    \draw (5.5+0.12,3.5-0.05) node[above left]{$a_5$};

    \draw (1.5-0.1,0.5+0.1) node[below right]{$a_2$};
    \draw (3.5-0.25,2) node[below right]{$a_4$};
    \draw (5+0.3*8-0.1,3+0.3*5+0.1) node[below right]{$a_6$};

\end{tikzpicture}
\caption{Klein polygons} \label{fig:klein_polygons}
\end{figure}

\subsubsection{Best approximations}

The convergents are \emph{best approximations} of $\theta$ (see \cite{cassels_DA}) in the following sense.

\begin{definition}
  A rational number $p/q$ is called \emph{best approximation} of $\theta$, if $p$ is the closest integer to $q\theta$ and for all rational numbers $p'/q'$ such that $q'<q$ we have
  \[ 
    |q\theta-p|<|q'\theta-p'|. 
  \]
\end{definition}

Geometrically, this means that there are no nonzero points of $\Z^2$ in the parallelogram centered at the origin with a vertex at the point $(q,p)$ and sides parallel to $\cL(\theta)$ and to the ordinate axis, apart from its vertices (the remaining two vertices may also belong to $\Z^2$, it happens if and only if $\theta\in\frac12\Z\backslash\Z$). Klein polygons provide a rather simple proof of the fact that if $\theta$ is not half-integer, then the set of its best approximations coincides exactly with the set of its convergents.

\begin{figure}[h]
\centering
\begin{tikzpicture}
    \draw[->,>=stealth'] (-2.3*13/8,0) -- (2*2.3*13/8,0) node[right] {$x$};
    \draw[->,>=stealth'] (0,-2.3) -- (0,2*2.3) node[above] {$y$};

    \draw plot[domain=-2.3*13/8:2*2.3*13/8] (\x, {8*\x/13}) node[above right]{$\cL(\theta)$};

    \draw[white,thick] (0,0) -- (0,2*2.3-5*8/13) -- (2*2.5,2*2.3) -- (2*2.5,5*8/13) -- cycle;
    \draw[thick,dashed] (0,0) -- (0,2*2.3-5*8/13) -- (2*2.5,2*2.3) -- (2*2.5,5*8/13) -- cycle;
    \draw[thick,dashed] (0,2*2.3-5*8/13) -- (2*2.5,5*8/13);

    \draw[thick] (-2.5,-2.3) -- (-2.5,2.3-5*8/13) -- (2.5,2.3) -- (2.5,-2.3+5*8/13) -- cycle;


    \fill[blue,opacity=0.15] (0,0) -- (0,2*2.3-5*8/13) -- (2*2.5,5*8/13);

    \node[fill=blue,circle,inner sep=1.8pt] at (0,0) {};
    \node[fill=blue,circle,inner sep=1.8pt] at (2.5,2.3) {};
    \node[fill=blue,circle,inner sep=1.8pt] at (-2.5,-2.3) {};
    \node[fill=blue,circle,inner sep=1.8pt] at (2*2.5,2*2.3) {};

    \node[fill=white,draw=blue,thick,circle,inner sep=1.5pt] at (0,2*2.3-5*8/13) {};
    \node[fill=white,draw=blue,thick,circle,inner sep=1.5pt] at (2*2.5,5*8/13) {};

    \node[fill=white,draw=blue,thick,circle,inner sep=1.5pt] at (-2.5,2.3-5*8/13) {};
    \node[fill=white,draw=blue,thick,circle,inner sep=1.5pt] at (2.5,-2.3+5*8/13) {};

    \draw[xshift=-0.15cm,yshift=0.1cm] (2.5,2.3) node[above right]{$\vec v$};
    \draw ($2*(2.5,2.3)$) node[above right]{$2\vec v$};
    \draw (-2.5,-2.3) node[below left]{$-\vec v$};
    \draw (0,0) node[below right]{$\vec 0$};

    \draw[->,>=stealth',thin] (4,1) node[right]{$\cQ_{\vec v}$ -- empty parallelogram} -- (2.2,1);
    \draw[->,>=stealth',thin] (4,1.8) node[right]{$\Delta_{\vec v}$ --- empty triangle} -- (2.2,1.8);

\end{tikzpicture}
\caption{Best approximations} \label{fig:best_approximations}
\end{figure}
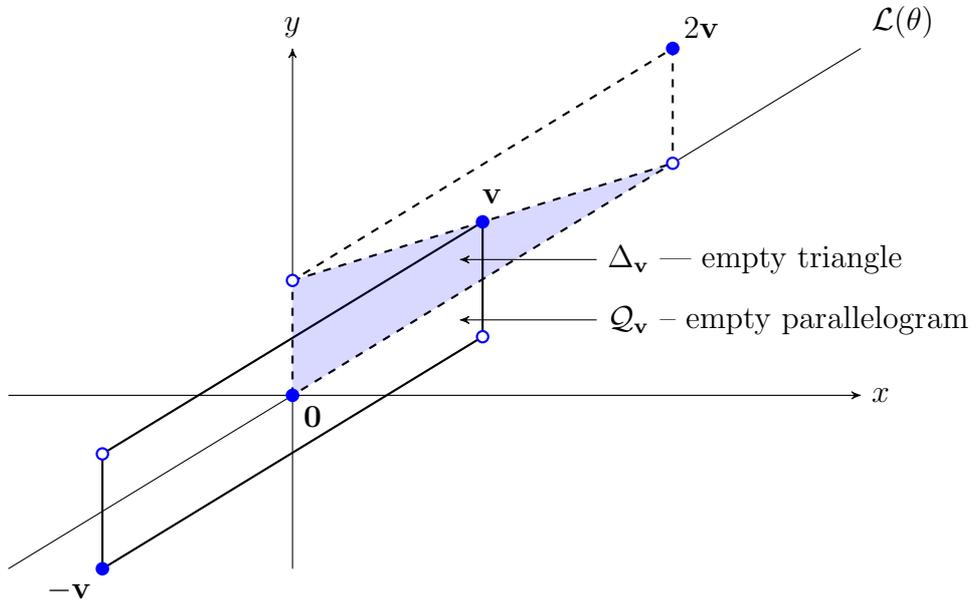

\begin{theorem} \label{t:vertices_are_best_approx}
  Let $\theta\in\R\backslash \big(\frac12\Z\big)$. Then the vertices $\vec v_k$, $k\geq0$, of the Klein polygons correspond to the best approximations of $\theta$.
\end{theorem}

\begin{proof}
  Let $\cQ_{\vec v}$ be the parallelogram centered at the origin with a vertex at a given point $\vec v$ and sides parallel to $\cL(\theta)$ and to the ordinate axis. Let $\Delta_{\vec v}$ be the half of $\cQ_{\vec v}$ translated by vector $\vec v$ (as shown in Fig.\ref{fig:best_approximations}). Then, since $\Z^2$ is closed under addition, the following equivalencies hold:
  \[ 
    \cQ_{\vec v}\cap\Z^2=\{\vec 0,\pm\vec v\}\quad\Longleftrightarrow\quad
    (\cQ_{\vec v}+\vec v)\cap\Z^2=\{\vec 0,\vec v,2\vec v\}\quad\Longleftrightarrow\quad
    \Delta_{\vec v}\cap\Z^2=\{\vec 0,\vec v\}. 
  \]

%

%
%


  For $\theta$ which is not half-integer, the former of these three statements is equivalent to the fact that $\vec v$ corresponds to a best approximation of $\theta$. The latter one -- to the fact that $\vec v$ is a vertex of the respective Klein polygon.

  Theorem is proved.
\end{proof}

\subsection{Triviality of the uniform analogue of the Diophantine exponent}\label{subsec:uniform_triviality}


Inequality \eqref{def:beta_1_dim} admits infinitely many solutions if and only if there are arbitrarily large values of $t$ such that the system
\begin{equation} \label{eq:beta_1_dim_system}
\begin{cases}
  0<q\leq t \\
  |\theta q-p|\leqslant t^{-\gamma}
\end{cases}
\end{equation}
has a solution in integer $p$, $q$. This is exactly what Corollary \ref{cor:dirichlet_irr} to Dirichlet's theorem claims (provided $\theta$ is irrational) for $\gamma=1$. Which gives the respective bound for $\omega(\theta)$. But Dirichlet's theorem itself (Theorem \ref{t:dirichlet}) claims for $\gamma=1$ that \eqref{eq:beta_1_dim_system} has a solution for \emph{every} $t$ large enough. Therefore, it is natural to consider the following \emph{uniform} analogue of $\omega(\theta)$:
\[
  \hat\omega(\theta)=
  \sup\Big\{ \gamma\in\R\ \Big|\ \exists\,T\in\R:\ \forall\,t\geq T \text{\ \ \eqref{eq:beta_1_dim_system} has a solution in }(p,q)\in\Z^2 \Big\}.
\]

Dirichlet's theorem implies the inequality
\begin{equation} \label{eq:alpha_1_dim}
  \hat\omega(\theta)\geq1.
\end{equation}

It appears that in the case of a single number the uniform exponent is almost trivial: for rational $\theta$ we have
\[ 
  \omega(\theta)=\hat\omega(\theta)=\infty 
\]
(see the proof of Proposition \ref{prop:beta_for_rationals}), and for irrational $\theta$, in view of 
Proposition \ref{prop:alpha_1_dim_system}, which we are about to prove, the sign of inequality in \eqref{eq:alpha_1_dim} can be replaced with the sign of equality.

\begin{proposition} \label{prop:alpha_1_dim_system}
  If $\theta$ is irrational, then there are arbitrarily large values of $t$ such that the system 
  \begin{equation} \label{eq:alpha_1_dim_system}
  \begin{cases}
    0<q<t \\
    |\theta q-p|<\dfrac{1}{2t}
  \end{cases}
  \end{equation}
  has no solutions in integer $p$, $q$.
\end{proposition}

\begin{proof}
  Consider an arbitrary vertex $\vec v$ of one of the Klein polygons. By Theorem \ref{t:vertices_are_best_approx} it corresponds to some best approximation of $\theta$, i.e. the parallelogram $\cQ_{\vec v}$ (see Fig.\ref{fig:best_approximations},\ref{fig:big_empty_parallelogram}) contains no integer points distinct from $\vec 0$, $\pm\vec v$.
  
  Let us expand $\cQ_{\vec v}$ along $\cL(\theta)$ until we meet an integer point $\vec w$ distinct from $\vec 0$, $\pm\vec v$. Such a point does exist by Minkowski's convex body theorem (Theorem \ref{t:minkowski_convex}). We get a parallelogram $\cQ'_{\vec v}$ with vertices at points $\vec a$, $\vec b$, $\vec c$, $\vec d$ (see Fig.\ref{fig:big_empty_parallelogram}).

\begin{figure}[h]
\centering
\begin{tikzpicture}
    \draw[->,>=stealth'] (-3-2.3*13/8,0) -- (3+2.3*13/8,0) node[right] {$x$};
    \draw[->,>=stealth'] (0,-2.3-3*8/13) -- (0,2.3+3*8/13) node[above] {$y$};

    \draw plot[domain=-3-2.3*13/8:3+2.3*13/8] (\x, {8*\x/13}) node[above right]{$\cL(\theta)$};

    \draw[thick,dashed] (-5.5,-2.3-3*8/13) -- (-5.5,-2*5.5*8/13+2.3+3*8/13) -- (5.5,2.3+3*8/13) -- (5.5,2*5.5*8/13-2.3-3*8/13) -- cycle;
    \draw[thick] (-2.5,-2.3) -- (-2.5,2.3-5*8/13) -- (2.5,2.3) -- (2.5,-2.3+5*8/13) -- cycle;

    \draw[blue] (0,0) -- (2.5,2.3) -- (5.5,2.9) -- cycle;
    \fill[blue,opacity=0.15] (0,0) -- (2.5,2.3) -- (5.5,2.9);
    \fill[blue,opacity=0.1] (-5.5,-2.3-3*8/13) -- (-5.5,-2*5.5*8/13+2.3+3*8/13) -- (5.5,2.3+3*8/13) -- (5.5,2*5.5*8/13-2.3-3*8/13) -- cycle;

    \node[fill=blue,circle,inner sep=1.8pt] at (0,0) {};
    \node[fill=blue,circle,inner sep=1.8pt] at (2.5,2.3) {};
    \node[fill=blue,circle,inner sep=1.8pt] at (-2.5,-2.3) {};
    \node[fill=blue,circle,inner sep=1.8pt] at (5.5,2.9) {};
    \node[fill=blue,circle,inner sep=1.8pt] at (-5.5,-2.9) {};

    \draw (0,0) node[below right]{$\vec 0$};
    \draw[xshift=0.05cm,yshift=-0.03cm] (2.5,2.3) node[above left]{$(\vec v_{k-1}=)\ \vec v$};
    \draw[xshift=-0.1cm] (-2.5,-2.3) node[below right]{$-\vec v$};
    \draw[xshift=0.05cm] (5.5,2.9) node[right]{$\vec w\ (=\vec v_k)$};
    \draw[xshift=-0.03cm] (-5.5,-2.9) node[left]{$-\vec w$};

    \node[fill=white,draw=blue,thick,circle,inner sep=1.5pt] at (-5.5,-2.3-3*8/13) {};
    \node[fill=white,draw=blue,thick,circle,inner sep=1.5pt] at (-5.5,-2*5.5*8/13+2.3+3*8/13) {};
    \node[fill=white,draw=blue,thick,circle,inner sep=1.5pt] at (5.5,2.3+3*8/13) {};
    \node[fill=white,draw=blue,thick,circle,inner sep=1.5pt] at (5.5,2*5.5*8/13-2.3-3*8/13) {};

    \draw (5.5,2*5.5*8/13-2.3-3*8/13) node[below right]{$\vec a$};
    \draw (5.5,2.3+3*8/13) node[above right]{$\vec b$};
    \draw (-5.5,-2*5.5*8/13+2.3+3*8/13) node[above left]{$\vec c$};
    \draw (-5.5,-2.3-3*8/13) node[below left]{$\vec d$};

    \draw[->,>=stealth',thin] (1,-2) node[below right,xshift=-0.35cm]{$\cQ_{\vec v}$} -- (1,0.25);
    \draw[->,>=stealth',thin] (3,-2) node[below right,xshift=-0.35cm]{$\cQ'_{\vec v}$ -- expanded $\cQ_{\vec v}$} -- (3,1.38);

\end{tikzpicture}
\caption{Big empty parallelogram} \label{fig:big_empty_parallelogram}
\end{figure}
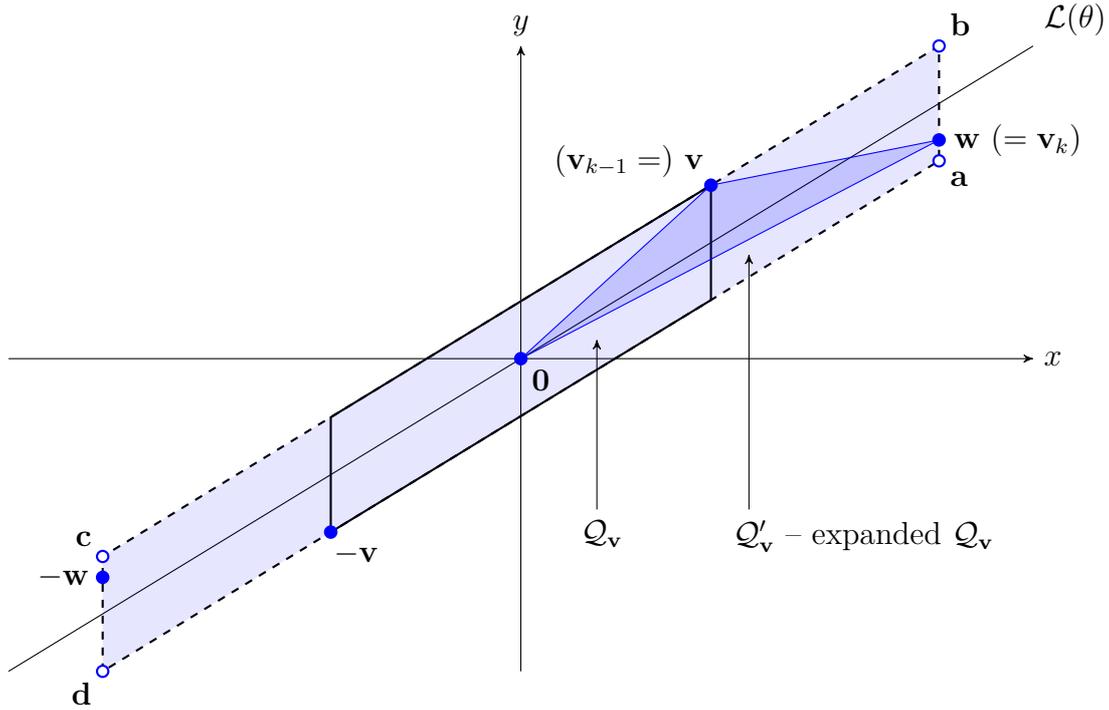

  Actually, $\vec v$ and $\vec w$ correspond to two consecutive convergents of $\theta$, but our argument does not require this fact. It suffices to observe that $\vec v$ and $\vec w$ are not collinear. Their non-collinearity implies that the area of the triangle with vertices at $\vec 0$, $\vec v$, $\vec w$ is not less than $1/2$. Hence the area of $\cQ'_{\vec v}$ is greater than $2$.
  
  Thus, for $t$ equal to the abscissa of $\vec w$ system \eqref{eq:alpha_1_dim_system} has no solution in integer $p$, $q$.
  
  We can apply the expanding procedure to $\cQ_{\vec w}$ and obtain the next ``empty'' parallelogram. The procedure can be repeated infinitely many times, as by the irrationality of $\theta$ the line $\cL(\theta)$ contains no nonzero integer points. Hence the abscissa of $\vec w$ can attain arbitrarily large values.

  Proposition is proved.
\end{proof}

\begin{corollary} \label{cor:alpha_1_dim_system}
  If $\theta$ is irrational, then $\hat\omega(\theta)=1$.
\end{corollary}

Thus, considering $\hat\omega(\theta)$ for a real $\theta$ is quite senseless due to the degeneracy of this quantity. This is, however, not the case in the multidimensional setting, to which we now proceed.

\section{Simultaneous approximation and approximation of zero by the values of a linear form}

\subsection{Two more Dirichlet's theorems}

In the previous Section our main question was how small the quantity
\[ 
  \theta x-y 
\]
can be for integer $x$, $y$. The multidimensional case, i.e. the case of $n$ numbers $\theta_1,\ldots,\theta_n$ instead of one number $\theta$, admits two ways to generalise this problem. We can study how small the quantities
\[ 
  \theta_1x-y_1,\ldots,\theta_nx-y_n
\]
can be, or we can study how small the values of the linear form
\[ 
  \theta_1x_1+\ldots+\theta_nx_n-y
\]
can be.

The first multidimensional result in Diophantine approximation belongs to Dirichlet and refers to both of these questions. In the aforementioned paper \cite{dirichlet} of the year 1842 he proved the following statement.


\begin{theorem}[G. Lejeune Dirichlet, 1842] \label{t:dirichlet_linear_form}
  Let $\Theta=(\theta_1,\ldots,\theta_n)\in\R^n$. Then, for each $t\geq1$, there are integers $x_1,\ldots,x_n,y$ satisfying the inequalities
  \begin{equation} \label{eq:dirichlet_linear_form}
  \begin{cases}
    0<\max_{1\leq i\leq n}|x_i|\leq t \\
    |\theta_1x_1+\ldots+\theta_nx_n-y|\leq t^{-n}
  \end{cases}.
  \end{equation}
\end{theorem}

It is worth making a remark similar to the one we made right after Theorem \ref{t:dirichlet}: the original Dirichlet's statement is weaker than the statement of Theorem \ref{t:dirichlet_linear_form}, but his argument proves Theorem \ref{t:dirichlet_linear_form} as well. In that same paper, Dirichlet generalises his result to the case of several linear forms. For the forms $\theta_1x-y_1,\ldots,\theta_nx-y_n$, this generalisation reads as follows.

\begin{theorem}[G. Lejeune Dirichlet, 1842] \label{t:dirichlet_simultaneous}
  Let $\Theta=(\theta_1,\ldots,\theta_n)\in\R^n$. Then, for each $t\geq1$, there are integers $x,y_1,\ldots,y_n$ satisfying the inequalities
  \begin{equation} \label{eq:dirichlet_simultaneous}
  \begin{cases}
    0<|x|\leq t \\
    \max_{1\leq i\leq n}|\theta_ix-y_i|\leq t^{-1/n}
  \end{cases}.
  \end{equation}
\end{theorem}

Thus, Theorem \ref{t:dirichlet_linear_form} provides an estimate for the order of approximation of zero by the values of the linear form
\[ 
  L_\Theta(x_1,\ldots,x_n,y)=\theta_1x_1+\ldots+\theta_nx_n-y
\]
at integer $x_1,\ldots,x_n,y$, and Theorem \ref{t:dirichlet_simultaneous} provides an estimate for the order of simultaneous approximation of the numbers $\theta_1,\ldots,\theta_n$ by rationals with the same denominator. This is how we get the first estimates for the Diophantine exponents of the $n$-tuple $\Theta$ and of the linear form $L_\Theta$, whose definitions are given in Section \ref{subsec:regular_and_uniform_exponents}.

It is worth noting that, same as Theorem \ref{t:dirichlet}, Theorems \ref{t:dirichlet_linear_form} and \ref{t:dirichlet_simultaneous} follow immediately from Minkowski's convex body theorem in the form of Corollary \ref{cor:minkowski_linear_forms} applied to the systems \eqref{eq:dirichlet_linear_form} and \eqref{eq:dirichlet_simultaneous}.

%

\subsection{Regular and uniform exponents}\label{subsec:regular_and_uniform_exponents}

\begin{definition} \label{def:beta_simultaneous}
  The supremum of real $\gamma$ satisfying the condition that there exist arbitrarily large $t$ such that the system
  \begin{equation} \label{eq:belpha_simultaneous}
  \begin{cases}
    0<|x|\leq t \\
    \max_{1\leq i\leq n}|\theta_ix-y_i|\leq t^{-\gamma}
  \end{cases}
  \end{equation}
  admits a solution in integer $x,y_1,\ldots,y_n$ is called the \emph{regular Diophantine exponent} of $\Theta$ and is denoted by $\omega(\Theta)$.
\end{definition}

\begin{definition} \label{def:beta_linear_form}
  The supremum of real $\gamma$ satisfying the condition that there exist arbitrarily large $t$ such that the system
  \begin{equation} \label{eq:belpha_linear_form}
  \begin{cases}
    0<\max_{1\leq i\leq n}|x_i|\leq t \\
    |L_\Theta(x_1,\ldots,x_n,y)|\leq t^{-\gamma}
  \end{cases}
  \end{equation}
  admits a solution in integer $x_1,\ldots,x_n,y$ is called the \emph{regular Diophantine exponent} of $L_\Theta$ and is denoted by $\omega(L_\Theta)$.
\end{definition}

Substituting the words ``there exist arbitrarily large $t$ such that'' with ``for every $t$ large enough'', we get \emph{uniform} analogues of regular exponents.

\begin{definition} \label{def:alpha_simultaneous}
  The supremum of real $\gamma$ satisfying the condition that for every $t$ large enough the system \eqref{eq:belpha_simultaneous}
  admits a solution in integer $x,y_1,\ldots,y_n$ is called the \emph{uniform Diophantine exponent} of $\Theta$ and is denoted by $\hat\omega(\Theta)$.
\end{definition}

\begin{definition} \label{def:alpha_linear_form}
  The supremum of real $\gamma$ satisfying the condition that for every $t$ large enough the system \eqref{eq:belpha_linear_form}
  admits a solution in integer $x_1,\ldots,x_n,y$ is called the \emph{uniform Diophantine exponent} of $L_\Theta$ and is denoted by $\hat\omega(L_\Theta)$.
\end{definition}

As we said before, Theorems \ref{t:dirichlet_linear_form} and \ref{t:dirichlet_simultaneous} provide ``trivial'' bounds for these four exponents:

\begin{equation} \label{eq:belpha_trivial_inequalities}
  \omega(\Theta)\geq\hat\omega(\Theta)\geq1/n,\qquad
  \omega(L_\Theta)\geq\hat\omega(L_\Theta)\geq n.
\end{equation}

These inequalities are sharp in the sense that there exist $n$-tuples $\Theta$ for which each of them turns into an equality.
For instance, Perron \cite{perron} proved the following statement.


\begin{theorem}[O.\,Perron, 1921] \label{t:perron}
  Let $\theta_1,\ldots,\theta_n$ be elements of a real algebraic number field of degree $n+1$, linearly independent with the unity over $\Q$. Then there is a constant $c>0$ depending only on $\theta_1,\ldots,\theta_n$ such that the inequality
  \[ 
    \max_{1\leq i\leq n}|\theta_ix-y_i|<cx^{-1/n} 
  \]
  admits only finitely many solutions in integer $x,y_1,\ldots,y_n$.
\end{theorem}

\begin{corollary} \label{cor:perron}
  If $\Theta$ is as in Theorem \ref{t:perron}, then $\omega(\Theta)=\hat\omega(\Theta)=1/n$.
\end{corollary}


Particularly, for $\Theta=(\theta,\theta^2,\ldots,\theta^n)$, where $\theta$ is a real algebraic number of degree $n+1$, we have $\omega(\Theta)=1/n$.

In his proof of Theorem \ref{t:perron} Perron uses an argument which also proves the implication
\begin{equation} \label{eq:n_implies_1/n}
  \omega(L_\Theta)=n\ \ \Longrightarrow\ \ \omega(\Theta)=1/n\,.
\end{equation}
The inverse implication
\begin{equation} \label{eq:1/n_implies_n}
  \omega(\Theta)=1/n\ \ \Longrightarrow\ \ \omega(L_\Theta)=n
\end{equation}
follows from the proof of the main result of Khintchine's paper \cite{khintchine_perron} (see also the collection of selected papers by Khintchine \cite{khintchine_selected_works}). This gives us an analogue of Corollary \ref{cor:perron}.

\begin{corollary} \label{cor:perrontchine}
  If $\Theta$ ia as in Theorem \ref{t:perron}, then $\omega(L_\Theta)=\hat\omega(L_\Theta)=n$.
\end{corollary}



In contrast to the case $n=1$, uniform exponents $\hat\omega(\Theta)$, $\hat\omega(L_\Theta)$ are no longer trivial for $n\geq2$. They can attain finite values strictly greater than $1/n$ and $n$ respectively. The existence of linear forms $L_\Theta$ with $\hat\omega(L_\Theta)=+\infty$ that do not vanish at nonzero integer points, as well as the existence of $\Theta$ with $\hat\omega(\Theta)=1$, was proved by Khintchine in 1926 in one of his most famous papers \cite{khintchine_palermo}. Such $n$-tuples and linear forms belong to a somewhat wider class of \emph{Khintchine's singular systems} (see Khintchine's papers \cite{khintchine_singular_1937}, \cite{khintchine_singular_1948}, and also the remarkable survey by Moshchevitin \cite{moshchevitin_UMN_2010}).

Finite values strictly greater than $1$ cannot be attained by $\hat\omega(\Theta)$. The reason is the very same effect that implies triviality of the uniform exponent in the case $n=1$. Indeed, let $\Theta=(\theta_1,\ldots,\theta_n)\in\R^n\backslash\Q^n$ and let $\e>0$ be given. Then, for irrational $\theta_i$, Proposition \ref{prop:alpha_1_dim_system} guarantees the existence of arbitrarily large $t$ such that the system
\[ 
  \begin{cases}
    0<x<t \\
    |\theta_ix-y_i|<\dfrac{1}{2t}
  \end{cases} 
\]
has no solution in integer $x$, $y_i$. If $t$ is large enough, the system
\[ 
  \begin{cases}
    0<x\leq t \\
    \max_{1\leq i\leq n}|\theta_ix-y_i|\leq t^{-1-\e}
  \end{cases} 
\]
neither has any solution in integer $x$, $y_1,\ldots,y_n$.

Thus, for every $\Theta=(\theta_1,\ldots,\theta_n)\in\R^n\backslash\Q^n$ we have

\begin{equation} \label{eq:alpha_leq_1}
  \hat\omega(\Theta)\leq1.
\end{equation}

\subsection{Transference principle} \label{subsec:transference_only_n}

Perron and Khintchine proved actually something stronger than \eqref{eq:n_implies_1/n} and \eqref{eq:1/n_implies_n}. Their constructions have ``local'' nature, therefore, they prove both the equivalence
\begin{equation} \label{eq:1/n_equiv_n}
  \omega(\Theta)=1/n\ \ \Longleftrightarrow\ \ \omega(L_\Theta)=n
\end{equation}
and the fact that $\Theta$ and $L_\Theta$ are simultaneously \emph{badly approximable}.

\begin{definition}
  An $n$-tuple $\Theta=(\theta_1,\ldots,\theta_n)$ is called \emph{badly approximable} if there is a constant $c>0$ depending only on $\Theta$ such that for every tuple of integers $x$, $y_1,\ldots,y_n$ with $x\neq0$ we have
  \[ 
    \max_{1\leq i\leq n}|\theta_ix-y_i|\geq c|x|^{-1/n}.
  \]
\end{definition}

\begin{definition}
   A linear form $L_\Theta(x_1,\ldots,x_n,y)$ is called \emph{badly approximable} if there is a constant $c>0$ depending only on $\Theta$ such that for every tuple of integers $x_1,\ldots,x_n$, $y$ with $x_1,\ldots,x_n$ not all zero we have
  \[ 
    |L_\Theta(x_1,\ldots,x_n,y)|\geq c\max_{1\leq i\leq n}|x_i|^{-n}. 
  \]
\end{definition}

Apparently, Perron's and Khintchine's constructions were the first examples of exploiting the ``duality'' of the problem of simultaneous approximation and the problem of approximating zero by the values of a linear form. We discuss the essence of this phenomenon in detail in Section \ref{subsec:ideas_and_methods}.

\subsubsection{Inequalities for regular exponents}\label{subsubsec:transference_only_n_regular}

The first truly outstanding result relating the problem of simultaneous approximation and the problem of approximating zero by the values of a linear form is Khintchine's theorem, which was also proved in the aforementioned paper \cite{khintchine_palermo}. In that same paper, he named the phenomenon he had discovered the \emph{transference principle}.

\begin{theorem}[A.\,Ya.\,Khintchine, 1926] \label{t:khintchine_transference}
  \begin{equation} \label{eq:khintchine_transference}
    \frac{1+\omega(L_\Theta)}{1+\omega(\Theta)}\geq n,\qquad
    \frac{1+\omega(L_\Theta)^{-1}}{1+\omega(\Theta)^{-1}}\geq \frac1n\,.
  \end{equation}
\end{theorem}

As Jarn\'{\i}k showed in his papers \cite{jarnik_1936_1}, \cite{jarnik_1936_2}, inequalities \eqref{eq:khintchine_transference} are sharp in the sense that for every $\gamma\in[n,+\infty]$ there are two $n$-tuples $\Theta'$ and $\Theta''$ such that
\[
  \omega(L_{\Theta'})=\omega(L_{\Theta''})=\gamma,\qquad
  \frac{1+\gamma}{1+\omega(\Theta')}=n,\qquad
  \frac{1+\gamma^{-1}}{1+\omega(\Theta'')^{-1}}=\frac1n\,.
\]

\subsubsection{Inequalities for uniform exponents}

The method used to prove inequalities \eqref{eq:khintchine_transference} enables proving the same inequalities for the uniform exponents $\hat\omega(\Theta)$ and $\hat\omega(L_\Theta)$:
\begin{equation} \label{eq:khintchine_transference_alpha}
  \frac{1+\hat\omega(L_\Theta)}{1+\hat\omega(\Theta)}\geq n,\qquad
  \frac{1+\hat\omega(L_\Theta)^{-1}}{1+\hat\omega(\Theta)^{-1}}\geq \frac1n\,.
\end{equation}
However, these exponents satisfy stronger inequalities. For $n=2$ Jarn\'{\i}k presented in his paper \cite{jarnik_tiflis} a surprising fact: in this case the uniform exponents are connected by an equality.

\begin{theorem}[V.\,Jarn\'{\i}k, 1938] \label{t:jarnik_identity}
  If $n=2$ and the components of $\Theta$ are linearly independent with the unity over $\Q$, then
  \begin{equation} \label{eq:jarnik_identity}
    \hat\omega(L_\Theta)^{-1}+\hat\omega(\Theta)=1.
  \end{equation}
\end{theorem}

Later on, in 1948, Khintchine published in \cite{khintchine_dobav} a rather simple proof of Theorem \ref{t:jarnik_identity}.

For $n\geq3$, in that same paper \cite{jarnik_tiflis}, Jarn\'{\i}k proved that, if $\theta_1,\ldots,\theta_n$ are linearly independent with the unity over $\Q$, then, along with \eqref{eq:khintchine_transference_alpha}, the following inequalities hold:
\begin{equation} \label{eq:jarnik_inequalities_cases}
  \begin{aligned}
    \displaystyle &
    \hat\omega(\Theta)\geq\frac{1}{n-1}\left(1-\frac{1}{\hat\omega(L_\Theta)-2n+4}\right),\ \ \text{ if }\ \hat\omega(L_\Theta)>n(2n-3), \\
    \displaystyle &
    \hat\omega(\Theta)\leq1-\frac{1^{\vphantom{|}}}{\hat\omega(L_\Theta)-n+2}\,,\qquad\qquad\quad\,\ \text{ if }\ \hat\omega(\Theta)>(n-1)/n.
  \end{aligned}
\end{equation}
In 2012, inequalities \eqref{eq:khintchine_transference_alpha}, \eqref{eq:jarnik_inequalities_cases} were improved by the author in \cite{german_MJCNT_2012}, \cite{german_AA_2012}.

\begin{theorem}[O.\,G., 2012] \label{t:my_simultaneous_vs_linear_form}
  For every $\Theta=(\theta_1,\ldots,\theta_n)\in\R^n\backslash\Q^n$ we have
  \begin{equation}\label{eq:my_simultaneous_vs_linear_form}
    \hat\omega(L_\Theta)\geq\frac{n-1}{1-\hat\omega(\Theta)}\,,\qquad
    \hat\omega(\Theta)\geq \frac{1-\hat\omega(L_\Theta)^{-1}}{n-1}\,.
  \end{equation}
\end{theorem}

It was shown by Marnat \cite{marnat_sharpness} and (independently) by Schmidt and Summerer \cite{schmidt_summerer_AA_2016} that both inequalities \eqref{eq:my_simultaneous_vs_linear_form} are sharp. More specifically, they showed that for every $\gamma\in[n,+\infty]$ and every
\begin{equation}\label{eq:delta_for_marnat}
  \delta\in\bigg[\frac{1-\gamma^{-1}}{n-1}\,,\,1-\frac{n-1}{\gamma}\bigg]
\end{equation}
there are continuously many tuples $\Theta$ with components linearly independent with the unity over $\Q$ such that $\hat\omega(L_{\Theta})=\gamma$, $\hat\omega(\Theta)=\delta$ (we note that, if $\gamma\geq n$, then the segment in \eqref{eq:delta_for_marnat} is correctly defined and non-empty).

\subsubsection{``Mixed'' inequalities}\label{subsubsec:transference_only_n_mixed}

Despite the fact that inequalities \eqref{eq:khintchine_transference} for the regular exponents are sharp, they can be improved if the uniform exponents are taken into account. For the first time, it was made by Laurent and Bugeaud in papers \cite{laurent_up_down}, \cite{bugeaud_laurent_up_down}. They proved the following.

\begin{theorem}[M.\,Laurent, Y.\,Bugeaud, 2009] \label{t:bugeaud_laurent}
  If the components of $\Theta$ are linearly independent with the unity over $\Q$, then
  \begin{equation}\label{eq:bugeaud_laurent}
    \frac{1+\omega(L_\Theta)}{1+\omega(\Theta)}\geq\frac{n-1}{1-\hat\omega(\Theta)}\,,\qquad
    \frac{1+\omega(L_\Theta)^{-1}}{1+\omega(\Theta)^{-1}}\geq \frac{1-\hat\omega(L_\Theta)^{-1}}{n-1}\,.
  \end{equation}
\end{theorem}

Laurent \cite{laurent_2dim} proved that for $n=2$ these inequalities are sharp.
%
%
%
%
However, for $n\geq3$, it was shown by Schleischitz \cite{schleischitz_laureaud_nonsharpness_2021} that inequalities \eqref{eq:bugeaud_laurent} are no longer sharp. This is not surprising in the light of the following inequalities obtained in 2013 by Schmidt and Summerer \cite{schmidt_summerer_2013} (see also papers \cite{german_moshchevitin_2013}, \cite{german_moshchevitin_2022}, where shorter proofs of their theorem are proposed).

\begin{theorem}[W.\,M.\,Schmidt, L.\,Summerer, 2013] \label{t:schmidt_summerer_2013}
  If the components of $\Theta$ are linearly independent with the unity over $\Q$, then
  \begin{equation}\label{eq:schmidt_summerer_2013}
    \hat\omega(L_\Theta)\leq\frac{1+\omega(L_\Theta)}{1+\omega(\Theta)}\,,\qquad
    \hat\omega(\Theta)\leq\frac{1+\omega(L_\Theta)^{-1}}{1+\omega(\Theta)^{-1}}\,.
  \end{equation}
\end{theorem}

It is easy to see that inequalities \eqref{eq:schmidt_summerer_2013} and \eqref{eq:my_simultaneous_vs_linear_form} immediately imply inequalities \eqref{eq:bugeaud_laurent}.

Thus, since Khintchine's inequalities follow from those of Laurent and Bugeaud, all the currently known transference inequalities that connect $\omega(\Theta)$, $\hat\omega(\Theta)$ with $\omega(L_\Theta)$, $\hat\omega(L_\Theta)$ are implied by inequalities \eqref{eq:my_simultaneous_vs_linear_form} and \eqref{eq:schmidt_summerer_2013} (and, of course, by ``trivial'' inequalities \eqref{eq:belpha_trivial_inequalities} and \eqref{eq:alpha_leq_1}).

\subsubsection{Inequalities between regular and uniform exponents}

There is another series of inequalities which, technically, cannot be classified as transference inequalities, as they connect regular and uniform exponents within the frames of one of the two problems under consideration -- the problem of simultaneous approximation and the problem of approximating zero by the values of a linear form. Nevertheless, it seems quite reasonable to place them alongside with transference inequalities, for transference inequalities themselves provide rather nontrivial relations of this kind (see Theorem \ref{t:german_moshchevitin_2022} below).

In 1950-s Jarn\'{\i}k discovered that if $\hat\omega(\Theta)$ is large, then $\omega(\Theta)$ cannot be too small. In papers \cite{jarnik_szeged_1950, jarnik_czech_1954}, he published the following estimates for $n=2$.

\begin{theorem}[V.\,Jarn\'{\i}k, 1954] \label{t:jarnik_inequalities}
  If $n=2$ and the components of $\Theta$ are linearly independent with the unity over $\Q$, then
  \begin{equation}\label{eq:jarnik_inequalities}
    \frac{\omega(L_\Theta)}{\hat\omega(L_\Theta)}\geq\hat\omega(L_\Theta)-1,\qquad
    \frac{\omega(\Theta)}{\hat\omega(\Theta)}\geq\frac{\hat\omega(\Theta)}{1-\hat\omega(\Theta)}\,.
  \end{equation}
\end{theorem}

These inequalities are sharp. Their sharpness was proved by Laurent \cite{laurent_2dim}. Jarn\'{\i}k \cite{jarnik_czech_1954} also obtained inequalities in the higher dimensions. He showed that the second inequality in \eqref{eq:jarnik_inequalities} holds for every $n\geq2$ and that
\[
  \frac{\omega(L_\Theta)}{\hat\omega(L_\Theta)}\geq\hat\omega(L_\Theta)^{1/(n-1)}-3,
\]
provided $\omega(\Theta)>(5n^2)^{n-1}$.

In 2012, Moshchevitin \cite{moshchevitin_czech_2012} obtained an optimal result for $n=3$ in the problem of simultaneous approximation.

\begin{theorem}[N.\,G.\,Moshchevitin, 2012] \label{t:moshchevitin_czech_2012}
  If $n=3$ and the components of $\Theta$ are linearly independent with the unity over $\Q$, then
  \begin{equation}\label{eq:moshchevitin_czech_2012}
    \frac{\omega(\Theta)}{\hat\omega(\Theta)}\geq\Gsim(\Theta),
  \end{equation}
  where $\Gsim(\Theta)$ is the greatest root of the polynomial
  \begin{equation}\label{eq:moshchevitin_polynomial}
    (1-\hat\omega(\Theta))x^2-\hat\omega(\Theta)x-\hat\omega(\Theta).
  \end{equation}
\end{theorem}

A year later, Schmidt and Summerer \cite{schmidt_summerer_MJCNT_2013} proved an analogue of Theorem \ref{t:moshchevitin_czech_2012} for the linear form problem.

\begin{theorem}[W.\,M.\,Schmidt, L.\,Summerer, 2013] \label{t:schmidt_summerer_MJCNT_2013}
  If $n=3$ and the components of $\Theta$ are linearly independent with the unity over $\Q$, then
  \begin{equation}\label{eq:schmidt_summerer_MJCNT_2013}
    \frac{\omega(L_\Theta)}{\hat\omega(L_\Theta)}\geq\Glin(L_\Theta),
  \end{equation}
  where $\Glin(\Theta)$ is the greatest root of the polynomial
  \begin{equation}\label{eq:schmidt_summerer_polynomial}
    x^2+x+(1-\hat\omega(L_\Theta)).
  \end{equation}
\end{theorem}

In 2018, Marnat and Moshchevitin generalised \eqref{eq:jarnik_inequalities}, \eqref{eq:moshchevitin_czech_2012}, and \eqref{eq:schmidt_summerer_MJCNT_2013} to the case of arbitrary $n\geq2$. Their result was published in 2020 in \cite{marnat_moshchevitin_2020}.

\begin{theorem}[A.\,Marnat, N.\,G.\,Moshchevitin, 2020] \label{t:marnat_moshchevitin_2020}
  Let $n\geq2$ and let the components of $\Theta$ be linearly independent with the unity over $\Q$. Then
  \begin{equation}\label{eq:marnat_moshchevitin_2020}
    \frac{\omega(\Theta)}{\hat\omega(\Theta)}\geq\Gsim(\Theta),
    \qquad
    \frac{\omega(L_\Theta)}{\hat\omega(L_\Theta)}\geq\Glin(\Theta),
  \end{equation}
  where $\Gsim(\Theta)$ and $\Glin(\Theta)$ are the greatest roots of the polynomials
  \begin{equation}\label{eq:marnat_moshchevitin_polynomials}
    (1-\hat\omega(\Theta))x^n-x^{n-1}+\hat\omega(\Theta),
    \qquad
    \hat\omega(L_\Theta)^{-1}x^n-x+(1-\hat\omega(L_\Theta)^{-1})
  \end{equation}
  respectively.
\end{theorem}

In that same paper \cite{marnat_moshchevitin_2020}, Marnat and Moshchevitin showed that their inequalities \eqref{eq:marnat_moshchevitin_2020} are sharp. A year later, a somewhat different proof of Theorem \ref{t:marnat_moshchevitin_2020} was proposed by Rivard-Cooke in his PhD thesis \cite{rivard_cooke} (see also papers \cite{nguyen_poels_roy_2020}, \cite{schleischitz_MJCNT_2022}).

  Note that using the notation $\Gsim$ and $\Glin$ both in Theorems \ref{t:moshchevitin_czech_2012}, \ref{t:schmidt_summerer_MJCNT_2013} and in Theorem \ref{t:marnat_moshchevitin_2020} is correct, since for $n=3$ the first (resp. second) polynomial in \eqref{eq:marnat_moshchevitin_polynomials} equals the first (resp. second) polynomial in \eqref{eq:moshchevitin_polynomial} multiplied by $x-1$ (resp. by $\hat\omega(L_\Theta)^{-1}(x-1)$).

It is easy to verify with the help of \eqref{eq:jarnik_identity} that for $n=2$ the right-hand sides of inequalities \eqref{eq:jarnik_inequalities} coincide and equal
\[
  \frac{1-\hat\omega(L_\Theta)^{-1}}{1-\hat\omega(\Theta)}\,.
\]
It is interesting that Schmidt--Summerer's inequalities \eqref{eq:schmidt_summerer_2013} imply that this very expression bounds from below the ratios $\omega(\Theta)/\hat\omega(\Theta)$ and $\omega(L_\Theta)/\hat\omega(L_\Theta)$ for every $n\geq2$.
This can be observed from the following result obtained in \cite{german_moshchevitin_2022}.

\begin{theorem}[O.\,G., N.\,G.\,Moshchevitin, 2022] \label{t:german_moshchevitin_2022}
  Let $n\geq2$ and let the components of $\Theta$ be linearly independent with the unity over $\Q$. Define $\Gsim(\Theta)$ and $\Glin(\Theta)$ the same way as in Theorem \ref{t:marnat_moshchevitin_2020}. Then
  \begin{equation}\label{eq:german_moshchevitin_2022}
  \begin{split}
    \min\bigg(\frac{\omega(\Theta)}{\hat\omega(\Theta)}\,,\,\frac{\omega(L_\Theta)}{\hat\omega(L_\Theta)}\bigg)
    & \geq
    \frac{1+\omega(\Theta)}{1+\omega(L_\Theta)^{-1}}\geq \\
    & \geq
    \frac{1-\hat\omega(L_\Theta)^{-1}}{1-\hat\omega(\Theta)}\geq
    \min\big(\Gsim(\Theta),\Glin(\Theta)\big).
  \end{split}
  \end{equation}
\end{theorem}

Another interesting corollary to Theorem \ref{t:german_moshchevitin_2022} is the fact that Schmidt--Summerer's inequalities \eqref{eq:schmidt_summerer_2013} imply at least one of Marnat--Moshchevitin's inequalities \eqref{eq:marnat_moshchevitin_2020} -- the one corresponding to the smallest number among $\Gsim(\Theta)$ and $\Glin(\Theta)$.


\subsection{Ideas and methods} \label{subsec:ideas_and_methods}

\subsubsection{Mahler's method} \label{subsubsec:mahler_method}

Ten years after Khintchine published Theorem \ref{t:khintchine_transference} describing transference principle, Mahler \cite{mahler_matsbornik_khintchine} found a very simple proof for it, which vividly demonstrated that the problem of simultaneous approximation is ``dual'' to the problem of approximating zero by the values of a linear form.

Mahler himself called his method arithmetic (see \cite{mahler_matsbornik_khintchine}), and though he exploits Minkowski's convex body theorem, he applies it in the form of Corollary \ref{cor:minkowski_linear_forms}, which allows not to go deep into geometry.

Let us explain
his argument from paper \cite{mahler_matsbornik_khintchine}, where he proves the right-hand inequality \eqref{eq:khintchine_transference}, i.e. the inequality
\begin{equation} \label{eq:khintchine_transference_mahler}
  \omega(\Theta)\geq\frac{\omega(L_\Theta)}{(n-1)\omega(L_\Theta)+n}\,.
\end{equation}

To this end let us ``embed'' the problem of simultaneous approximation and the problem of approximating zero by the values of a linear form into the same $(n+1)$-dimensional Euclidean space. Let $u_1,\ldots,u_{n+1}$ be the Cartesian coordinates in $\R^{n+1}$. Let us identify the aforementioned variables $x$, $y_1,\ldots,y_n$ respectively with $u_1,\ldots,u_{n+1}$, and variables $x_1,\ldots,x_n$, $y$ -- respectively with $u_2,\ldots,u_{n+1}$, $-u_1$.
By analogy with Section \ref{subsec:geometry_of_dirichlet} we denote by $\cL=\cL(\Theta)$ the one-dimensional subspace with the generating vector $(1,\theta_1,\ldots,\theta_n)$, and by $\cL^\perp$ -- the orthogonal complement to $\cL$.
Then $\cL$ and $\cL^\perp$ coincide with the spaces of solutions to
\[
  \max_{1\leq i\leq n}|\theta_ix-y_i|=0
  \qquad\text{ and }\qquad
  L_\Theta(x_1,\ldots,x_n,y)=0.
\]

Assuming that there exist $t$ large enough, positive $\gamma$, and a nonzero point $\vec v=(v_1,\ldots,v_{n+1})\in\Z^{n+1}$ such that
\begin{equation} \label{eq:mahler_inequalities_for_v}
  \begin{cases}
    \max_{1\leq i\leq n}
    |v_{i+1}|\leq t
    \\
    |v_1+\theta_1v_2+\ldots+\theta_nv_{n+1}|\leq t^{-\gamma}
  \end{cases},
\end{equation}
%
%
Mahler applies Minkowski's theorem -- more specifically, Corollary \ref{cor:minkowski_linear_forms} -- to the parallelepiped consisting of the points $\vec u=(u_1,\ldots,u_{n+1})$ satisfying the inequalities
\begin{equation} \label{eq:mahler_parallelepiped_for_khintchine}
  \begin{cases}
    |v_1u_1+\ldots+v_{n+1}u_{n+1}|<1 \\
    \max_{1\leq i\leq n}
    |\theta_iu_1-u_{i+1}|\leq|v_1+\theta_1v_2+\ldots+\theta_nv_{n+1}|^{1/n}
    .
  \end{cases}
\end{equation}
Determinant of the set of linear forms involved in \eqref{eq:mahler_parallelepiped_for_khintchine} equals
\[ 
  \det
  \left(
  \begin{matrix}
    v_1     & \theta_1 & \theta_2 & \cdots & \theta_n \\
    v_2     & -1       &  0       & \cdots &  0       \\
    v_3     &  0       & -1       & \cdots &  0       \\
    \vdots  & \vdots   & \vdots   & \ddots & \vdots   \\
    v_{n+1} &  0       &  0       & \cdots & -1
  \end{matrix}
  \right)=
  (-1)^n(v_1+\theta_1v_2+\ldots+\theta_nv_{n+1}),
\]
which is equal by the absolute value to the product of the right-hand sides of \eqref{eq:mahler_parallelepiped_for_khintchine}. Therefore, by Corollary \ref{cor:minkowski_linear_forms} system \eqref{eq:mahler_parallelepiped_for_khintchine} admits a nonzero integer solution $\vec w=(w_1,\ldots,w_{n+1})$.

Since $\vec v$ is integer, it follows from the first inequality \eqref{eq:mahler_parallelepiped_for_khintchine} that
\begin{equation} \label{eq:orthogonal_plane}
  v_1w_1+\ldots+v_{n+1}w_{n+1}=0,
\end{equation}
whence
\begin{equation*} 
  w_1(v_1+\theta_1v_2+\ldots+\theta_nv_{n+1})=\sum_{i=1}^nv_{i+1}(\theta_iw_1-w_{i+1}).
\end{equation*}
Hence by the second inequality \eqref{eq:mahler_parallelepiped_for_khintchine}
\begin{equation} \label{eq:w_1_first_estimate}
  |w_1|\leq n\max_{1\leq i\leq n}|v_{i+1}||v_1+\theta_1v_2+\ldots+\theta_nv_{n+1}|^{\frac1n-1}.
\end{equation}
Combining \eqref{eq:mahler_inequalities_for_v}, \eqref{eq:mahler_parallelepiped_for_khintchine}, \eqref{eq:w_1_first_estimate}, we get
\begin{equation} \label{eq:mahler_inequalities_for_w}
  \begin{cases}
    |w_1|\leq nt^{1-\gamma\big(\frac1n-1\big)}=nt^{\frac{(n-1)\gamma+n}{n}}=t^{\frac{(n-1)\gamma+n}{n}+\frac{\ln n}{\ln t}} \\
    \max_{1\leq i\leq n}|\theta_iw_1-w_{i+1}|\leq
    t^{-\gamma/n}=\Big(t^{\frac{(n-1)\gamma+n}{n}+\frac{\ln n}{\ln t}}\Big)^{-\frac{\gamma}{(n-1)\gamma+n}+o(1)}
  \end{cases}.
\end{equation}


This proves \eqref{eq:khintchine_transference_mahler}.

The key ingredient in Mahler's method is relation \eqref{eq:orthogonal_plane}, which says that the required point happens to be in the orthogonal component to the line spanned by $\vec v$. More specifically, in the intersection of this orthogonal complement with the cylinder determined by the second inequality of \eqref{eq:mahler_parallelepiped_for_khintchine}, whose axis is the line spanned by $(1,\theta_1,\ldots,\theta_n)$.

%
%

Mahler generalised his method by proving in \cite{mahler_casopis_linear} his famous \emph{theorem on a bilinear form}.

\begin{theorem}[K.\,Mahler, 1937] \label{t:mahler}
  Assume two sets of homogeneous linear forms in $\vec u\in\mathbb R^d$ are given:

  forms $f_1(\vec u),\ldots,f_d(\vec u)$ with matrix $F$, $\det F\neq0$, and

  forms $g_1(\vec u),\ldots,g_d(\vec u)$ with matrix $G$, $\det G=1$. \\
  Suppose the bilinear form
  \begin{equation} \label{eq:mahler_BLF}
    \Phi(\vec u',\vec u'')=\sum_{i=1}^df_i(\vec u')g_i(\vec u'')
  \end{equation}
  has integer coefficients. 
  Suppose the system
  \begin{equation} \label{eq:mahler_f}
    |f_i(\vec u)|\leq\lambda_i,\ \ i=1,\ldots,d
  \end{equation}
  admits a solution in $\Z^d\backslash\{\vec 0\}$. 
  Then so does the system
  \begin{equation} \label{eq:mahler_g}
    |g_i(\vec u)|\leq(d-1)\lambda/\lambda_i,\ \ i=1,\ldots,d,
  \end{equation}
  where
  \begin{equation} \label{eq:mahler_lambda}
    \lambda=\Big(\prod_{i=1}^d\lambda_i\Big)^{\frac1{d-1}}.
  \end{equation}
\end{theorem}

As we shall see in Section \ref{sec:several_linear_forms}, this theorem provides a rather simple proof of transference inequalities in the most general problem of homogeneous linear Diophantine approximation -- when zero is to be approximated simultaneously by the values of several linear forms at integer points. We shall actually reformulate it in terms of pseudocompounds (Theorem \ref{t:mahler_reformulated} below). Thus formulated, Mahler's theorem becomes rather concise and very convenient for applications.


\subsubsection{Pseudocompounds and dual lattices}\label{subsubsec:pseudocompounds_and_dual_lattices}

In 1955, in his papers \cite{mahler_compound_bodies_I}, \cite{mahler_compound_bodies_II}, Mahler developed the theory of \emph{compound bodies} (see also Gruber and Lekkerkerker's book \cite{gruber_lekkerkerker}). This theory appeared to be rather fruitful in the context of problems related to the transference principle. We shall reformulate Theorem \ref{t:mahler} with the help of a construction from Schmidt's book \cite{schmidt_DA}, which is a simplification of what Mahler calls the \emph{$(d-1)$-th compound} of a parallelepiped.

\begin{definition}\label{def:pseudo_compound}
  Let $\eta_1,\ldots,\eta_d$ be positive real numbers. Consider the parallelepiped
  \begin{equation}\label{eq:pseudo_compound}
    \cP=\Big\{ \vec z=(z_1,\ldots,z_d)\in\R^d \,\Big|\, |z_i|\leq\eta_i,\ i=1,\ldots,d \Big\}.
  \end{equation}
  Then the parallelepiped
  \[
    \cP^\ast=\Big\{ \vec z=(z_1,\ldots,z_d)\in\R^d \,\Big|\, |z_i|\leq\frac1{\eta_i}\prod_{j=1}^d\eta_j,\ i=1,\ldots,d \Big\}
  \]
  is called the \emph{$(d-1)$-th pseudocompound} of $\cP$.
\end{definition}

We shall often call $\cP^\ast$ simply the \emph{compound} of $\cP$, omitting ``$(d-1)$-th''.

We also recall the definition of the dual lattice.

\begin{definition}\label{def:dual_lattice}
  Let $\La$ be a full rank lattice in $\R^d$. Then its \emph{dual} lattice is defined as
  \[
    \La^\ast=\big\{\, \vec z\in\R^d \,\big|\ \langle\vec z,\vec z'\rangle\in\Z\text{ for all }\vec z'\in\La \,\big\},
  \]
  where $\langle \,\cdot\,,\cdot\,\rangle$ denotes the inner product.
\end{definition}


Note that the relation of duality is symmetric in the case of lattices, i.e.
\[
  (\La^\ast)^\ast=\La.
\]

Let $F$ and $G$ be the matrices from Theorem \ref{t:mahler}. Consider the lattices $F\Z^d$ and $G\Z^d$. In view of Definition \ref{def:dual_lattice}, the fact that the coefficients of the form \eqref{eq:mahler_BLF} are integer means exactly that each of these two lattices is a sublattice of the other's dual one. Set $\La=G\Z^d$. Then $F\Z^d\subseteq\La^\ast$.

Given positive $\lambda_1,\ldots,\lambda_d$, let $\lambda$ be defined by \eqref{eq:mahler_lambda}. Set $\eta_i=\lambda/\lambda_i$, $i=1,\ldots,d$. Consider the parallelepiped $\cP$ determined by \eqref{eq:pseudo_compound}. Then
\begin{equation}\label{eq:lambdas_via_etas}
  \frac1{\eta_i}\prod_{j=1}^d\eta_j=
  \frac{\lambda_i\lambda^{d-1}}{\prod_{j=1}^d\lambda_j}=
  \lambda_i,\qquad
  i=1,\ldots,d,
\end{equation}
i.e.
\[
  \cP^\ast=\Big\{ \vec z=(z_1,\ldots,z_d)\in\R^d \,\Big|\, |z_i|\leq\lambda_i,\ i=1,\ldots,d \Big\}.
\]

Thus, Theorem \ref{t:mahler} states actually that there is a nonzero point of the unimodular lattice $\La$ in $(d-1)\cP$, provided that $\cP^\ast$ contains a nonzero point of some sublattice of $\La^\ast$. It is clear that in this statement the words ``of some sublattice'' can be omitted. We get the following reformulation of Theorem \ref{t:mahler}.

\begin{theorem}\label{t:mahler_reformulated}
  Let $\La$ be a full rank lattice in $\R^d$ with determinant $1$. Let $\cP$ be a parallelepiped in $\R^d$ centered at the origin with faces parallel to the coordinate planes. Then
  \begin{equation*}
    \cP^\ast\cap\La^\ast\neq\{\vec 0\}
    \implies
    (d-1)\cP\cap\La\neq\{\vec 0\}.
  \end{equation*}
\end{theorem}

Note that, since $(\La^\ast)^\ast=\La$, we can interchange $\La$ and $\La^\ast$ in Theorem \ref{t:mahler_reformulated}.

In this form, Mahler's theorem admits a rather vivid purely geometric proof. It can be described as follows.

Suppose $\cP^\ast$ contains a nonzero point $\vec v$ of $\La^\ast$. We can assume that $\vec v$ is primitive. Consider $(\R\vec v)^\perp$ -- the orthogonal complement to the one-dimensional subspace spanned by $\vec v$ and the section $\cS=\cP\cap(\R\vec v)^\perp$ (see Fig. \ref{fig:mahler}).

\begin{figure}[h]
\centering
\begin{tikzpicture}

  \begin{scope}[scale=1.0,x=0.5cm,y=4.5cm,z=-0.2cm]

  \draw (9,0.1,-9) -- (-9,0.1,-9);
  \draw[very thin] (-9,-0.1,-9) -- (9,-0.1,-9);
  \draw[very thin] (-9,-0.1,-9) -- (-9,0.1,-9);
  \draw[very thin] (-9,-0.1,-9) -- (-9,-0.1,9);

  \fill[white, opacity=1] (1.1,0.46,-1) -- (1.1,0.25,-1) -- (0.9,0.25,-1) -- (0.9,0.46,-1) -- cycle;
  \fill[white, opacity=1] (-1.1,0.46,-1) -- (-1.1,0.25,-1) -- (-0.9,0.25,-1) -- (-0.9,0.46,-1) -- cycle;
  \fill[white, opacity=1] (1.1,0.46+0.09,1) -- (1.1,0.25+0.09,1) -- (0.9,0.25+0.09,1) -- (0.9,0.46+0.09,1) -- cycle;
  \fill[white, opacity=1] (-1.1,0.46+0.09,1) -- (-1.1,0.25+0.09,1) -- (-0.9,0.25+0.09,1) -- (-0.9,0.46+0.09,1) -- cycle;

  \draw[very thin] (-1,-1,-1) -- (1,-1,-1);
  \draw[very thin] (-1,-1,-1) -- (-1,1,-1);
  \draw[very thin] (-1,-1,-1) -- (-1,-1,1);

  \fill[white, opacity=1] (-0.3,1.01,1) -- (-0.1,1.01,1) -- (-0.1,0.99,1) -- (-0.3,0.99,1) -- cycle;
  \fill[white, opacity=1] (0.3,-1.01,-1) -- (0.1,-1.01,-1) -- (0.1,-0.99,-1) -- (0.3,-0.99,-1) -- cycle;

  \draw[fill=blue!20!, opacity=0.7, draw=blue] (-1,1,0) -- (-1/6,1,-1) -- (1,-1/6,-1) -- (1,-1,0) -- (1/6,-1,1) -- (-1,1/6,1) -- cycle;

  \draw[blue] (0,0,0) -- (9*1,0.1*1-0.05,9*5/6);

  \draw (1,1,-1) -- (1,-1,-1) -- (1,-1,1) -- (-1,-1,1) -- (-1,1,1) -- (-1,1,-1) -- cycle;
  \draw (1,1,1) -- (1,1,-1);
  \draw (1,1,1) -- (1,-1,1);
  \draw (1,1,1) -- (-1,1,1);

  \fill[white, opacity=1] (7.3,-0.11,-9) -- (7.1,-0.11,-9) -- (7.1,-0.09,-9) -- (7.3,-0.09,-9) -- cycle;
  \fill[white, opacity=1] (-7.3,0.11,9) -- (-7.1,0.11,9) -- (-7.1,0.09,9) -- (-7.3,0.09,9) -- cycle;
  \fill[white, opacity=1] (2.18,0.11,9) -- (5.02,0.11,9) -- (5.02,0.09,9) -- (2.18,0.09,9) -- cycle;
  \fill[white, opacity=1] (2.18,-0.09,9) -- (5.02,-0.09,9) -- (5.02,-0.11,9) -- (2.18,-0.11,9) -- cycle;
  \fill[white, opacity=1] (7.27,-0.11,3.79) -- (7.13,-0.11,3.79) -- (7.13,-0.09,3.79) -- (7.27,-0.09,3.79) -- cycle;

  \draw (9,0.1,-9) -- (9,-0.1,-9) -- (9,-0.1,9) -- (-9,-0.1,9) -- (-9,0.1,9) -- (-9,0.1,-9);
  \draw (9,0.1,9) -- (9,0.1,-9);
  \draw (9,0.1,9) -- (9,-0.1,9);
  \draw (9,0.1,9) -- (-9,0.1,9);

  \node[fill=blue,circle,inner sep=1pt] at (9*1,0.1*1-0.05,9*5/6) {};
  \node[right] at (9*1,0.1*1-0.05,9*5/6) {$\vec v$};
  \node[fill=blue,circle,inner sep=1pt] at (0,0,0) {};
  \node[below] at (0,0,0) {$\vec 0$};
  \node[left] at (-9,0.13,4) {$\cP^\ast$};
  \node[left] at (-1.15,1,1) {$\cP$};
  \node[right] at (-0.4,0.75,1) {$\cS$};

  \end{scope}

\end{tikzpicture}
\caption{Section $\cS$ of $\cP$ orthogonal to $\vec v$} \label{fig:mahler}
\end{figure}

The set $\Gamma=\La\cap(\R\vec v)^\perp$ is a lattice of rank $d-1$ with determinant equal to $|\vec v|_2$, where $|\cdot|_2$ denotes the Euclidean norm. Hence, if a constant $c$ is chosen so that the $(d-1)$-dimensional volume of $c\cS$ is smaller than $2^{d-1}|\vec v|_2$, then by Minkowski's convex body theorem there is a nonzero point of $\Gamma$ in $c\cS$ and, therefore, there is a nonzero point of $\La$ in $c\cP$.

Finding an appropriate constant $c$ can be arranged with the help of Vaaler's theorem \cite{vaaler} on central sections of a cube. Consider the cube $\cB=[-1,1]^d$. By Vaaler's theorem the volume of any central $(d-1)$-dimensional section of $\cB$ is not less than $2^{d-1}$.
Consider the diagonal operators $A=\textup{diag}\big(\eta_1^{-1},\ldots,\eta_d^{-1}\big)$ and $C=\textup{diag}\big(\lambda_1^{-1},\ldots,\lambda_d^{-1}\big)$. Then
\[
  A\cP=C\cP^\ast=\cB.
\]
By \eqref{eq:lambdas_via_etas} $C$ coincides with the cofactor matrix of $A$, i.e. $C=(\det A)(A^\ast)^{-1}$. Therefore, since $\cS\perp\vec v$, we have
\[
  \frac{\vol(\cS)}{|\vec v|_2}=
  (\det A)^{-1}\frac{\vol(A\cS)}{\big|(A^\ast)^{-1}\vec v\big|_2}=
  \frac{\vol(A\cS)}{|C\vec v|_2}\geq
  \frac{2^{d-1}}{\sqrt d}\,.
\]
Thus, by choosing $c=\big(\sqrt d\big)^{1/(d-1)}$ we get $\vol(c\cS)\geq2^{d-1}|\vec v|_2$. As it was said before, the rest is done by applying Minkowski's convex body theorem to $c\cS$ and $\Gamma$.

The argument we have just discussed actually proves a stronger statement than Theorem \ref{t:mahler_reformulated}. It proves that
\begin{equation}\label{eq:mahler_improved}
  \cP^\ast\cap\La^\ast\neq\{\vec 0\}
  \implies
  c\cP\cap\La\neq\{\vec 0\}
\end{equation}
with $c=\big(\sqrt d\big)^{1/(d-1)}$ -- a constant which is less than $d-1$ and, moreover, which tends to $1$ as $d\to\infty$. However, it is worth mentioning that a combination of Mahler's theorem on successive minima proved in \cite{mahler_casopis_convex} with Minkowski's theorem on successive minima (both theorems can be found in Schmidt's book \cite{schmidt_DA}) provides an improvement of Theorem \ref{t:mahler} with the constant $c^2$. A detailed account can be found in \cite{german_evdokimov}. In that same paper some further improvements of Theorem \ref{t:mahler} are presented.

\subsubsection{Two two-parametric families of parallelepipeds} \label{subsubsec:parallelepiped_families}

Let us address once again Theorem \ref{t:khintchine_transference} of Khintchine, this time -- in the light of Theorem \ref{t:mahler_reformulated}.
Consider the lattice
\[
  \La=
  \left(
  \begin{matrix}
        1     &    0   &    0   & \cdots &    0   \\
    -\theta_1 &    1   &    0   & \cdots &    0   \\
    -\theta_2 &    0   &    1   & \cdots &    0   \\
     \phantom{-}
     \vdots   & \vdots & \vdots & \ddots & \vdots \\
    -\theta_n &    0   &    0   & \cdots &    1
  \end{matrix}
  \right)
  \Z^{n+1}.
\]
Then the dual lattice is
\[
  \La^\ast=
  \left(
  \begin{matrix}
       1   & \theta_1 & \theta_2 & \cdots & \theta_n \\
       0   &    1     &    0     & \cdots &    0     \\
       0   &    0     &    1     & \cdots &    0     \\
    \vdots & \vdots   & \vdots   & \ddots & \vdots   \\
       0   &    0     &    0     & \cdots &    1
  \end{matrix}
  \right)
  \Z^{n+1}.
\]
For every positive $t$, $\gamma$, $s$, $\delta$ define the parallelepipeds
\begin{align}
  \label{eq:t_gamma_family}
  \cP(t,\gamma) & =\Bigg\{\,\vec z=(z_1,\ldots,z_{n+1})\in\R^{n+1} \ \Bigg|
                        \begin{array}{l}
                          |z_1|\leq t \\
                          |z_i|\leq t^{-\gamma},\ \ i=2,\ldots,n+1
                        \end{array} \Bigg\}, \\
  \label{eq:s_delta_family}
  \cQ(s,\delta) & =\Bigg\{\,\vec z=(z_1,\ldots,z_{n+1})\in\R^{n+1} \ \Bigg|
                        \begin{array}{l}
                          |z_1|\leq s^{-\delta} \\
                          |z_i|\leq s,\quad\,\ i=2,\ldots,n+1
                        \end{array} \Bigg\}.
\end{align}
Then
\begin{equation}\label{eq:omega_vs_parallelepipeds}
\begin{aligned}
  \omega(\Theta) & =\sup\bigg\{ \gamma\geq\frac1n \,\bigg|\ \forall\,t_0\in\R\,\ \exists\,t>t_0:\text{\,we have }\cP(t,\gamma)\cap\La\neq\{\vec 0\} \bigg\}, \\
  \omega(L_\Theta) & =\sup\bigg\{ \,\delta\geq n\ \bigg|\ \forall\,s_0\in\R\,\ \exists\,s>s_0:\text{\,we have }\cQ(s,\delta)\cap\La^\ast\neq\{\vec 0\} \bigg\}.
\end{aligned}
\end{equation}
Each parallelepiped from \eqref{eq:s_delta_family} is the pseudocompound of some parallelepiped from \eqref{eq:t_gamma_family}, and vice versa. Indeed, if
\begin{equation}\label{eq:Q_is_P_star}
  t=s^{((n-1)\delta+n)/n},
  \qquad
  \gamma=\frac{\delta}{(n-1)\delta+n}\,,
\end{equation}
then $\cQ(s,\delta)=\cP(t,\gamma)^\ast$. Conversely, if
\begin{equation}\label{eq:P_is_Q_star}
  s=t^{1/n},
  \qquad
  \delta=n\gamma+n-1,
\end{equation}
then $\cP(t,\gamma)=\cQ(s,\delta)^\ast$.


Let us apply Mahler's theorem in disguise of Theorem \ref{t:mahler_reformulated}.

Assuming \eqref{eq:Q_is_P_star} holds, we have
\[
  \cQ(s,\delta)\cap\La^\ast\neq\{\vec 0\}
  \implies
  (d-1)\cP(t,\gamma)\cap\La\neq\{\vec 0\},
\]
whence it follows, in view of \eqref{eq:omega_vs_parallelepipeds}, that
\[
  \omega(L_\Theta)\geq\delta
  \implies
  \omega(\Theta)\geq\gamma=\frac{\delta}{(n-1)\delta+n}\,.
\]
Thus,
\begin{equation} \label{eq:khintchine_transference_right}
  \omega(\Theta)\geq\frac{\omega(L_\Theta)}{(n-1)\omega(L_\Theta)+n}\,.
\end{equation}

Assuming \eqref{eq:P_is_Q_star} holds, we have
\[
  \cP(t,\gamma)\cap\La\neq\{\vec 0\}
  \implies
  (d-1)\cQ(s,\delta)\cap\La^\ast\neq\{\vec 0\},
\]
whence it follows, once again, in view of \eqref{eq:omega_vs_parallelepipeds}, that
\[
  \omega(\Theta)\geq\gamma
  \implies
  \omega(L_\Theta)\geq\delta=n\gamma+n-1.
\]
Thus,
\begin{equation} \label{eq:khintchine_transference_left}
  \omega(L_\Theta)\geq n\omega(\Theta)+n-1.
\end{equation}

It can be easily verified that \eqref{eq:khintchine_transference_right} and \eqref{eq:khintchine_transference_left} are nothing else but the right-hand and the left-hand inequalities \eqref{eq:khintchine_transference} respectively. The proof of Khintchine's theorem is completed.

\subsubsection{Uniform exponents and an analogue of Mahler's theorem}

The proof of Theorem \ref{t:mahler_reformulated} given in Section \ref{subsubsec:pseudocompounds_and_dual_lattices} is based on Minkowski's convex body theorem, Vaaler's theorem on central sections of a cube, and the fact that for any primitive vector $\vec v$ of a unimodular lattice $\La$ the set $\La^\ast\cap(\R\vec v)^\perp$ is a lattice of rank $d-1$ with determinant equal to the Euclidean norm of $\vec v$. The latter fact can be reformulated as the equality of the determinants of the lattices $\La\cap(\R\vec v)$ and $\La^\ast\cap(\R\vec v)^\perp$ of ranks $1$ and $d-1$ respectively.

A more general statement holds. It is very useful when working with uniform exponents.

\begin{proposition}\label{prop:orthogonal_sublattices}
  Let $\La$ be a full rank lattice in $\R^d$, $\det\La=1$. Let $\cL$ be a $k$-dimensional subspace of $\R^d$ and let the lattice $\Gamma=\cL\cap\La$ have rank $k$. Consider the orthogonal complement $\cL^\perp$ and set $\Gamma^\perp=\cL^\perp\cap\La^\ast$. Then $\Gamma^\perp$ is a lattice of rank $d-k$ and
  \[
    \det\Gamma^\perp=\det\Gamma.
  \]
\end{proposition}

This statement is rather well known, its proof can be found, for instance, in \cite{schmidt_DADE} or \cite{german_MJCNT_2012}.

The nature of uniform exponents requires working with sublattices of rank $2$, therefore, Proposition \ref{prop:orthogonal_sublattices} is applied with $k=2$. Respectively, when working with two-dimensional and $(d-2)$-dimensional subspaces, it is natural to involve the $(d-2)$-th pseudocompounds instead of the $(d-1)$-th ones.

Let $\vec e_1,\ldots,\vec e_d$ be the standard basis of $\R^d$. For each multivector $\vec Z\in\bigwedge^2\R^d$, let us consider its representation
\[
  \vec Z=\sum_{1\leq i<j\leq d}Z_{ij}\vec e_i\wedge\vec e_j.
\]

\begin{definition}\label{def:second_pseudo_compound}
  Let $\eta_1,\ldots,\eta_d$ be positive real numbers. Consider the parallelepiped
  \[
    \cP=\Big\{ \vec z=(z_1,\ldots,z_d)\in\R^d \,\Big|\, |z_i|\leq\eta_i,\ i=1,\ldots,d \Big\}.
  \]
  Then the parallelepiped
  \[\cP^\circledast=\bigg\{ \vec Z 
                                   \in{\textstyle\bigwedge^2\R^d} \,\bigg|\,
                             |Z_{ij}|\leq\frac1{\eta_i\eta_j}\prod_{k=1}^d\eta_k,\ 1\leq i<j\leq d \bigg\}\]
  is called the \emph{$(d-2)$-th pseudocompound} of $\cP$.
\end{definition}

Given a full rank lattice $\La$ in $\R^d$ and its dual lattice $\La^\ast$, let us denote by $\La^{\circledast}$ the set of all the decomposable elements of the lattice $\bigwedge^2\La^\ast$, i.e.
\begin{equation*}
  \La^{\circledast}=\Big\{ \vec z_1\wedge\vec z_2 \,\Big|\, \vec z_1,\vec z_2\in\La^\ast \Big\}.
\end{equation*}

The following theorem proved in \cite{german_mathmatika_2020} is an analogue of Theorem \ref{t:mahler_reformulated}.

\begin{theorem}\label{t:second_pseudo_compound}
  Let $\La$ be a full rank lattice in $\R^d$ with determinant $1$. Let $\cP$ be a parallelepiped in $\R^d$ centered at the origin with faces parallel to the coordinate planes. Then
  \[
    \cP^{\circledast}\cap\La^{\circledast}\neq\{\vec 0\}
    \implies
    c\cP\cap\La\neq\{\vec 0\},
  \]
  where $c=\big(\frac{d(d-1)}2\big)^{\raisebox{1ex}{$\frac1{2(d-2)}$}}$.
\end{theorem}

The proof of Theorem \ref{t:second_pseudo_compound} resembles pretty much the proof of Theorem \ref{t:mahler_reformulated} described in Section \ref{subsubsec:pseudocompounds_and_dual_lattices}. It is also based on three facts. The first one is that same Minkowski's convex body theorem. The second one is Vaaler's theorem on central sections of a cube -- this time it is applied to $(d-2)$-dimensional sections. The third ingredient is Proposition \ref{prop:orthogonal_sublattices} with $k=2$.

Theorem \ref{t:second_pseudo_compound} is the key tool for proving Theorem \ref{t:my_simultaneous_vs_linear_form} on uniform exponents. It is applied within the frames of a parametric construction, which is most comprehensively described in terms of ``nodes'' and ``leaves''.

\subsubsection{Families of ``nodes'', ``antinodes'', and ``leaves''}\label{subsubsec:nodes_and_leaves}

%

Let us describe the construction of ``nodes'' and ``leaves'' in the form adapted for the proof of the right-hand inequality \eqref{eq:my_simultaneous_vs_linear_form}, i.e. the inequality
\begin{equation}\label{eq:my_simultaneous_vs_linear_form_right}
  \hat\omega(\Theta)\geq \frac{1-\hat\omega(L_\Theta)^{-1}}{n-1}\,.
\end{equation}
Let us adopt the notation \eqref{eq:t_gamma_family} and \eqref{eq:s_delta_family} once again.

Fix arbitrary $h,\alpha,\beta\in\R$ such that
\[
  h>1,\quad\beta>0,\quad 0<\alpha\leq\beta
\]
and set
\[
  H=h^{\beta/\alpha}.
\]
To each $r$ in the interval $h\leq r\leq H$, let us assign the parallelepiped
\begin{equation}\label{eq:Q_r}
  \cQ_r=\Bigg\{\,\vec z=(z_1,\ldots,z_d) \in\R^d \ \Bigg|
                 \begin{array}{l}
                   |z_1|\leq (hH/r)^{-\alpha} \\
                   |z_i|\leq r,\qquad\qquad i=2,\ldots,n+1
                 \end{array} \Bigg\}.
\end{equation}
It is easy to see that $\cQ_r$ belongs to \eqref{eq:s_delta_family}:
\[
  \cQ_r=\cQ\big(r,\alpha\log_r(hH/r)\big).
\]
Consider the following three families of parallelepipeds:
\begin{equation*}
\begin{aligned}
  \mathfrak S & =\mathfrak S(h,\alpha,\beta)=\Big\{ \cQ_r \,\Big|\, h\leq r\leq\sqrt{hH} \Big\}, \\
  \mathfrak A & =\mathfrak A(h,\alpha,\beta)=\Big\{ \cQ_r \,\Big|\, \sqrt{hH}\leq r\leq H \Big\}, \\
  \mathfrak L & =\mathfrak L(h,\alpha,\beta)=\Big\{ \cQ(r,\alpha) \,\Big|\, h\leq r\leq H \Big\}.
\end{aligned}
\end{equation*}
Let us call $\mathfrak S$ the \emph{``stem'' family}, $\mathfrak A$ the \emph{``anti-stem'' family}, $\mathfrak L$ the \emph{``leaves'' family}. Let us call each element of $\mathfrak S$ a \emph{``node''}, each element of $\mathfrak A$ an \emph{``anti-node''}, each element of $\mathfrak L$ a \emph{``leaf''}.
We say that a ``leaf'' $\cQ(r,\alpha)$ is \emph{produced} by a ``node'' or an ``anti-node'' $\cQ_{r'}$ if
\begin{equation*}
  r=r'\quad\text{ or }\quad r=hH/r'.
\end{equation*}
We call $\cQ_h$ the \emph{root} ``node''.

``Nodes'' and ``leaves'' enjoy the following properties:
\begin{itemize}
  \item[(i)]

    $\cQ_h=\cQ(h,\beta)$;

  \item[(ii)]

    if $r<r'$, then $\cQ_r\subset\cQ_{r'}$;

  \item[(iii)]

    for each $r$ in the interval $h\leq r\leq\sqrt{hH}$ the ``node'' $\cQ_r$ and the ``anti-node'' $\cQ_{hH/r}$ produce exactly two ``leaves''
    \[
      \cQ(r,\alpha)
      \quad\text{ and }\quad
      \cQ(hH/r,\alpha),
    \]
    while the intersection of these ``leaves'' coincides with the ``node'' $\cQ_r$\,, and their union is contained in the ``anti-node'' $\cQ_{hH/r}$\,;

  \item[(iv)]

    every ``leaf'' $\cQ(r,\alpha)$ is produced by exactly one ``node'' $\cQ_{r'}$ and one ``anti-node'' $\cQ_{hH/r'}$, where
    \begin{equation*}
      r'=
      \begin{cases}
        r,\qquad\ \ \text{ if }r\leq\sqrt{hH} \\
        hH/r,   \ \ \text{ if }r\geq\sqrt{hH}
      \end{cases};
    \end{equation*}

  \item[(v)]

    if every ``leaf'' in $\mathfrak L$ contains nonzero points of $\La^\ast$, but the root ``node'' does not, then there is a ``leaf'' that contains at least two non-collinear points of $\La^\ast$, one of which lies in the ``node'' that produces the ``leaf''.
\end{itemize}

%
%
%

These properties are illustrated by Fig. \ref{fig:nodes_and_leaves}, where we use $u$ and $v$ to denote $\max\big(|z_2|,\ldots,|z_{n+1}|\big)$ and $|z_1|$ respectively.

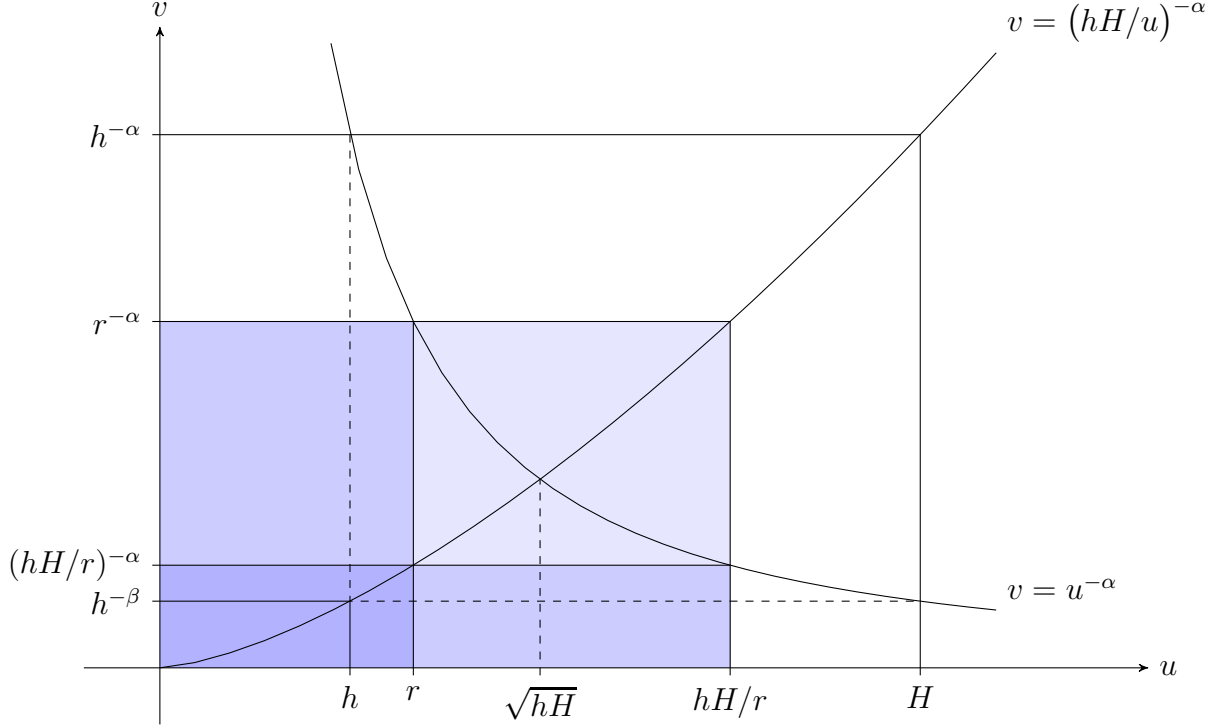
\begin{figure}[h]
\centering
\begin{tikzpicture}[scale=5]
  \fill[blue!10!] (0,0.5*1.5^1.5) -- (1.5,0.5*1.5^1.5) -- (1.5,0) -- (0,0) -- cycle;
  \fill[blue!20!] (0,0.5*2^1.5/3^1.5) -- (1.5,0.5*2^1.5/3^1.5) -- (1.5,0) -- (0,0) -- cycle;
  \fill[blue!20!] (0,0.5*1.5^1.5) -- (2/3,0.5*1.5^1.5) -- (2/3,0) -- (0,0) -- cycle;
  \fill[blue!30!] (0,0.5*2^1.5/3^1.5) -- (2/3,0.5*2^1.5/3^1.5) -- (2/3,0) -- (0,0) -- cycle;

  \draw[->,>=stealth'] (-0.2,0) -- (2.6,0) node[right] {$u$};
  \draw[->,>=stealth'] (0,-0.15) -- (0,0.5*3.4) node[above] {$v$};

  \draw[color=black] plot[domain=0:2.2] (\x, 0.5*\x^1.5) node[above right]{$v=\big(hH/u\big)^{-\alpha}$};
  \draw[color=black] plot[domain=0.45:2.2] (\x, 0.5/\x^1.5) node[above right]{$v=u^{-\alpha}$};

  \draw[color=black] (-0.02,0.5*2^1.5) -- (2,0.5*2^1.5) -- (2,-0.02);
  \draw[color=black] (-0.02,0.5*0.5^1.5) -- (0.5,0.5*0.5^1.5) -- (0.5,-0.02);
  \draw[color=black] (-0.02,0.5*2^1.5/3^1.5) -- (1.5,0.5*2^1.5/3^1.5);
  \draw[color=black] (2/3,0.5*1.5^1.5) -- (2/3,-0.02);
  \draw[color=black] (-0.02,0.5*1.5^1.5) -- (1.5,0.5*1.5^1.5) -- (1.5,-0.02);

  \draw[color=black,dashed] (0.5,0.5*2^1.5) -- (0.5,0.5*0.5^1.5) -- (2,0.5*0.5^1.5);
  \draw[color=black,dashed] (1,-0.02) -- (1,0.5);

  \draw (0.5,-0.02) node[below]{$h$};
  \draw (2/3,-0.02) node[below]{$r$};
  \draw (1,-0.02) node[below]{$\sqrt{hH}$};
  \draw (1.5,-0.02) node[below]{$hH/r$};
  \draw (2,-0.02) node[below]{$H$};

  \draw (-0.02,0.5*0.5^1.5) node[left]{$h^{-\beta}$};
  \draw (-0.02,0.5*2^1.5/3^1.5) node[left]{$(hH/r)^{-\alpha}$};
  \draw (-0.02,0.5*1.5^1.5) node[left]{$r^{-\alpha}$};
  \draw (-0.02,0.5*2^1.5) node[left]{$h^{-\alpha}$};
\end{tikzpicture}
\caption{A ``node'', its ``anti-node'', and their ``leaves''} \label{fig:nodes_and_leaves}
\end{figure}

The construction we have described is adapted, as we said before, for the proof of \eqref{eq:my_simultaneous_vs_linear_form_right}, which is the right-hand inequality \eqref{eq:my_simultaneous_vs_linear_form}. Nevertheless, Fig. \ref{fig:nodes_and_leaves} can be used unaltered for the proof of the left-hand inequality \eqref{eq:my_simultaneous_vs_linear_form}. Moreover, it can also be used in more general problems -- in the problem of approximating zero with the values of several linear forms (Section \ref{sec:several_linear_forms}) and in the analogous problem with weights (Section \ref{sec:weights}).

\subsubsection{Outline of the proof of the transference inequalities for uniform exponents}\label{subsubsec:proof_of_uniform_transference}

Let us demonstrate how to apply Theorem \ref{t:second_pseudo_compound} and the parametric construction described above to prove the right-hand inequality \eqref{eq:my_simultaneous_vs_linear_form}. By analogy with \eqref{eq:omega_vs_parallelepipeds}, we have
\begin{equation}\label{eq:omega_hat_vs_parallelepipeds}
\begin{aligned}
  \hat\omega(\Theta) & =\sup\bigg\{ \gamma\geq\frac1n \,\bigg|\ \exists\,t_0\in\R:\ \forall\,t>t_0\text{ we have }\cP(t,\gamma)\cap\La\neq\{\vec 0\} \bigg\}, \\
  \hat\omega(L_\Theta) & =\sup\bigg\{ \,\delta\geq n\ \bigg|\ \exists\,s_0\in\R:\ \forall\,s>s_0\text{ we have }\cQ(s,\delta)\cap\La^\ast\neq\{\vec 0\} \bigg\}.
\end{aligned}
\end{equation}
Let us fix $s>1$ and $\delta\geq n$. Same as in \eqref{eq:Q_is_P_star}, let us set
\begin{equation*}
  t=s^{((n-1)\delta+n)/n},
  \qquad
  \gamma=\frac{\delta}{(n-1)\delta+n}\,.
\end{equation*}
Set also
\[
  h=s,\qquad
  \beta=\delta,\qquad
  \alpha=\frac{(n-1)\delta+n}{n}\,.
\]
Note that, with such a choice of parameters, the quantities $\gamma$ and $\alpha$ are related by
\begin{equation}\label{eq:gamma_via_alpha}
  \gamma=\frac{1-\alpha^{-1}}{n-1}\,.
\end{equation}
Then $\cQ(s,\delta)$ is the root ``node''. If it contains a nonzero point of $\La^\ast$, then by Theorem \ref{t:mahler_reformulated} or, to be more exact, by its improved version \eqref{eq:mahler_improved} we have
\begin{equation}\label{eq:if_the_root_is_nonempty}
  \cQ(s,\delta)\cap\La^\ast\neq\{\vec 0\}
  \implies
  c_1\cP(t,\gamma)\cap\La\neq\{\vec 0\},
\end{equation}
where $c_1=d^{\raisebox{1ex}{$\frac1{2(d-1)}$}}$. If $\cQ(s,\delta)$ does not contain nonzero points of $\La^\ast$, let us assume that every ``leaf'' $\cQ(r,\alpha)$ in $\mathfrak L$ does contain nonzero points of $\La^\ast$ and consider a ``leaf'' $\cQ(r_0,\alpha)$ whose existence is guaranteed by property (v) from the previous Section. Then there exist non-collinear points $\vec v_1$ and $\vec v_2$ in $\La^\ast$ such that
\[
  \vec v_1\in\cQ_{r_0}\,,
  \qquad
  \vec v_2\in\cQ_{hH/r_0}\,.
\]
The parameters of the parallelepipeds $\cQ_{r_0}$ and $\cQ_{hH/r_0}$ are known, hence we can estimate from above the coordinates of the multivector 
\[
  \vec v_1\wedge\vec v_2=\sum_{1\leq i<j\leq n+1}V_{ij}\vec e_i\wedge\vec e_j.
\]
Detailed calculations can be found in \cite{german_mathmatika_2020}. The estimates lead eventually to the key relation
\[
  \vec v_1\wedge\vec v_2\in2\cP(t,\gamma)^{\circledast},
\]
which enables us to apply Theorem \ref{t:second_pseudo_compound}. We get the following chain of implications:
\begin{multline}\label{eq:if_the_root_is_empty}
  \cQ(r,\alpha)\cap\La^\ast\neq\{\vec 0\}\text{ for each }r\in[h,H]
  \implies \\
  \vphantom{\frac{\big|}{}}
  \hskip 40mm
  \implies
  2\cP(t,\gamma)^{\circledast}\cap\La^{\circledast}\neq\{\vec 0\}\implies
  c_2\cP(t,\gamma)\cap\La\neq\{\vec 0\},
\end{multline}
where $c_2=\big(2d(d-1)\big)^{\raisebox{1ex}{$\frac1{2(d-2)}$}}$.

Thus, if $\cQ(s,\delta)$ contains a nonzero point of $\La^\ast$, then by \eqref{eq:if_the_root_is_nonempty} there is a nonzero point of $\La$ in $c_1\cP(t,\gamma)$. If $\cQ(s,\delta)$ does not contain nonzero points of $\La^\ast$, but each $\cQ(r,\alpha)$ in $\mathfrak L$ does, then by \eqref{eq:if_the_root_is_empty} a nonzero point of $\La$ can be found in $c_2\cP(t,\gamma)$.

Taking into account \eqref{eq:gamma_via_alpha}, we get
\[
  \hat\omega(L_\Theta)\geq\alpha
  \implies
  \hat\omega(\Theta)\geq\gamma=\frac{1-\alpha^{-1}}{n-1}\,,
\]
whence it follows that
\[
  \hat\omega(\Theta)\geq \frac{1-\hat\omega(L_\Theta)^{-1}}{n-1}\,.
\]
This is how the right-hand inequality \eqref{eq:my_simultaneous_vs_linear_form} is proved.

Proving the left-hand inequality \eqref{eq:my_simultaneous_vs_linear_form} requires an ``inverse'' argument. Instead of $\cQ_r$, parallelepipeds 
\begin{equation*}
  \cP_r=\Bigg\{\,\vec z=(z_1,\ldots,z_d) \in\R^d \ \Bigg|
                 \begin{array}{l}
                   |z_1|\leq r \\
                   |z_i|\leq (hH/r)^{-\alpha},\ \ i=2,\ldots,n+1
                 \end{array} \Bigg\}
\end{equation*}
are to be considered, $\cQ$ should be replaced with $\cP$ in the definitions of the families $\mathfrak S$, $\mathfrak A$, $\mathfrak L$, and also, $u$ and $v$ in Fig. \ref{fig:nodes_and_leaves} should denote $|z_1|$ and $\max\big(|z_2|,\ldots,|z_{n+1}|\big)$ respectively (in the previous Section it was vice versa). Note that nothing changes in Fig. \ref{fig:nodes_and_leaves} itself. Upon fixing $t>1$ and $\gamma\geq1/n$ we should set, as in \eqref{eq:P_is_Q_star},
\[
  s=t^{1/n},
  \qquad
  \delta=n\gamma+n-1,
\]
and also
\[
  h=t,\qquad
  \beta=\gamma,\qquad
  \alpha=\frac{n\gamma}{n\gamma+n-1}\,.
\]
With such a choice of parameters, the quantities $\delta$ and $\alpha$ are related by
\[
  \delta=\frac{n-1}{1-\alpha}\,.
\]
Further argument is absolutely analogous to the one discussed above. It leads to the implication
\[
  \hat\omega(\Theta)\geq\alpha
  \implies
  \hat\omega(L_\Theta)\geq\delta=\frac{n-1}{1-\alpha}\,,
\]
whence we get
\[
  \hat\omega(L_\Theta)\geq \frac{n-1}{1-\hat\omega(\Theta)}\,.
\]
This is how the left-hand inequality \eqref{eq:my_simultaneous_vs_linear_form} and, therefore, the whole Theorem \ref{t:my_simultaneous_vs_linear_form} is proved.

\subsubsection{Empty cylinder lemma and ``mixed'' inequalities}

Inequalities \eqref{eq:my_simultaneous_vs_linear_form} estimate uniform exponents from below. For estimating them from above -- for instance, as in Theorem \ref{t:schmidt_summerer_2013} of Schmidt and Summerer -- the following rather simple statement proved in \cite{german_moshchevitin_2022} is very effective. It is reasonable to call it ``empty cylinder lemma''. 

As before, let us denote by $\cL=\cL(\Theta)$ the one-dimensional subspace spanned by $(1,\theta_1,\ldots,\theta_n)$ and by $\cL^\perp$ its orthogonal complement. For $\vec u\in\R^{n+1}$ let us denote by $r(\vec u)$ the Euclidean distance from $\vec u$ to $\cL$ and by $h(\vec u)$ the Euclidean distance from $\vec u$ to $\cL^\perp$.

\begin{lemma}[``Empty cylinder lemma'']\label{l:empty_cylinder}
  Let $t,\alpha,\beta$ be positive real numbers such that $t^{\beta-\alpha}\geq2$ (or, equivalently, $t^{-\alpha}-t^{-\beta}\geq t^{-\beta}$). Suppose $\vec v\in\Z^{n+1}$ satisfies
  \begin{equation}\label{eq:empty_cylinder_v}
    r(\vec v)=t^{\alpha-1-\beta},\qquad
    h(\vec v)=t^\alpha.
  \end{equation}
  Consider the half open cylinder
  \begin{equation}\label{eq:empty_cylinder_C}
    \cC=\cC(t,\alpha,\beta)=\Big\{\, \vec u\in\R^{n+1}\ \Big|\ r(\vec u)<t,\ t^{-\beta}\leq h(\vec u)\leq t^{-\alpha}-t^{-\beta} \,\Big\}.
  \end{equation}
  Then $\cC\cap\Z^{n+1}=\varnothing$.
\end{lemma}

Geometrical meaning of this statement is that the cylinder $\cC$ is contained in the open ``layer'' between the planes determined by the equations $\langle\vec v,\vec u\rangle=0$ and $\langle\vec v,\vec u\rangle=1$, where $\langle\,\cdot\,,\cdot\,\rangle$ denotes, as before, the inner product in $\R^{n+1}$. To see this, let us consider the two-dimensional subspace $\pi$ spanned by $\cL$ and $\vec v$. It is demonstrated in Fig. \ref{fig:empty_cylinder}. Consider an arbitrary point $\vec w\in\cC$ and let $\vec w'$ be its orthogonal projection onto $\pi$. Within the plane $\pi$ the functionals $h(\cdot)$ and $r(\cdot)$ can be identified with the absolute values of coordinates of points w.r.t. the coordinate axes $\cL$ and $\cL^\perp\cap\pi$. Then
\[
  \langle\vec v,\vec w\rangle=
  \langle\vec v,\vec w'\rangle>
  \begin{pmatrix}
    t^{\alpha-1-\beta} & t^\alpha
  \end{pmatrix}
  \begin{pmatrix}
    -t \\
    t^{-\beta}
  \end{pmatrix}=0
  \qquad
\]
and
\[
  \langle\vec v,\vec w\rangle=
  \langle\vec v,\vec w'\rangle<
  \begin{pmatrix}
    t^{\alpha-1-\beta} & t^\alpha
  \end{pmatrix}
  \begin{pmatrix}
    t \\
    t^{-\alpha}-t^{-\beta}
  \end{pmatrix}=1.
\]
Thus, $0<\langle\vec v,\vec w\rangle<1$ and, therefore, $\vec w$ cannot be an integer point. Lemma \ref{l:empty_cylinder} is proved.

\begin{figure}[h]
\centering
\begin{tikzpicture}[scale=0.5]
  \draw[->,>=stealth'] (-12,0) -- (13.3,0) node[right] {$\cL^\perp\cap\pi$};
  \draw[->,>=stealth'] (0,-3) -- (0,14) node[above] {$\cL$};

  \draw[color=blue] (0,0) -- (3,12);
  \node[fill=blue,circle,inner sep=1.2pt] at (3,12) {};
  \draw (3,12) node[right]{$\vec v$};

  \draw (1.5,12) node[above]{$r(\vec v)$};
  \draw (3,8) node[right]{$h(\vec v)$};

  \draw[color=blue] plot[domain=-12:12] (\x, {-\x/4}) node[right,color=black]{$\langle\vec v,\vec u\rangle=0$};
  \draw[color=blue] plot[domain=-12:12] (\x, {-\x/4+7}) node[right,color=black]{$\langle\vec v,\vec u\rangle=1$};

  \fill[blue,opacity=0.2] (8,2) -- (8,5) -- (-8,5) -- (-8,2) -- cycle;
  \draw[color=blue] (-8,2) -- (8,2);
  \draw[color=blue] (-8,5) -- (8,5);

  \draw[dashed] (-0.1,12) -- (3,12);
  \draw[dashed] (3,-0.1) -- (3,12);
  \draw[dashed] (8,-0.1) -- (8,2);
  \draw[dashed] (-8,-0.1) -- (-8,2);

  \draw (8,0) node[above right]{$t$};
  \draw (-8,0) node[above left]{$-t$};
  \draw (3,0) node[above right]{$t^{\alpha-1-\beta}$};
  \draw (0,12) node[above left]{$t^{\alpha}$};
  \draw (0,5) node[above left]{$t^{-\alpha}-t^{-\beta}$};
  \draw (0,2) node[above left]{$t^{-\beta}$};

  \draw (-7.1,3.5) node[right]{$\cC$};

  \node[fill=blue,circle,inner sep=1.2pt] at (6.5,3) {};
  \draw (6.5,3.15) node[right]{$\vec w'$};
\end{tikzpicture}
\caption{Empty cylinder} \label{fig:empty_cylinder}
\end{figure}
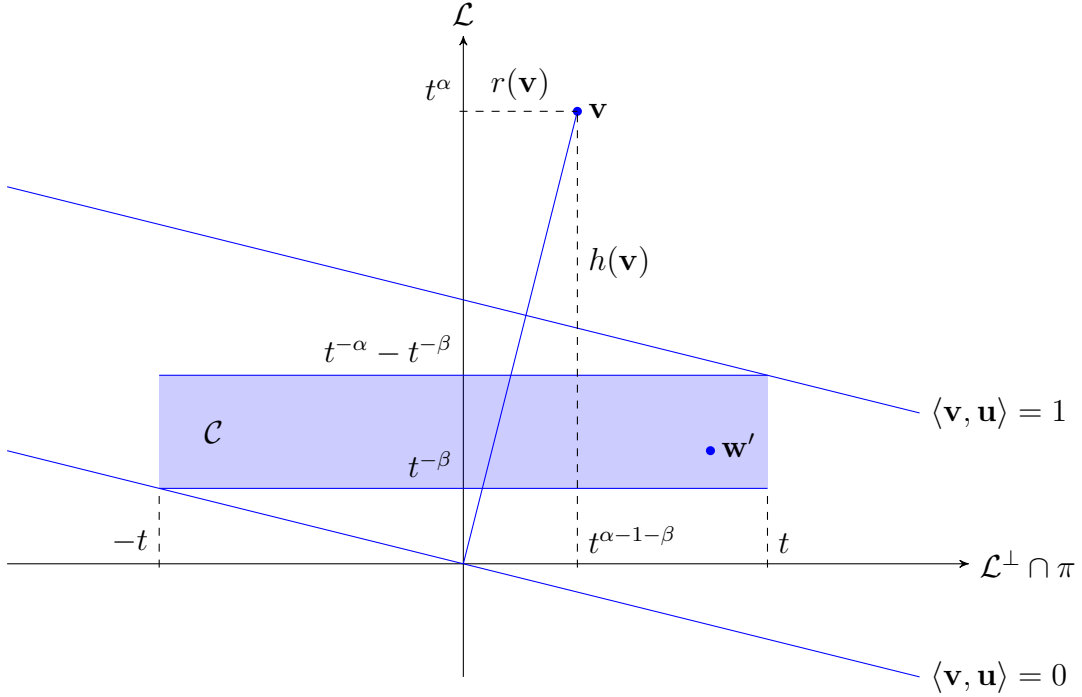

Let us outline the proof of Theorem \ref{t:schmidt_summerer_2013} of Schmidt and Summerer involving Lemma \ref{l:empty_cylinder}. We omit certain details for simplicity. They can be found in \cite{german_moshchevitin_2022}.

Consider an arbitrary nonzero point $\vec v\in\Z^{n+1}$ and
set $t=t(\vec v)$ to be the smallest positive real number such that the cylinder
\begin{equation}\label{eq:C_v_definition}
  \cC_{\vec v}=\bigg\{\, \vec u\in\R^{n+1}\ \bigg|\ r(\vec u)\leq t,\ h(\vec u)\leq t\cdot\frac{r(\vec v)}{h(\vec v)} \,\bigg\}
\end{equation}
contains a nonzero point of $\Z^{n+1}$. Geometrical meaning of the inequalities determining $\cC_{\vec v}$ is that each of the two bases of this cylinder has exactly one common point with the plane $\langle\vec v,\vec u\rangle=0$. Define also $\gamma=\gamma(\vec v)$, $\alpha=\alpha(\vec v)$, and $\beta=\beta(\vec v)$ by
\begin{equation}\label{eq:gamma_alpha_beta_definition}
  r(\vec v)=h(\vec v)^{-\gamma},\qquad
  h(\vec v)=t^\alpha,\qquad
  \alpha=\frac{1+\beta}{1+\gamma}\,.
\end{equation}
It can be easily verified that with such a choice of parameters $\vec v$ satisfies \eqref{eq:empty_cylinder_v} and $\cC_{\vec v}$ satisfies
\[
  \cC_{\vec v}=\Big\{\, \vec u\in\R^{n+1}\ \Big|\ r(\vec u)\leq t,\ h(\vec u)\leq t^{-\beta} \,\Big\}.
\]
Hence, if the condition
\begin{equation}\label{eq:t_alpha_beta_condition}
  t^{-\alpha}-t^{-\beta}\geq t^{-\beta}
\end{equation}
is fulfilled, $\cC_{\vec v}$ can be supplemented with nonempty cylinders $\cC$ and $-\cC$ determined by \eqref{eq:empty_cylinder_C}, which do not contain integer points by Lemma \ref{l:empty_cylinder} (see Fig. \ref{fig:scmidt_summerer}).

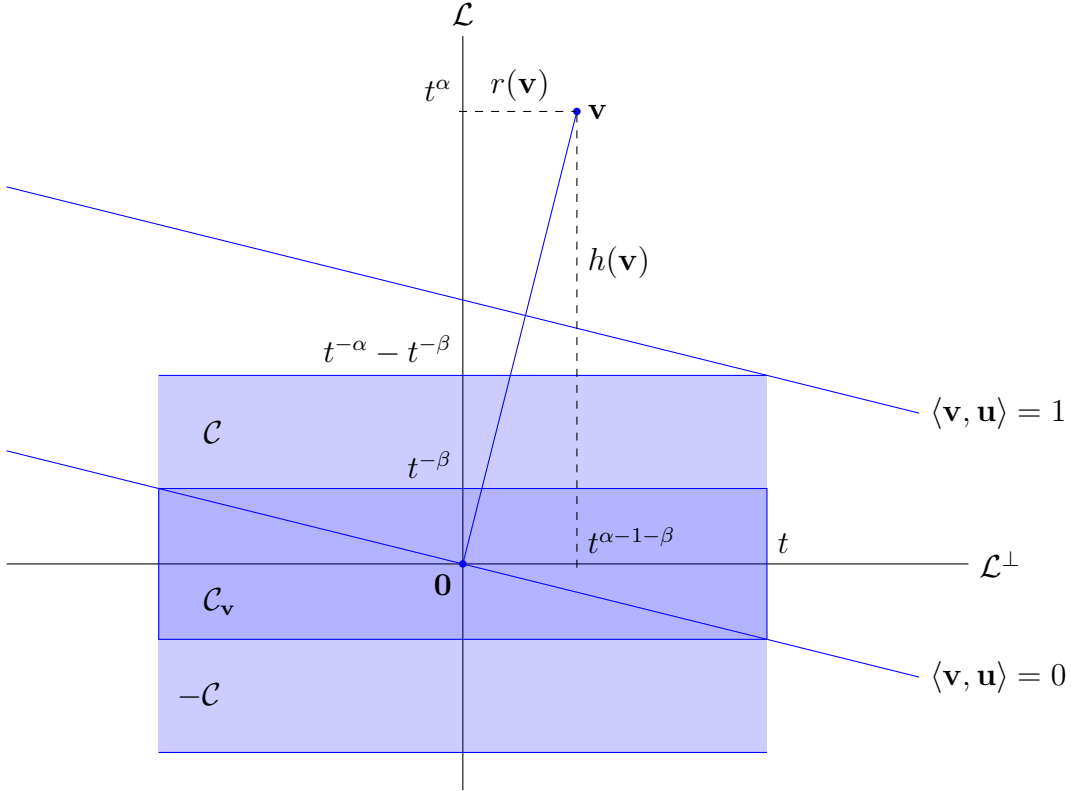
\begin{figure}[h]
\centering
\begin{tikzpicture}[scale=0.5]
  \draw (-12,0) -- (13.3,0) node[right] {$\cL^\perp$};
  \draw (0,-6) -- (0,14) node[above] {$\cL$};

  \draw[blue] (0,0) -- (3,12);
  \node[fill=blue,circle,inner sep=1pt] at (3,12) {};
  \draw (3,12) node[right]{$\vec v$};

  \draw (1.5,12) node[above]{$r(\vec v)$};
  \draw (3,8) node[right]{$h(\vec v)$};

  \draw[color=blue] plot[domain=-12:12] (\x, {-\x/4}) node[right, color=black]{$\langle\vec v,\vec u\rangle=0$};
  \draw[color=blue] plot[domain=-12:12] (\x, {-\x/4+7}) node[right, color=black]{$\langle\vec v,\vec u\rangle=1$};

  \fill[blue,opacity=0.2] (8,2) -- (8,5) -- (-8,5) -- (-8,2) -- cycle;
  \fill[blue,opacity=0.2] (8,-2) -- (8,-5) -- (-8,-5) -- (-8,-2) -- cycle;
  \fill[blue,opacity=0.3] (-8,2) -- (8,2) -- (8,-2) -- (-8,-2) -- cycle;
  \draw[blue] (-8,5) -- (8,5);
  \draw[blue] (-8,-5) -- (8,-5);
  \draw[blue] (-8,2) -- (8,2) -- (8,-2) -- (-8,-2) -- cycle;

  \draw[dashed] (-0.1,12) -- (3,12);
  \draw[dashed] (3,-0.1) -- (3,12);

  \draw (8,0) node[above right]{$t$};
  \draw (3,0) node[above right]{$t^{\alpha-1-\beta}$};
  \draw (0,12) node[above left]{$t^{\alpha}$};
  \draw (0,5) node[above left]{$t^{-\alpha}-t^{-\beta}$};
  \draw (0,2) node[above left]{$t^{-\beta}$};

  \draw (-7.1,3.5) node[right]{$\cC$};
  \draw (-7.1,-1) node[right]{$\cC_{\vec v}$};
  \draw (-7.8,-3.5) node[right]{$-\cC$};


  \node[fill=blue,circle,inner sep=1pt] at (0,0) {};
  \draw (0,0) node[below left]{$\vec 0$};

\end{tikzpicture}
\caption{Illustration for the proof of Schmidt--Summerer's inequalities} \label{fig:scmidt_summerer}
\end{figure}

Thus, to each point $\vec v$ satisfying \eqref{eq:t_alpha_beta_condition}, we have assigned the cylinder $\cC_{\vec v}$, which contains integer points in its boundary, and the cylinder
\[
  \cC'_{\vec v}=\cC_{\vec v}\cup\cC\cup(-\cC),
\]
which does not contain nonzero integer points in its interior. Furthermore, the condition \eqref{eq:t_alpha_beta_condition} is guaranteed to be fulfilled if $h(\vec v_k)r(\vec v_k)^n$ is small enough. The reason is that the lattice which is the orthogonal projection of $\Z^{n+1}\cap(\R\vec v)^{\perp}$ onto $\cL^\perp$ has determinant equal to $h(\vec v_k)$. It follows by Minkowski's convex body theorem that $t^n\leq2^nh(\vec v_k)/B$, where $B$ equals the volume of the $n$-dimensional Euclidean ball of radius $1$. Thus, if $h(\vec v_k)r(\vec v_k)^n\leq4^{-n}B$, then $t^n\leq2^{-n}r(\vec v_k)^{-n}=(2t^{\alpha-1-\beta})^{-n}$ holds, i.e. $t^{\beta-\alpha}\geq2$, which is equivalent to \eqref{eq:t_alpha_beta_condition}. Geometrically, this means that the cylinder $\cC_{\vec v}$ does not ``reach'' the plane $\langle\vec v,\vec u\rangle=1$.

In order to prove the left-hand inequality \eqref{eq:schmidt_summerer_2013}, let us consider a sequence of points $\vec v_k$ such that
\[
  h(\vec v_k)\to\infty
  \qquad\text{ and }\qquad
  \gamma_k=\gamma(\vec v_k)\to\omega(\Theta)
  \qquad\text{ as }\qquad
  k\to\infty.
\]
Then, if $\omega(\Theta)>1/n$, the quantity $h(\vec v_k)r(\vec v_k)^n$ is small for large $k$, i.e. the condition \eqref{eq:t_alpha_beta_condition} is fulfilled. Set $\alpha_k=\alpha(\vec v)$ and $\beta_k=\beta(\vec v)$. Then
\[
  \omega(L_\Theta)\geq\limsup_{k\to\infty}\beta_k\,,
  \qquad
  \hat\omega(L_\Theta)\leq\liminf_{k\to\infty}\alpha_k\,.
\]
Taking into account \eqref{eq:gamma_alpha_beta_definition}, we get
\[
  \hat\omega(L_\Theta)\leq
  \liminf_{k\to\infty}\frac{1+\beta_k}{1+\gamma_k}\leq
  \frac{1+\limsup_{k\to\infty}\beta_k}{1+\lim_{k\to\infty}\gamma_k}\leq
  \frac{1+\omega(L_\Theta)}{1+\omega(\Theta)}\,.
\]
The left-hand inequality \eqref{eq:schmidt_summerer_2013} is proved.

In order to prove the right-hand inequality \eqref{eq:schmidt_summerer_2013}, let us consider a sequence of points $\vec v_k$ such that
\[
  h(\vec v_k)\to0
  \qquad\text{ and }\qquad
  \gamma_k=\gamma(\vec v_k)\to\omega(L_\Theta)^{-1}
  \qquad\text{ as }\qquad
  k\to\infty.
\]
Then, if $\omega(L_\Theta)>n$, the quantity $h(\vec v_k)r(\vec v_k)^n$ is small for large $k$, i.e. the condition \eqref{eq:t_alpha_beta_condition} is again fulfilled. Set  $\alpha_k=\alpha(\vec v)$ and $\beta_k=\beta(\vec v)$. Then
\[
  \omega(\Theta)\geq\limsup_{k\to\infty}\beta_k^{-1}\,,
  \qquad
  \hat\omega(\Theta)\leq\liminf_{k\to\infty}\alpha_k^{-1}\,.
\]
Taking into account \eqref{eq:gamma_alpha_beta_definition}, we get
\[
  \hat\omega(\Theta)\leq
  \liminf_{k\to\infty}\frac{1+\gamma_k}{1+\beta_k}\leq
  \frac{1+\lim_{k\to\infty}\gamma_k}{1+\liminf_{k\to\infty}\beta_k}\leq
  \frac{1+\omega(L_\Theta)^{-1}}{1+\omega(\Theta)^{-1}}\,.
\]
The right-hand inequality \eqref{eq:schmidt_summerer_2013} is also proved.

\subsubsection{Parametric geometry of numbers}\label{subsubsec:parametric_geometry_of_numbers}

Developing the ideas proposed by Schmidt in his paper \cite{schmidt_luminy_1982}, Schmidt and Summerer published a fundamental paper \cite{schmidt_summerer_2009} in 2009, where they proposed a new point of view at the problems under discussion. They called their approach ``parametric geometry of numbers''.

Most generally parametric geometry of numbers can be outlined as follows. Let $\La$ be a full rank lattice in $\R^d$ with determinant $1$.
Consider the cube $\cB=[-1,1]^d$. Define the  \emph{space of parameters} by
\[
  \cT=\Big\{ \pmb\tau=(\tau_1,\ldots,\tau_d)\in\R^d\, \Big|\,\tau_1+\ldots+\tau_d=0 \Big\}.
\]
For each $\pmb\tau\in\cT$ set
\[
  \cB_{\pmb\tau}=
  \textup{diag}(e^{\tau_1},\ldots,e^{\tau_d})
  \cB.
\]
Let $\mu_k(\cB_{\pmb\tau},\La)$, $k=1,\ldots,d$, denote the $k$-th successive minimum of $\cB_{\pmb\tau}$ w.r.t. $\La$, i.e. the smallest positive $\mu$ such that  $\mu\cB_{\pmb\tau}$ contains not less than $k$ linearly independent vectors of $\La$. Finally, for each $k=1,\ldots,d$, let us define the functions
\[
  L_k(\pmb\tau)=L_k(\La,\pmb\tau)=\log\big(\mu_k(\cB_{\pmb\tau},\La)\big),\qquad
  S_k(\pmb\tau)=S_k(\La,\pmb\tau)=\sum_{1\leq j\leq k}L_j(\La,\pmb\tau).
\]

Many problems in Diophantine approximation can be formulated as questions concerning the asymptotic behaviour of $L_k(\pmb\tau)$ and $S_k(\pmb\tau)$. Each problem requires a proper choice of $\La$ and of a subset of $\cT$, with respect to which the asymptotics of these functions is to be studied. In the case of the problems we discuss in this Section $\pmb\tau$ is supposed to tend to infinity along certain one-dimensional subspaces of $\cT$. Let $\gT$ be the path in $\cT$ determined by a mapping $s\mapsto\pmb\tau(s)$, $s\in[0,\infty)$. For our purpose, we may assume that this mapping is linear.

\begin{definition} \label{def:schmidt_psi}
  Suppose a lattice $\La$ and a path $\gT$ are given. Let $k\in\{1,\ldots,d\}$. The quantities
  \[
    \bpsi_k(\La,\gT)=
    \liminf_{s\to+\infty}
    \frac{L_k(\La,\pmb\tau(s))}{s}
    \qquad\text{ and }\qquad
    \apsi_k(\La,\gT)=
    \limsup_{s\to+\infty}
    \frac{L_k(\La,\pmb\tau(s))}{s}
  \]
  are respectively called the \emph{$k$-th lower} and \emph{upper Schmidt--Summerer exponents of the first type}.
\end{definition}

\begin{definition} \label{def:schmidt_Psi}
  Suppose a lattice $\La$ and a path $\gT$ are given. Let $k\in\{1,\ldots,d\}$. The quantities
  \[
    \bPsi_k(\La,\gT)=
    \liminf_{s\to+\infty}
    \frac{S_k(\La,\pmb\tau(s))}{s}
    \qquad\text{ and }\qquad
    \aPsi_k(\La,\gT)=
    \limsup_{s\to+\infty}
    \frac{S_k(\La,\pmb\tau(s))}{s}
  \]
  are respectively called the \emph{$k$-th lower} and \emph{upper Schmidt--Summerer exponents of the second type}.
\end{definition}


For the problem of simultaneous approximation and the problem of approximating zero with the values of a linear form, the lattices and the paths can be chosen as follows. Set $d=n+1$.

As in Section \ref{subsubsec:parallelepiped_families}, consider the lattices
\begin{equation}\label{eq:lattices_sim_and_lin}
  \La=
  \left(
  \begin{matrix}
        1     &    0   &    0   & \cdots &    0   \\
    -\theta_1 &    1   &    0   & \cdots &    0   \\
    -\theta_2 &    0   &    1   & \cdots &    0   \\
     \phantom{-}
     \vdots   & \vdots & \vdots & \ddots & \vdots \\
    -\theta_n &    0   &    0   & \cdots &    1
  \end{matrix}
  \right)
  \Z^{n+1},
  \qquad
  \La^\ast=
  \left(
  \begin{matrix}
       1   & \theta_1 & \theta_2 & \cdots & \theta_n \\
       0   &    1     &    0     & \cdots &    0     \\
       0   &    0     &    1     & \cdots &    0     \\
    \vdots & \vdots   & \vdots   & \ddots & \vdots   \\
       0   &    0     &    0     & \cdots &    1
  \end{matrix}
  \right)
  \Z^{n+1}.
\end{equation}
Define the paths $\gT$ and $\gT^\ast$ respectively by the mappings
\begin{equation}\label{eq:path_for_sim}
\begin{aligned}
  & s\mapsto\pmb\tau(s)=\big(\tau_1(s),\ldots,\tau_{n+1}(s)\big), \\
  & \tau_1(s)=s,\quad\tau_2(s)=\ldots=\tau_{n+1}(s)=-s/n
\end{aligned}
\end{equation}
and
\begin{equation*}
\begin{aligned}
  & s\mapsto\pmb\tau^\ast(s)=\big(\tau^\ast_1(s),\ldots,\tau^\ast_{n+1}(s)\big), \\
  & \tau^\ast_1(s)=-ns,\quad\tau^\ast_2(s)=\ldots=\tau^\ast_{n+1}(s)=s.
\end{aligned}
\end{equation*}
Then $\pmb\tau^\ast(s)=-n\pmb\tau(s)$, i.e. $\gT\cup\gT^\ast$ is a one-dimensional subspace of $\cT$. As we noted in Section \ref{subsubsec:parallelepiped_families}, $\La^\ast$ is the dual of $\La$.

Regular and uniform Diophantine exponents are related to Schmidt--Summerer exponents by the equalities
\begin{equation}\label{eq:belpha_via_psis}
\begin{aligned}
  & \big(1+\omega(\Theta)\big)\big(1+\bpsi_1(\La,\gT)\big)=\big(1+\hat\omega(\Theta)\big)\big(1+\apsi_1(\La,\gT)\big)=(n+1)/n, \\
  & \big(1+\omega(L_\Theta)\big)\big(1+\bpsi_1(\La^\ast,\gT^\ast)\big)=\big(1+\hat\omega(L_\Theta)\big)\big(1+\apsi_1(\La^\ast,\gT^\ast)\big)=n+1.
\end{aligned}
\end{equation}
They are deduced directly from the definitions (see \cite{schmidt_summerer_2009}, \cite{german_AA_2012}). We note that analogous relations for Schmidt--Summerer exponents with indices greater than $1$ hold in the context of problems of approximating a given subspace by rational subspaces of fixed dimension. In those problems so called \emph{intermediate} exponents arise. They are studied in papers \cite{schmidt_annals_1967}, \cite{laurent_up_down}, \cite{bugeaud_laurent_up_down}, \cite{german_AA_2012}, the latter of which contains the equalities relating Schmidt--Summerer exponents and intermediate exponents.

By \eqref{eq:belpha_via_psis} each relation proved for Schmidt--Summerer exponents generates a relation for regular and uniform Diophantine exponents. As for relations for Schmidt--Summerer exponents, they are obtained by analysing the properties of the functions $L_k(\pmb\tau)$ and $S_k(\pmb\tau)$. The first observation that can be made is that $L_k(\pmb\tau)$ and $S_k(\pmb\tau)$ are continuous and piece-wise linear on $\cT$ and their restrictions to the paths $\gT$ and $\gT^\ast$ have two possible slopes on linearity intervals. Next, a very important role in studying the behaviour of $L_k(\pmb\tau)$ and $S_k(\pmb\tau)$ is played by two classical theorems on successive minima -- theorem of Minkowski \cite{minkowski} and theorem of Mahler \cite{mahler_casopis_convex}. Let us state these theorems for parallelepipeds $\cB_{\pmb\tau}$ and lattices $\La$, $\La^\ast$ (assuming that $d=n+1$).

\begin{theorem}[H.\,Minkowski, 1896] \label{t:minkowski_minima}
  For each $\pmb\tau\in\cT$ we have
  \begin{equation} \label{eq:minkowski_minima}
    \frac1{d!}\leq\prod_{i=1}^{d}\mu_i(\cB_{\pmb\tau},\La)\leq1.
  \end{equation}
\end{theorem}

\begin{theorem}[K.\,Mahler, 1938]\label{t:mahler_koerper}
  For each $k=1,\ldots,d$ and each $\pmb\tau\in\cT$ we have
  \begin{equation} \label{eq:mahler_koerper}
    \frac1d\leq\mu_k(\cB_{\pmb\tau},\La)\mu_{d+1-k}(\cB_{-\pmb\tau},\La^\ast)\leq d!\,.
  \end{equation}
\end{theorem}

The following local properties holding for every $\pmb\tau\in\cT$ are easily deduced from Theorems \ref{t:minkowski_minima}, \ref{t:mahler_koerper} (see \cite{german_monatshefte_2022}):
\begin{itemize}
  \item[(i)]

    $L_k(\La,\pmb\tau)=-L_{d+1-k}(\La^\ast,-\pmb\tau)+O(1),\ \ k=1,\ldots,d$; \vphantom{$\frac{\big|}{|}$}

  \item[(ii)]

    $S_k(\La,\pmb\tau)=S_{d-k}(\La^\ast,-\pmb\tau)+O(1),\ \ k=1,\ldots,d-1$; \vphantom{$\frac{\big|}{|}$}

  \item[(iii)]

    $\displaystyle
    S_1(\La,\pmb\tau)\leq\ldots\leq
    \frac{S_k(\La,\pmb\tau)}{k}\leq\ldots\leq
    \frac{S_{d-1}(\La,\pmb\tau)}{d-1}\leq
    \frac{S_{d}(\La,\pmb\tau)}{d}=O(1)$;

  \item[(iv)]

    $\displaystyle
    \frac{S_1(\La,\pmb\tau)}{d-1}\geq\ldots\geq
    \frac{S_k(\La,\pmb\tau)}{d-k}\geq\ldots\geq
    S_{d-1}(\La,\pmb\tau)$.
\end{itemize}
The constants implied by $O(\cdot)$ depend only on $d$. These properties, as well as Theorems \ref{t:minkowski_minima} and \ref{t:mahler_koerper}, are valid for every unimodular lattice. Particularly, they remain valid if $\La$ is replaced with $\La^\ast$.

Properties (ii) and (iii) immediately imply the inequality
\begin{equation}\label{eq:essence_of_transference}
  S_1(\La,\pmb\tau)\leq
  \dfrac{S_1(\La^\ast,-\pmb\tau)}{d-1}+O(1),
\end{equation}
whence we get (by applying \eqref{eq:essence_of_transference} to $\La$ and $\La^\ast$) the following inequalities for Schmidt--Summerer exponents corresponding to the problem of simultaneous approximation and the problem of approximating zero with the values of a linear form:
\begin{equation}\label{eq:essence_of_transference_in_exponents}
  \bpsi_1(\La^\ast,\gT^\ast)
  \leq
  \bpsi_1(\La,\gT)
  \leq
  \dfrac{\bpsi_1(\La^\ast,\gT^\ast)}{n^2}\,.
\end{equation}
Rewriting \eqref{eq:essence_of_transference_in_exponents} in terms of $\omega(\Theta)$, $\omega(L_\Theta)$, with the help of \eqref{eq:belpha_via_psis}, we get Khintchine's inequalities \eqref{eq:khintchine_transference}. It is shown in \cite{german_AA_2012} that this way of proving Khintchine's inequalities enables to improve them by splitting each of them into a chain of inequalities between the aforementioned intermediate exponents. The first to obtain such a chain of inequalities were Laurent and Bugeaud (see \cite{laurent_up_down}, \cite{bugeaud_laurent_up_down}).

Restrictions of $L_1(\pmb\tau),\ldots,L_d(\pmb\tau)$ to the path $\gT$ enjoy a series of elementary properties: they are continuous, piece-wise linear, they have two possible slopes on linearity intervals, at each point their values are ordered as $L_1(\pmb\tau)\leq\ldots\leq L_d(\pmb\tau)$, their sum $S_d(\pmb\tau)=L_1(\pmb\tau)+\ldots+L_d(\pmb\tau)$ is nonpositive and bounded. In all other respects the behaviour of $L_1,\ldots,L_d$ is rather chaotic. Nevertheless, there is a class of $d$-tuples of functions having quite regular behaviour, enjoying the mentioned properties of $L_1,\ldots,L_d$, and approximating $L_1,\ldots,L_d$ up to bounded functions. This outstanding result belongs to Roy \cite{roy_annals_2015}. Its formulation requires the concept of a \emph{$d$--system}. We adapt the respective definition from \cite{roy_zeitschrift_2016} to 
the path $\gT$ defined by \eqref{eq:path_for_sim} (we remind that this path corresponds to the problem of simultaneous approximation). As before, we assume that $d=n+1$.

\begin{definition}
  Let $I\subset[0,\infty)$ be an interval with nonempty interior. A continuous piece-wise linear mapping $\vec P=(P_1,\ldots,P_d):I\to\R^d$ is called a \emph{$d$--system} on $I$ if the following conditions are fulfilled:
  \begin{itemize}
    \item[(i)]

      for each $s\in I$ we have $P_1(s)\leq\ldots\leq P_d(s)$ and $P_1(s)+\ldots+P_d(s)=0$;

    \item[(ii)]

      for each open interval $J\subset I$ on which $\vec P$ is differentiable, there is an index $i$ such that the slope of $P_i$ equals $-1$ on $J$, while the slope of every $P_j$ with $j\neq i$ equals $1/(d-1)$;

    \item[(iii)]

      for every point $s$ in the interior of $I$ at which $\vec P$ is not differentiable and all indices $i,j$ such that $i<j$ and $P'_i(s-0)=P'_j(s+0)=-1$, we have $P_i(s)=P_{i+1}(s)=\ldots=P_j(s)$.

  \end{itemize}
\end{definition}

Assuming that the path $\gT$ and the lattice $\La=\La(\Theta)$ are defined by \eqref{eq:path_for_sim} and \eqref{eq:lattices_sim_and_lin}, let us define the mapping $\vec L_\Theta:[0,\infty)\to\R^d$ by
\[
  s\mapsto\Big(L_1\big(\La,\pmb\tau(s)\big),\ldots,L_d\big(\La,\pmb\tau(s)\big)\Big).
\]

\begin{theorem}[D.\,Roy, 2015]\label{t:roy_annals_2015}
  For each nonzero $\Theta\in\R^n$ there exist $s_0\geq0$ and a $d$--system $\vec P$ on $[s_0,\infty)$ such that $\vec L_\Theta-\vec P$ is bounded on $[s_0,\infty)$. Conversely, for each $d$--system $\vec P$ on $[s_0,\infty)$ with arbitrary $s_0\geq0$ there exists a nonzero $\Theta\in\R^n$ such that $\vec L_\Theta-\vec P$ is bounded on $[s_0,\infty)$.
\end{theorem}

Roy's theorem proved to be much in demand for proving the existence of $\Theta$ with prescribed Diophantine properties. It was Theorem \ref{t:roy_annals_2015} that enabled proving sharpness of many transference inequalities, and also Theorem \ref{t:marnat_moshchevitin_2020}.

\section{Several linear forms}\label{sec:several_linear_forms}

\subsection{General setting for the main problem of homogeneous linear Diophantine approximation}

Given a matrix
\[ \Theta=
   \begin{pmatrix}
     \theta_{11} & \cdots & \theta_{1m} \\
     \vdots & \ddots & \vdots \\
     \theta_{n1} & \cdots & \theta_{nm}
   \end{pmatrix},\qquad
   \theta_{ij}\in\R,\quad m+n\geq3, \]
consider the system of linear equations
\begin{equation*} 
  \Theta\vec x=\vec y
\end{equation*}
with variables $\vec x=(x_1,\ldots,x_m)\in\R^m$, $\vec y=(y_1,\ldots,y_n)\in\R^n$.

In the most general form, the main question of homogeneous linear Diophantine approximation is how small the quantity $|\Theta\vec x-\vec y|$ can be with $\vec x\in\Z^m$, $\vec y\in\Z^n$ satisfying the restriction $0<|\vec x|\leq t$. By $|\cdot|$ we denote hereafter the sup-norm.

It is easy to see that with $m=1$ we get the problem of simultaneous approximation, and with $n=1$ we get the problem of approximating zero with the values of a linear form. Same as in those problems, the respective Diophantine exponents are naturally defined. The following definition contains as particular cases Definitions \ref{def:beta_simultaneous}, \ref{def:beta_linear_form}, \ref{def:alpha_simultaneous}, \ref{def:alpha_linear_form}.

\begin{definition} \label{def:belpha}
  The supremum of real $\gamma$ satisfying the condition that there exist arbitrarily large $t$ such that (resp. for every $t$ large enough) the system 
  \begin{equation} \label{eq:belpha_1_definition}
    |\vec x|\leq t,\qquad|\Theta\vec x-\vec y|\leq t^{-\gamma}
  \end{equation}
  admits a nonzero solution $(\vec x,\vec y)\in\Z^m\oplus\Z^n$ is called the \emph{regular} (resp. \emph{uniform}) \emph{Diophantine exponent} of $\Theta$ and is denoted by $\omega(\Theta)$ (resp. $\hat\omega(\Theta)$).
\end{definition}

\subsection{Dirichlet's theorem again}

In the case of several linear forms an analogue of Theorems \ref{t:dirichlet}, \ref{t:dirichlet_linear_form}, \ref{t:dirichlet_simultaneous} holds. It was also proved in Dirichlet's paper \cite{dirichlet}, though, same as with these theorems, he originally formulated it in a weaker form. We remind that $|\cdot|$ denotes the sup-norm. Let us also denote by $\R^{n\times m}$ the set of $n\times m$ real matrices.

\begin{theorem}[G. Lejeune Dirichlet, 1842] \label{t:dirichlet_matrix}
  Let $\Theta\in\R^{n\times m}$. Then, for each $t\geq1$, there is a nonzero pair $(\vec x,\vec y)\in\Z^m\oplus\Z^n$ such that
  \begin{equation} \label{eq:dirichlet_matrix}
    |\vec x|\leq t,\qquad|\Theta\vec x-\vec y|\leq t^{-m/n}.
  \end{equation}
\end{theorem}

We note again that Theorem \ref{t:dirichlet_matrix} follows immediately from Minkowski's convex body theorem in the form of Corollary \ref{cor:minkowski_linear_forms} applied to the system \eqref{eq:dirichlet_matrix}.

\begin{corollary} \label{cor:belpha_matrix_trivial_inequalities}
  For each $\Theta\in\R^{n\times m}$ we have
  \begin{equation} \label{eq:belpha_matrix_trivial_inequalities}
    \omega(\Theta)\geq\hat\omega(\Theta)\geq m/n.
  \end{equation}
\end{corollary}

These inequalities, same as inequalities \eqref{eq:belpha_trivial_inequalities}, are sharp. Moreover, Theorem \ref{t:perron} of Perron can be generalised to the matrix setting. Perron's idea can be used to prove (for instance, this is done in Schmidt's book \cite{schmidt_DA}) that there exist matrices with algebraic entries that are \emph{badly approximable}.

\begin{definition}\label{def:BA_matrix}
  A matrix $\Theta$ is called \emph{badly approximable} if there is a constant $c>0$ depending only on $\Theta$ such that for every pair $(\vec x,\vec y)\in\Z^m\oplus\Z^n$ with nonzero $\vec x$ we have
  \[ 
    |\Theta\vec x-\vec y|^n|\vec x|^m\geqslant c. 
  \]
\end{definition}

In that same Schmidt's book \cite{schmidt_DA} it is proved that a matrix is badly approximable if and only if so is its transpose. In Section \ref{subsubsec:badly_approximable_matrices} we show how to deduce this fact immediately from Mahler's theorem. This fact is a manifestation of the \emph{transference principle}, which we already paid attention to in Section \ref{subsec:transference_only_n}.

\subsection{Transference theorems}

In Section \ref{subsec:transference_only_n} we discussed the \emph{transference principle} discovered by Khintchine, which relates the problem of simultaneous approximation and the problem of approximating zero by the values of a linear form. This relation can be generalised to the case of an arbitrary matrix $\Theta\in\R^{n\times m}$. Let us denote by $\tr\Theta$ the transposed matrix.

\subsubsection{Inequalities for regular exponents}

In 1947 Dyson \cite{dyson} generalised Khintchine's transference principle (Theorem \ref{t:khintchine_transference}) as follows.

\begin{theorem}[F.\,Dyson, 1947] \label{t:dyson_transference}
  For every matrix $\Theta\in\R^{n\times m}$ we have
  \begin{equation} \label{eq:dyson_transference}
    \omega(\tr\Theta)\geq\frac{n\omega(\Theta)+n-1}{(m-1)\omega(\Theta)+m}\,.
  \end{equation}
\end{theorem}


Note that it follows from \ref{eq:dyson_transference} and \ref{eq:belpha_matrix_trivial_inequalities} that
\begin{equation} \label{eq:dyson_n_m_equivalence}
  \omega(\Theta)=m/n\iff\omega(\tr\Theta)=n/m.
\end{equation}

A year later, Khintchine published in \cite{khintchine_dobav} a simpler proof of Theorem \ref{t:dyson_transference}. But actually, Theorem \ref{t:dyson_transference} could have been derived in 1937 as a corollary to Theorem \ref{t:mahler} of Mahler. We discuss this in more detail in Section \ref{subsec:embedding_into_R_d}.

Note that any statement proved for \emph{an arbitrary matrix $\Theta\in\R^{n\times m}$ with arbitrary $n$, $m$} automatically generates an analogous statement for $\tr\Theta$. We simply need to swap the triple $(n,m,\Theta)$ for the triple $(m,n,\tr\Theta)$. Therefore, for an arbitrary matrix $\Theta\in\R^{n\times m}$, we also have
\begin{equation} \label{eq:dyson_transference_swapped}
  \omega(\Theta)\geq\frac{m\omega(\tr\Theta)+m-1}{(n-1)\omega(\tr\Theta)+n}\,.
\end{equation}

As we mentioned in Section \ref{subsubsec:transference_only_n_regular}, in the case when either $m=1$, or $n=1$, sharpness of \eqref{eq:dyson_transference} was proved by Jarn\'{\i}k in \cite{jarnik_1936_1}, \cite{jarnik_1936_2}. For $\min(n,m)>1$ sharpness of \eqref{eq:dyson_transference} is proved only if
\[
  \omega(\Theta)\geq\max\Big(\frac mn,\frac{n-1}{m-1}\Big)
\]
(particularly, if $m\geq n$). This result also belongs to Jarn\'{\i}k \cite{jarnik_bulg_1959}. In all the remaining cases sharpness of \eqref{eq:dyson_transference} remains unproved.

\subsubsection{Inequalities for uniform exponents}

Jarn\'{\i}k's transference inequalities, i.e. inequalities \eqref{eq:khintchine_transference_alpha}, \eqref{eq:jarnik_inequalities_cases}, were generalised in 1951 by Apfelbeck \cite{apfelbeck} to the case of arbitrary $n$, $m$. He proved the ``uniform'' analogue of \eqref{eq:dyson_transference}
\begin{equation} \label{eq:apfel_dyson}
  \hat\omega(\tr\Theta)\geq\frac{n\hat\omega(\Theta)+n-1}{(m-1)\hat\omega(\Theta)+m}
\end{equation}
and also the inequality
\begin{equation} \label{eq:apfelbeck}
  \hat\omega(\tr\Theta)\geq\frac 1m
  \left(n+\frac{n(n\hat\omega(\Theta)-m)-2n(m+n-3)}{(m-1)(n\hat\omega(\Theta)-m)+m-(m-2)(m+n-3)}\right),
\end{equation}
under the assumption that $m>1$, $\hat\omega(\Theta)>(2(m+n-1)(m+n-3)+m)/n$. In 2012 inequalities \eqref{eq:apfel_dyson}, \eqref{eq:apfelbeck} were improved by the author in papers \cite{german_MJCNT_2012}, \cite{german_AA_2012}. The following theorem generalises Theorem \ref{t:my_simultaneous_vs_linear_form} to the case of arbitrary $n$, $m$.

\begin{theorem}[O.\,G., 2012] \label{t:my_inequalities}
  For every $\Theta\in\R^{n\times m}$, $m+n\geq3$, we have
  \begin{equation} \label{eq:my_inequalities}
    \hat\omega(\tr\Theta)\geq
    \begin{cases}
      \dfrac{n-1}{m-\hat\omega(\Theta)}\quad\ \ \text{ if }\ \hat\omega(\Theta)\leq1, \\
      \dfrac{n-\hat\omega(\Theta)^{-1\vphantom{\big|}}}{m-1}\quad\text{ if }\ \hat\omega(\Theta)\geq1.
    \end{cases}
  \end{equation}
\end{theorem}

Note that, generally, $\hat\omega(\Theta)$ and $\hat\omega(\tr\Theta)$ may equal $+\infty$, and this gives sense to \eqref{eq:my_inequalities} when some of the denominators equals zero.

If $\min(n,m)>1$, the question whether \eqref{eq:my_inequalities} is sharp remains open. We remind that in the case when either $m=1$ or $n=1$ sharpness of \eqref{eq:my_inequalities} was proved by Marnat \cite{marnat_sharpness} and (independently) by Schmidt and Summerer \cite{schmidt_summerer_AA_2016}.

\subsubsection{``Mixed'' inequalities}

In those same papers \cite{german_MJCNT_2012}, \cite{german_AA_2012} the following theorem is proved. It generalises Theorem \ref{t:bugeaud_laurent} of Laurent and Bugeaud and refines Theorem \ref{t:dyson_transference} of Dyson.

\begin{theorem}[O.\,G., 2012] \label{t:loranoyadenie}
  Given $\Theta\in\R^{n\times m}$, $m+n\geq3$, assume that the space of rational solutions of the equation
  \[ 
    \Theta\vec x=\vec y 
  \]
  is not one-dimensional. Then we have
  \begin{align}
    & \omega(\tr\Theta)\geq
    \frac{n\omega(\Theta)+n-1}{(m-1)\omega(\Theta)+m}\,,
    \label{eq:loranoyadenie_1} \\
    & \omega(\tr\Theta)\geq
    \frac{(n-1)(1+\omega(\Theta))-(1-\hat\omega(\Theta))}{(m-1)(1+\omega(\Theta))+(1-\hat\omega(\Theta))}\,, \label{eq:loranoyadenie_2} \\
    & \omega(\tr\Theta)\geq
    \frac{(n-1)(1+\omega(\Theta)^{-1})-(\hat\omega(\Theta)^{-1}-1)}{(m-1)(1+\omega(\Theta)^{-1})+(\hat\omega(\Theta)^{-1}-1)}
    \,. \label{eq:loranoyadenie_3}
  \end{align}
\end{theorem}

Clearly, \eqref{eq:loranoyadenie_1} coincides with \eqref{eq:dyson_transference}. This inequality is stronger than \eqref{eq:loranoyadenie_2} and \eqref{eq:loranoyadenie_3} if and only if
\[
  \hat\omega(\Theta)<\min\left(\frac{(m-1)\omega(\Theta)+m}{m+n-1}\,,\,
  \frac{(m+n-1)\omega(\Theta)}{(n-1)+n\omega(\Theta)}\right).
\]
For instance, if $\hat\omega(\Theta)<(m+n-1)/n$ and $\omega(\Theta)$ is large enough. Thus, if $\min(n,m)>1$, inequalities \eqref{eq:loranoyadenie_2}, \eqref{eq:loranoyadenie_3} are not guaranteed to improve upon Dyson's inequality. Their sharpness is also doubtful, for, as we noted in Section \ref{subsubsec:transference_only_n_mixed}, in the cases $m=1$, $n\geq3$ and $n=1$, $m\geq3$ they are not sharp.

We note also that, up to now, no analogue of Schmidt--Summerer's inequality \eqref{eq:schmidt_summerer_2013} for the case when $\min(n,m)>1$ is known. It would be very interesting to find such an analogue that, combined with \eqref{eq:my_inequalities}, it would give \eqref{eq:loranoyadenie_2} and \eqref{eq:loranoyadenie_3} -- same as \eqref{eq:schmidt_summerer_2013} combined with \eqref{eq:my_simultaneous_vs_linear_form} gives \eqref{eq:bugeaud_laurent}.

\subsubsection{Inequalities between regular and uniform exponents}

In paper \cite{jarnik_czech_1954} Jarn\'{\i}k obtained the following result.

\begin{theorem}[V.\,Jarn\'{\i}k, 1954] \label{t:jarnik_czech_1954}
  Given $\Theta\in\R^{n\times m}$, assume that the equation
  \[
    \Theta\vec x=\vec y
  \]
  admits no nonzero integer solutions. Then

  (i) for $m=2$ we have
  \begin{equation}\label{eq:jarnik_czech_1954_m=2}
    \frac{\omega(\Theta)}{\hat\omega(\Theta)}\geq\hat\omega(\Theta)-1;
  \end{equation}

  (ii) for $m\geq3$ and $\hat\omega\geq(5m^2)^{m-1}$ we have
  \[
    \frac{\omega(\Theta)}{\hat\omega(\Theta)}\geq\hat\omega(\Theta)^{1/(m-1)}-3.
  \]
\end{theorem}

It is easy to see that \eqref{eq:jarnik_czech_1954_m=2} is stronger than the trivial estimate $\omega(\Theta)\geq\hat\omega(\Theta)$ only for $\hat\omega(\Theta)>2$. In 2013 Moshchevitin \cite{moshchevitin_acta_sci_2013} improved upon Jarn\'{\i}k's result for $m=2$, $n\geq2$. His inequality is stronger than \eqref{eq:jarnik_czech_1954_m=2} and than the trivial inequality for each $\hat\omega(\Theta)>1$.

\begin{theorem}[N.\,G.\,Moshchevitin, 2013] \label{t:moshchevitin_acta_sci_2013}
  Let $\Theta\in\R^{n\times 2}$, $n\geq2$. Assume that among the rows of $\Theta$ there are two ones that are linearly independent with the vectors $(1,0)$ and $(0,1)$ over $\Q$. Assume also that $\hat\omega(\Theta)\geq1$. Then
  \[
    \frac{\omega(\Theta)}{\hat\omega(\Theta)}\geq G(\Theta),
  \]
  where $G(\Theta)$ is defined as the greatest root of the polynomial
  \[
    \hat\omega(\Theta)x^2-\big(\hat\omega(\Theta)^2-\hat\omega(\Theta)+1\big)x-\big(\hat\omega(\Theta)-1\big)^2,
  \]
  if $1\leq\hat\omega(\Theta)\leq2$, and as the greatest root of the polynomial
  \[
    \hat\omega(\Theta)x^2-\big(\hat\omega(\Theta)^2-1\big)x-\big(\hat\omega(\Theta)-1\big),
  \]
  if $\hat\omega(\Theta)\geq2$.
\end{theorem}

For $m=2$, $n=2$ Theorem \ref{t:moshchevitin_acta_sci_2013} is not the first improvement of inequality \eqref{eq:jarnik_czech_1954_m=2}. In his survey \cite{moshchevitin_UMN_2010} Moshchevitin presented an improvement of \eqref{eq:jarnik_czech_1954_m=2}, however, it is indeed an improvement only if $1<\hat\omega(\Theta)<(3+\sqrt5)/2$. Theorem \ref{t:moshchevitin_acta_sci_2013} is stronger than this result.

It is worth mentioning that 
Schmidt noticed that the statement of Theorem \ref{t:moshchevitin_acta_sci_2013} in Moshchevitin's paper \cite{moshchevitin_acta_sci_2013} contains two misprints. First, Moshchevitin's proof works for every $n\geq2$, however, in the statement of the theorem the restriction $n\geq3$ is imposed. Second, there is a misprint in the formula which defines the constant $G(\Theta)$.


Currently, there are no estimates for the ratio $\omega(\Theta)/\hat\omega(\Theta)$ in the case when $\min(n,m)>1$ stronger than Theorems \ref{t:jarnik_czech_1954} and \ref{t:moshchevitin_acta_sci_2013} of Jarn\'{\i}k and Moshchevitin.

\subsection{Embedding into $\R^{m+n}$} \label{subsec:embedding_into_R_d}

Let us demonstrate how to use the constructions discussed in Section \ref{subsec:ideas_and_methods} to prove transference theorems. Let us generalise the approach described in Section \ref{subsubsec:parallelepiped_families} by ``embedding'' the problems corresponding to $\Theta$ and $\tr\Theta$ into an $(m+n)$-dimensional Euclidean space. Set
\[
  d=m+n.
\]
Consider the lattices
\begin{equation}\label{eq:lattices_arbitrary_nm}
  \La=\La(\Theta)=
  \begin{pmatrix}
    \vec I_m & \\
    -\Theta  & \vec I_n
  \end{pmatrix}
  \Z^d,
  \qquad
  \La^\ast=\La^\ast(\Theta)=
  \begin{pmatrix}
    \vec I_m & \tr\Theta \\
    & \vec I_n
  \end{pmatrix}
  \Z^d,
\end{equation}
and two families of parallelepipeds
\begin{align}
  \label{eq:t_gamma_family_arbitrary_nm}
  \cP(t,\gamma) & =\Bigg\{\,\vec z=(z_1,\ldots,z_d)\in\R^d \ \Bigg|
                        \begin{array}{l}
                          |z_j|\leq t,\qquad\ \ j=1,\ldots,m \\
                          |z_{m+i}|\leq t^{-\gamma},\ \ i=1,\ldots,n
                        \end{array} \Bigg\}, \\
  \label{eq:s_delta_family_arbitrary_nm}
  \cQ(s,\delta) & =\Bigg\{\,\vec z=(z_1,\ldots,z_d)\in\R^d \ \Bigg|
                        \begin{array}{l}
                          |z_j|\leq s^{-\delta},\quad\ \ j=1,\ldots,m \\
                          |z_{m+i}|\leq s,\quad\ \ i=1,\ldots,n
                        \end{array} \Bigg\}.
\end{align}
Then
\begin{equation}\label{eq:omega_vs_parallelepipeds_arbitrary_nm}
\begin{aligned}
  \omega(\Theta) & =\sup\bigg\{ \gamma\geq\frac mn \,\bigg|\ \forall\,t_0\in\R\,\ \exists\,t>t_0:\text{\,we have }\cP(t,\gamma)\cap\La\neq\{\vec 0\} \bigg\}, \\
  \hat\omega(\Theta) & =\sup\bigg\{ \gamma\geq\frac mn \,\bigg|\ \exists\,t_0\in\R:\ \forall\,t>t_0\text{ we have }\cP(t,\gamma)\cap\La\neq\{\vec 0\} \bigg\}, \\
  \omega(\tr\Theta) & =\sup\bigg\{ \delta\geq\frac nm\ \bigg|\ \forall\,s_0\in\R\,\ \exists\,s>s_0:\text{\,we have }\cQ(s,\delta)\cap\La^\ast\neq\{\vec 0\} \bigg\}, \\
  \hat\omega(\tr\Theta) & =\sup\bigg\{ \delta\geq\frac nm\ \bigg|\ \exists\,s_0\in\R:\ \forall\,s>s_0\text{ we have }\cQ(s,\delta)\cap\La^\ast\neq\{\vec 0\} \bigg\}.
\end{aligned}
\end{equation}

\subsubsection{Dyson's theorem}\label{subsubsec:dyson_theorem}

If
\begin{equation}\label{eq:Q_is_P_star_arbitrary_nm}
  t=s^{((n-1)\delta+n)/(d-1)},
  \qquad
  \gamma=\frac{m\delta+m-1}{(n-1)\delta+n}\,,
\end{equation}
then $\cQ(s,\delta)$ is the $(d-1)$-th compound (see Definition \ref{def:pseudo_compound} in Section \ref{subsubsec:pseudocompounds_and_dual_lattices}) of $\cP(t,\gamma)$, i.e. $\cQ(s,\delta)=\cP(t,\gamma)^\ast$. By Mahler's theorem (once again, in disguise of Theorem \ref{t:mahler_reformulated}) we have
\begin{equation}\label{eq:dyson_transference_implication}
  \cQ(s,\delta)\cap\La^\ast\neq\{\vec 0\}
  \implies
  (d-1)\cP(t,\gamma)\cap\La\neq\{\vec 0\}.
\end{equation}
Hence by \eqref{eq:omega_vs_parallelepipeds_arbitrary_nm}
\[
  \omega(\tr\Theta)\geq\delta
  \implies
  \omega(\Theta)\geq\gamma=\frac{m\delta+m-1}{(n-1)\delta+n}\,.
\]
We get
\begin{equation*}
  \omega(\Theta)\geq\frac{\omega(m\tr\Theta)+m-1}{(n-1)\omega(\tr\Theta)+n}\,.
\end{equation*}
This is how inequality \eqref{eq:dyson_transference_swapped} and, therefore, Theorem \ref{t:dyson_transference} of Dyson is proved.

\subsubsection{Badly approximable matrices}\label{subsubsec:badly_approximable_matrices}

Correspondence \eqref{eq:Q_is_P_star_arbitrary_nm} and implication \eqref{eq:dyson_transference_implication} provide a very simple proof of the fact that $\Theta$ is badly approximable if and only if so is $\tr\Theta$ (see Definition \ref{def:BA_matrix}). Indeed, $\Theta$ is badly approximable if and only if there is a constant $c>0$ such that for every $t>1$ the parallelepiped $c\cP(t,m/n)$ contains no nonzero points of $\La$. Similarly, $\tr\Theta$ is badly approximable if and only if there is a constant $c>0$ such that for every $s>1$ the parallelepiped $c\cQ(s,n/m)$ contains no nonzero points of $\La^\ast$. Applying \eqref{eq:dyson_transference_implication} completes the argument.

\subsubsection{Theorem on uniform exponents}

Theorem \ref{t:my_inequalities} on uniform exponents can be proved with the help of the ``nodes'' and ``leaves'' construction described in Section \ref{subsubsec:nodes_and_leaves}. Now, instead of parallelepipeds $\cQ_r$ determined by \eqref{eq:Q_r}, we are to consider parallelepipeds
\begin{equation*}
  \cQ_r=\Bigg\{\,\vec z=(z_1,\ldots,z_d) \in\R^d \ \Bigg|
                 \begin{array}{l}
                   |z_j|\leq (hH/r)^{-\alpha},\quad\, j=1,\ldots,m \\
                   |z_{m+i}|\leq r,\qquad\qquad i=1,\ldots,n
                 \end{array} \Bigg\}.
\end{equation*}
Respectively, parallelepipeds $\cP(t,\gamma)$ and $\cQ(s,\delta)$ are to be defined not by \eqref{eq:t_gamma_family}, \eqref{eq:s_delta_family}, but by \eqref{eq:t_gamma_family_arbitrary_nm}, \eqref{eq:s_delta_family_arbitrary_nm}. Then we can use Fig. \ref{fig:nodes_and_leaves} unaltered, with the agreement that $u$ and $v$ denote now $\max\big(|z_{m+1}|,\ldots,|z_d|\big)$ and $\max\big(|z_1|,\ldots,|z_m|\big)$ respectively.

Upon fixing $s>1$ and $\delta\geq n/m$, we set, as in \eqref{eq:Q_is_P_star_arbitrary_nm},
\[
  t=s^{((n-1)\delta+n)/(d-1)},
  \qquad
  \gamma=\frac{m\delta+m-1}{(n-1)\delta+n}\,.
\]
Set also
\[
  h=s,\qquad
  \beta=\delta,\qquad
  \alpha=
  \begin{cases}
    \dfrac{(d-1)\delta}{m\delta+m-1}
    \quad\,\ \text{ if }\ \delta\leq\dfrac{m-1}{n-1}\,, \\
    \dfrac{(n-1)\delta+n\vphantom{1^{\big|}}}{d-1}
    \quad\ \text{ if }\ \delta\geq\dfrac{m-1}{n-1}\,.
  \end{cases}
\]
With such a choice of parameters, the quantities $\gamma$ and $\alpha$ are related by
\begin{equation}\label{eq:gamma_via_alpha_arbitrary_nm}
  \gamma=
  \begin{cases}
    \dfrac{m-1}{n-\alpha}\qquad\ \text{ if }\ \alpha\leq1, \\
    \dfrac{m-\alpha^{-1\vphantom{\big|}}}{n-1}\quad\ \text{ if }\ \alpha\geq1.
  \end{cases}
\end{equation}
Arguing in the manner of Section \ref{subsubsec:proof_of_uniform_transference}, we get the implication
\[
  \hat\omega(\tr\Theta)\geq\alpha
  \implies
  \hat\omega(\Theta)\geq\gamma,
\]
whence it follows, in view of \eqref{eq:gamma_via_alpha_arbitrary_nm}, that
\[
  \hat\omega(\Theta)\geq
  \begin{cases}
    \dfrac{m-1}{n-\hat\omega(\tr\Theta)}\qquad\,\text{ if }\ \hat\omega(\tr\Theta)\leq1, \\
    \dfrac{m-\hat\omega(\tr\Theta)^{-1\vphantom{\big|}}}{n-1}\quad\text{ if }\ \hat\omega(\tr\Theta)\geq1.
  \end{cases}
\]
Swapping the triple $(n,m,\Theta)$ for the triple $(m,n,\tr\Theta)$, we come to \eqref{eq:my_inequalities}. This is how Theorem \ref{t:my_inequalities} is proved.

\subsubsection{Parametric geometry of numbers}\label{subsubsec:parametric_geometry_of_numbers_arbitrary_nm}

Let us interpret the problem of approximating zero with the values of several linear forms in the spirit of parametric geometry of numbers. According to Section \ref{subsubsec:parametric_geometry_of_numbers}, let us choose a lattice and a path. Define the lattices $\La$ and $\La^\ast$ by \eqref{eq:lattices_arbitrary_nm}. Define the paths $\gT$ and $\gT^\ast$ respectively by the mappings
\begin{equation*}
\begin{aligned}
  & s\mapsto\pmb\tau(s)=\big(\tau_1(s),\ldots,\tau_d(s)\big), \\
  & \tau_1(s)=\ldots=\tau_m(s)=s,\quad\tau_{m+1}(s)=\ldots=\tau_d(s)=-ms/n
\end{aligned}
\end{equation*}
and
\begin{equation*}
\begin{aligned}
  & s\mapsto\pmb\tau^\ast(s)=\big(\tau^\ast_1(s),\ldots,\tau^\ast_d(s)\big), \\
  & \tau^\ast_1(s)=\ldots=\tau^\ast_m(s)=-ns/m,\quad\tau^\ast_{m+1}(s)=\ldots=\tau^\ast_d(s)=s.
\end{aligned}
\end{equation*}
Then the regular and uniform Diophantine exponents can be expressed in terms of Schmidt--Summerer exponents with the help of the relations
\begin{equation}\label{eq:belpha_via_psis_arbitrary_nm}
\begin{aligned}
  & \big(1+\omega(\Theta)\big)\big(1+\bpsi_1(\La,\gT)\big)=\big(1+\hat\omega(\Theta)\big)\big(1+\apsi_1(\La,\gT)\big)=d/n, \\
  & \big(1+\omega(\tr\Theta)\big)\big(1+\bpsi_1(\La^\ast,\gT^\ast)\big)=\big(1+\hat\omega(\tr\Theta)\big)\big(1+\apsi_1(\La^\ast,\gT^\ast)\big)=d/m.
\end{aligned}
\end{equation}
They are deduced, same as relations \eqref{eq:belpha_via_psis}, directly from the definitions (see \cite{german_AA_2012}). 
Since the functions $L_k(\pmb\tau)$ and $S_k(\pmb\tau)$ enjoy properties (i)--(iv) formulated in Section \ref{subsubsec:parametric_geometry_of_numbers} with any given lattice, we also have 
\begin{equation*}
  S_1(\La,\pmb\tau)\leq
  \dfrac{S_{d-1}(\La,\pmb\tau)}{d-1}=
  \dfrac{S_1(\La^\ast,-\pmb\tau)}{d-1}+O(1).
\end{equation*}
For Schmidt--Summerer exponents this gives
\begin{equation}\label{eq:essence_of_transference_in_exponents_arbitrary_nm}
  \bPsi_1(\La,\gT)
  \leq
  \dfrac{\bPsi_{d-1}(\La,\gT)}{d-1}=
  \dfrac{n}{m(d-1)}\,\bPsi_1(\La^\ast,\gT^\ast),
\end{equation}
since $\pmb\tau^\ast(s)=-\frac nm\pmb\tau(s)$. Taking into account that $\bPsi_1(\La,\gT)=\bpsi_1(\La,\gT)$ and $\bPsi_1(\La^\ast,\gT^\ast)=\bpsi_1(\La^\ast,\gT^\ast)$, we get
\begin{equation}\label{eq:dyson_schmimmerered}
  \bpsi_1(\La,\gT)
  \leq
  \dfrac{n}{m(d-1)}\,\bpsi_1(\La^\ast,\gT^\ast).
\end{equation}
Rewriting \eqref{eq:dyson_schmimmerered} with the help of \eqref{eq:belpha_via_psis_arbitrary_nm} in terms of $\omega(\Theta)$, $\omega(\tr\Theta)$ leads exactly to Dyson's inequality \eqref{eq:dyson_transference}. Thus, parametric geometry of numbers provides another way to prove Dyson's theorem. And same as with Khintchine's inequalities, this method enables to refine Dyson's inequality by splitting it into a chain of inequalities between intermediate exponents. The explicit formulation of this result, as well as its proof, can be found in \cite{german_AA_2012}.

\section{Multiplicative exponents}\label{sec:mult}

As in the previous Section, given a matrix
\[ 
  \Theta=
  \begin{pmatrix}
    \theta_{11} & \cdots & \theta_{1m} \\
    \vdots & \ddots & \vdots \\
    \theta_{n1} & \cdots & \theta_{nm}
  \end{pmatrix},\qquad
  \theta_{ij}\in\R,\quad m+n\geq3,
\]
let us consider the same system of linear equations
\begin{equation*} 
  \Theta\vec x=\vec y
\end{equation*}
with variables $\vec x=(x_1,\ldots,x_m)\in\R^m$, $\vec y=(y_1,\ldots,y_n)\in\R^n$.

Up until now, we have been dealing with the question how small the quantity $|\Theta\vec x-\vec y|$ can be with $\vec x\in\Z^m$, $\vec y\in\Z^n$ satisfying the restriction $0<|\vec x|\leq t$, where $|\cdot|$ denotes $\ell^\infty$-norm (sup-norm). The choice of $\ell^\infty$-norm for estimating the ``magnitude'' of a vector is quite conditional. Considering $\ell^2$-norm slightly alters the problem preserving, however, the values of Diophantine exponents. The reason is that any two norms in a finite dimensional space are known to be equivalent. The problem changes essentially if a vector's ``magnitude'' is measured by the product of its coordinates. If we associate $\ell^1$-norm with the arithmetic mean, then the product of coordinates can be associated with the geometric mean. Such an approach is also absolutely classical. As an example, we can recall the famous Littlewood conjecture.

\begin{littlewood}
  For any $\theta_1,\theta_2\in\R$ and any $\e>0$ the inequality
  \[
    \prod_{i=1,2}|\theta_ix-y_i|
    \leq
    \e|x|^{-1}
  \]
  admits infinitely many solutions in $(x,y_1,y_2)\in\Z^3$ with nonzero $x$.
\end{littlewood}

Respectively, it is natural to consider as a multiplicative analogue of the regular Diophantine exponent of a pair $(\theta_1,\theta_2)$ the supremum of real $\gamma$ such that the inequality
\[
  \prod_{i=1,2}|\theta_ix-y_i|^{1/2}
  \leq
  |x|^{-\gamma}
\]
admits infinitely many solutions in $(x,y_1,y_2)\in\Z^3$ with nonzero $x$. This is an example for the case $n=2$, $m=1$.

In order to give the definition of multiplicative exponents in the general case, let us set for each $\vec z=(z_1,\ldots,z_k)\in\R^k$
\[
  \Pi(\vec z)=\prod_{\begin{subarray}{c}1\leq i\leq k\end{subarray}}|z_i|^{1/k}
  \quad\text{ and }\quad
  \Pi'(\vec z)=\prod_{1\leq i\leq k}\max\big(1,|z_i|\big)^{1/k}.
\]

\begin{definition} \label{def:mbeta}
  The supremum of real $\gamma$ satisfying the condition that there exist arbitrarily large $t$ such that (resp. for every $t$ large enough) the system
  \begin{equation} \label{eq:mbeta}
    \Pi'(\vec x)\leq t,
    \qquad
    \Pi(\Theta\vec x-\vec y)\leq t^{-\gamma}
  \end{equation}
  admits a nonzero solution $(\vec x,\vec y)\in\Z^m\oplus\Z^n$ with nonzero $\vec x$ is called the \emph{regular} (resp. \emph{uniform}) \emph{multiplicative Diophantine exponent} of $\Theta$ and is denoted by $\omega_\times(\Theta)$ (resp. $\hat\omega_\times(\Theta)$).
\end{definition}

Ordinary and multiplicative exponents are related by the inequalities
\begin{equation} \label{eq:ord_and_mult_trivial}
\begin{aligned}
  &
  \omega(\Theta)\leq
  \omega_\times(\Theta)\leq
  \begin{cases}
    m\omega(\Theta),\quad
    \text{ if }n=1, \\
    +\infty,\qquad\
    \text{ if }n\geq2,
  \end{cases} \\
  &
  \hat\omega(\Theta)\leq
  \hat\omega_\times(\Theta)\leq
  \begin{cases}
    m\hat\omega(\Theta),\quad
    \text{ if }n=1, \\
    +\infty,\qquad\
    \text{ if }n\geq2,
  \end{cases}
\end{aligned}
\end{equation}
which follow from the fact that for each $\vec z\in\R^k$ we have
\[
  \Pi(\vec z)\leq|\vec z|,
\]
and for each $\vec z\in\Z^k$ we have
\[
  |\vec z|^{1/k}\leq\Pi'(\vec z)\leq|\vec z|.
\]
Corollary \ref{cor:belpha_matrix_trivial_inequalities} to Dirichlet's theorem implies ``trivial'' inequalities
\[
\begin{aligned}
  &
  \omega_\times(\Theta)\geq\omega(\Theta)\geq m/n,
  \qquad
  \omega_\times(\tr\Theta)\geq\omega(\tr\Theta)\geq n/m, \\
  &
  \hat\omega_\times(\Theta)\geq\hat\omega(\Theta)\geq m/n,
  \qquad
  \hat\omega_\times(\tr\Theta)\geq\hat\omega(\tr\Theta)\geq n/m,
\end{aligned}
\]
where $\tr\Theta$, as before, denotes the transposed matrix.

\subsection{Analogue of Dyson's theorem}

In 1979 Schmidt and Wang showed in \cite{schmidt_wang} that, same as in the case of ordinary Diophantine exponents (see \eqref{eq:dyson_n_m_equivalence}), we have
\begin{equation}\label{eq:schmidt_wang}
  \omega_\times(\Theta)=m/n\iff\omega_\times(\tr\Theta)=n/m.
\end{equation}
Again, this is a manifestation of the transference principle. As Bugeaud noted in \cite{bugeaud_multiplicative_proceedings}, the argument Schmidt and Wang used to prove \eqref{eq:schmidt_wang} can be applied to prove an analogue of Dyson's inequality \eqref{eq:dyson_transference}, under certain restrictions imposed on $\Theta$. The respective analogue, with no restriction on $\Theta$, was obtained by the author in \cite{german_mult}.

\begin{theorem}[O.\,G, 2011]\label{t:german_multiplicative_dysonlike}
  For every $\Theta\in\R^{n\times m}$ we have
  \begin{equation}\label{eq:german_multiplicative_dysonlike}
    \omega_\times(\tr\Theta)\geq\frac{n\omega_\times(\Theta)+n-1}{(m-1)\omega_\times(\Theta)+m}\,.
  \end{equation}
\end{theorem}

The method \eqref{eq:german_multiplicative_dysonlike} is proved with gives an analogous inequality for uniform exponents:
\begin{equation}\label{eq:german_multiplicative_uniform_dysonlike}
  \hat\omega_\times(\tr\Theta)\geq\frac{n\hat\omega_\times(\Theta)+n-1}{(m-1)\hat\omega_\times(\Theta)+m}\,.
\end{equation}
However, nothing stronger concerning uniform multiplicative exponents is currently known. The method Theorem \ref{t:my_inequalities} concerning ordinary uniform exponents was proved with does not work in the multiplicative setting because of the non-convexity of the functionals $\Pi(\cdot)$ and $\Pi'(\cdot)$. Besides that, there exist neither ``mixed'' inequalities in the spirit of Theorems \ref{t:bugeaud_laurent}, \ref{t:schmidt_summerer_2013}, \ref{t:loranoyadenie}, nor inequalities relating regular and uniform exponents -- even in the case $m+n=3$.

\subsection{Applying Mahler's theorem and dimension reduction}

The standard application of Mahler's theorem -- either in the form of Theorem \ref{t:mahler}, or in the form of Theorem \ref{t:mahler_reformulated} -- is not enough for proving Theorem \ref{t:german_multiplicative_dysonlike}. The reason is that $\Pi'(\cdot)$ differs from $\Pi(\cdot)$. In the aforementioned paper \cite{schmidt_wang} Schmidt and Wang evade this obstacle by induction with respect to dimension. The same technique is used in paper \cite{german_mult}. Here we describe a somewhat more explicit construction that also requires reduction of dimension.

\subsubsection{Embedding into $\R^{m+n}$} \label{subsubsec:embedding_into_R_d_mult}

Working with the functionals $\Pi(\cdot)$ and $\Pi'(\cdot)$ requires considering more vast families of parallelepipeds than \eqref{eq:t_gamma_family_arbitrary_nm} and \eqref{eq:s_delta_family_arbitrary_nm}.

As in Section \ref{subsec:embedding_into_R_d}, let us set
\[
  d=m+n
\]
and consider the lattices
\begin{equation}\label{eq:lattices_mult}
  \La=\La(\Theta)=
  \begin{pmatrix}
    \vec I_m & \\
    -\Theta  & \vec I_n
  \end{pmatrix}
  \Z^d,
  \qquad
  \La^\ast=\La^\ast(\Theta)=
  \begin{pmatrix}
    \vec I_m & \tr\Theta \\
    & \vec I_n
  \end{pmatrix}
  \Z^d.
\end{equation}
For each tuple $(\pmb\lambda,\pmb\mu)=(\lambda_1,\ldots,\lambda_m,\mu_1,\ldots,\mu_n)\in\R_+^d$, let us define the parallelepiped $\cP(\pmb\lambda,\pmb\mu)$ by
\begin{equation}\label{eq:prallelepipeds_mult}
  \cP(\pmb\lambda,\pmb\mu)=\Bigg\{\,\vec z=(z_1,\ldots,z_d)\in\R^d \ \Bigg|
  \begin{array}{l}
    |z_j|\leq\lambda_j,\quad\,\ j=1,\ldots,m \\
    |z_{m+i}|\leq\mu_i,\ \ i=1,\ldots,n
  \end{array} \Bigg\}.
\end{equation}
Let us also define, for all positive $t$, $\gamma$, $s$, $\delta$, the families
\begin{align}
  \label{eq:t_gamma_family_arbitrary_nm_for_mult}
  \cF(t,\gamma) & =\Big\{\, \cP(\pmb\lambda,\pmb\mu) \ \Big|\
                            \Pi(\pmb\lambda)=t,\
                            \Pi(\pmb\mu)=t^{-\gamma},\
                            \min_{1\leq j\leq m}\lambda_j\geq1 \Big\}, \\
  \label{eq:s_delta_family_arbitrary_nm_for_mult}
  \cG(s,\delta) & =\Big\{\, \cP(\pmb\lambda,\pmb\mu) \ \Big|\
                            \Pi(\pmb\lambda)=s^{-\delta},\
                            \Pi(\pmb\mu)=s,\
                            \min_{1\leq i\leq n}\mu_i\geq1 \Big\}.
\end{align}
Every parallelepiped $\cP(\pmb\lambda,\pmb\mu)$ satisfying the conditions
\begin{equation}\label{eq:lambda_mu_for_omega_mult}
  \Pi'(\pmb\lambda)\leq t,
  \qquad
  \Pi(\pmb\mu)\leq t^{-\gamma},
\end{equation}
is contained in a parallelepiped from \eqref{eq:t_gamma_family_arbitrary_nm_for_mult}. Conversely, each parallelepiped $\cP(\pmb\lambda,\pmb\mu)$ from \eqref{eq:t_gamma_family_arbitrary_nm_for_mult} satisfies \eqref{eq:lambda_mu_for_omega_mult}. Similarly, every parallelepiped $\cP(\pmb\lambda,\pmb\mu)$ satisfying the conditions
\begin{equation}\label{eq:lambda_mu_for_omega_mult_transpose}
  \Pi(\pmb\lambda)\leq s^{-\delta},
  \qquad
  \Pi'(\pmb\mu)\leq s
\end{equation}
is contained in a parallelepiped from \eqref{eq:s_delta_family_arbitrary_nm_for_mult}. And conversely, each parallelepiped $\cP(\pmb\lambda,\pmb\mu)$ from \eqref{eq:s_delta_family_arbitrary_nm_for_mult} satisfies \eqref{eq:lambda_mu_for_omega_mult_transpose}. Therefore, the following analogue of \eqref{eq:omega_vs_parallelepipeds} and \eqref{eq:omega_vs_parallelepipeds_arbitrary_nm} holds for multiplicative exponents:
\begin{equation}\label{eq:omega_mult_vs_parallelepipeds}
\begin{aligned}
  \omega_\times(\Theta) & =
  \sup\bigg\{ \gamma\geq\frac mn \,\bigg|\
              \forall\,t_0\in\R\,\ \exists\,t>t_0:\,\
              \exists\cP\in\cF(t,\gamma):\ \cP\cap\La\neq\{\vec 0\} \bigg\}, \\
  \omega_\times(\tr\Theta) & =
  \sup\bigg\{ \delta\geq\frac nm\ \bigg|\
              \forall\,s_0\in\R\,\ \exists\,s>s_0:\,\
              \exists\cP\in\cG(s,\delta):\ \cP\cap\La^\ast\neq\{\vec 0\} \bigg\}.
\end{aligned}
\end{equation}

\subsubsection{Pseudocompounds} \label{subsubsec:pseudocompound_mult}

Let us assign to each tuple $(\pmb\lambda,\pmb\mu)=(\lambda_1,\ldots,\lambda_m,\mu_1,\ldots,\mu_n)\in\R_+^d$ the tuple $(\pmb\lambda^\ast,\pmb\mu^\ast)=(\lambda_1^\ast,\ldots,\lambda_m^\ast,\mu_1^\ast,\ldots,\mu_n^\ast)$,
\begin{equation}\label{eq:lambda_mu_ast}
\begin{aligned}
  & \lambda_j^\ast=\lambda_j^{-1}\Pi(\pmb\lambda)^m\Pi(\pmb\mu)^n,
    \qquad j=1,\ldots,m,
    \vphantom{\bigg|} \\
  & \mu_i^\ast=\mu_i^{-1}\Pi(\pmb\lambda)^m\Pi(\pmb\mu)^n,
    \qquad\,i=1,\ldots,n.
\end{aligned}
\end{equation}
Then $\cP(\pmb\lambda,\pmb\mu)^\ast=\cP(\pmb\lambda^\ast,\pmb\mu^\ast)$, i.e. $\cP(\pmb\lambda^\ast,\pmb\mu^\ast)$ is the pseudocompound of $\cP(\pmb\lambda,\pmb\mu)$ (see Definition \ref{def:pseudo_compound} in Section \ref{subsubsec:pseudocompounds_and_dual_lattices}).

As in Section \ref{subsubsec:dyson_theorem}, let us assign to each pair $(s,\delta)$ the pair $(t,\gamma)$ determined by \ref{eq:Q_is_P_star_arbitrary_nm}, i.e.
\begin{equation}\label{eq:Q_is_P_star_arbitrary_nm_for_mult}
  t=s^{((n-1)\delta+n)/(d-1)},
  \qquad
  \gamma=\frac{m\delta+m-1}{(n-1)\delta+n}.
\end{equation}
Consider the family
\[
  \cF^\ast(t,\gamma)=\big\{ \cP^\ast \,\big|\, \cP\in\cF(t,\gamma) \big\}
\]
of pseudocompounds. If $\cP(\pmb\lambda,\pmb\mu)\in\cF(t,\gamma)$, then
\[
\begin{aligned}
  & \,\ \Pi(\pmb\lambda^\ast)=
    \Pi(\pmb\lambda)^{m-1}\Pi(\pmb\mu)^n=
    t^{(m-1)-n\gamma}=
    s^{-\delta}, \\
  & \,\ \Pi(\pmb\mu^\ast)=
    \Pi(\pmb\lambda)^m\Pi(\pmb\mu)^{n-1}=
    t^{m-(n-1)\gamma}=
    s,\vphantom{\bigg|} \\
  & \max_{1\leq j\leq m}\lambda_j^\ast
    \leq
    \Pi(\pmb\lambda)^m\Pi(\pmb\mu)^n.
\end{aligned}
\]
For convenience, let us set
\begin{equation}\label{eq:lambda_mu_product}
  \pi(\pmb\lambda,\pmb\mu)=\Pi(\pmb\lambda)^m\Pi(\pmb\mu)^n.
\end{equation}
Then $\cF^\ast(t,\gamma)$ can be written as
\begin{equation}\label{eq:t_gamma_family_arbitrary_nm_for_mult_pseudocompound}
  \cF^\ast(t,\gamma)=\Big\{\, \cP(\pmb\lambda,\pmb\mu) \ \Big|\
                              \Pi(\pmb\lambda)=s^{-\delta},\
                              \Pi(\pmb\mu)=s,\
                              \max_{1\leq j\leq m}\lambda_j\leq\pi(\pmb\lambda,\pmb\mu) \Big\}.
\end{equation}
As we can see, the inclusion $\cF^\ast(t,\gamma)\subset\cG(s,\delta)$ does not hold because of the absence of the condition $\min_{1\leq i\leq n}\mu_i\geq1$ in \eqref{eq:t_gamma_family_arbitrary_nm_for_mult_pseudocompound}. Therefore, we consider more narrow families
\begin{equation}\label{eq:G_prime_definition}
\begin{aligned}
  \cG'(s,\delta) & =\Big\{\, \cP(\pmb\lambda,\pmb\mu)\in\cF^\ast(t,\gamma) \ \Big|\
                             \min_{1\leq i\leq n}\mu_i\geq1 \Big\}= \\
                 & =\Big\{\, \cP(\pmb\lambda,\pmb\mu)\in\cG(s,\delta) \ \Big|\
                             \max_{1\leq j\leq m}\lambda_j\leq\pi(\pmb\lambda,\pmb\mu) \Big\}=\vphantom{\frac{\big|}{}} \\
                 & =\Bigg\{\, \cP(\pmb\lambda,\pmb\mu) \ \Bigg|\
                              \begin{array}{l}
                                \Pi(\pmb\lambda)=s^{-\delta},\ \
                                \displaystyle
                                \max_{1\leq j\leq m}\lambda_j\leq\pi(\pmb\lambda,\pmb\mu)\vphantom{1^{\big|}} \\
                                \Pi(\pmb\mu)=s,\quad\ \
                                \displaystyle
                                \min_{1\leq i\leq n}\mu_i\geq1
                              \end{array} \Bigg\},
\end{aligned}
\end{equation}
\begin{equation}\label{eq:F_prime_definition}
\begin{aligned}
  \cF'(t,\gamma) & =\Big\{\, \cP(\pmb\lambda,\pmb\mu)\in\cF(t,\gamma) \ \Big|\
                             \max_{1\leq i\leq n}\mu_i\leq\pi(\pmb\lambda,\pmb\mu) \Big\}= \vphantom{\Big|} \\
                 & =\Bigg\{\, \cP(\pmb\lambda,\pmb\mu) \ \Bigg|\
                              \begin{array}{l}
                                \Pi(\pmb\lambda)=t,\quad\,\
                                \displaystyle
                                \min_{1\leq j\leq m}\lambda_j\geq1\vphantom{1^{\big|}} \\
                                \Pi(\pmb\mu)=t^{-\gamma},\ \
                                \displaystyle
                                \max_{1\leq i\leq n}\mu_i\leq\pi(\pmb\lambda,\pmb\mu)
                              \end{array} \Bigg\}.\ \
\end{aligned}
\end{equation}
Then
\[
  \cG'(s,\delta)=
  \big\{ \cP^\ast \,\big|\, \cP\in\cF'(t,\gamma) \big\}.
\]
Thus, Mahler's theorem in disguise of Theorem \ref{t:mahler_reformulated} gives the implication
\begin{equation}\label{eq:multi_transference_implication_before_induction}
  \exists\cP\in\cG'(s,\delta):\cP\cap\La^\ast\neq\{\vec 0\}
  \implies
  \exists\cP\in\cF'(t,\gamma):(d-1)\cP\cap\La\neq\{\vec 0\}.
\end{equation}
But proving Theorem \ref{t:german_multiplicative_dysonlike} requires an analogous implication for the families $\cG(s,\delta)$ and $\cF(t,\gamma)$. Note, however, that it suffices to show that we can replace $\cG'(s,\delta)$ with $\cG(s,\delta)$ in \eqref{eq:multi_transference_implication_before_induction}, since $\cF'(t,\gamma)\subset\cF(t,\gamma)$.

\subsubsection{Dimension reduction} \label{subsubsec:dimension_reduction}

Let us show that we can indeed improve \eqref{eq:multi_transference_implication_before_induction} the way indicated above, i.e. that the following statement holds.

\begin{theorem}\label{t:multi_transference_implication_after_induction}
  Let $\La$, $\La^\ast$ be defined by \eqref{eq:lattices_mult}. Let $t$, $\gamma$, $s$, $\delta$ be positive real numbers related by \eqref{eq:Q_is_P_star_arbitrary_nm_for_mult}. Let $\cF(s,\delta)$, $\cG(s,\delta)$, $\cF'(t,\gamma)$, $\cG'(s,\delta)$ be defined by \eqref{eq:t_gamma_family_arbitrary_nm_for_mult}, \eqref{eq:s_delta_family_arbitrary_nm_for_mult}, \eqref{eq:G_prime_definition}, \eqref{eq:F_prime_definition}. Then
  \begin{equation}\label{eq:multi_transference_implication_after_induction}
    \exists\cP\in\cG(s,\delta):\cP\cap\La^\ast\neq\{\vec 0\}
    \implies
    \exists\cP\in\cF'(t,\gamma):(d-1)\cP\cap\La\neq\{\vec 0\}.
  \end{equation}
\end{theorem}

It is clear that for $m=1$ the families $\cG(s,\delta)$ and $\cG'(s,\delta)$ coincide, since $s>1$. Let us assume that $m\geq2$.

Consider arbitrary tuples $\pmb\lambda=(\lambda_1,\ldots,\lambda_m)\in\R_+^m$ and $\pmb\mu=(\mu_1,\ldots,\mu_n)\in\R_+^n$ satisfying the conditions
\begin{equation}\label{eq:lambda_mu_for_omega_mult_induction}
  \Pi(\pmb\lambda)=t,
  \qquad
  \Pi(\pmb\mu)=t^{-\gamma},
\end{equation}
and define $\pmb\lambda^\ast$, $\pmb\mu^\ast$ by \eqref{eq:lambda_mu_ast}. Suppose,
\begin{equation}\label{eq:suppose_P_ast_is_bad}
  \cP(\pmb\lambda^\ast,\pmb\mu^\ast)\in\cG(s,\delta)\backslash\cG'(s,\delta).
\end{equation}
Changing the ordering, if required, we can assume that
\[
  \lambda_1\leq\ldots\leq\lambda_m.
\]
Then by \eqref{eq:lambda_mu_ast}
\[
  \lambda_1^\ast\geq\ldots\geq\lambda_m^\ast.
\]
It follows from \eqref{eq:suppose_P_ast_is_bad} that $\lambda_1^\ast>\pi(\pmb\lambda,\pmb\mu)$. Hence by \eqref{eq:lambda_mu_ast} we also have $\lambda_1<1$. Let $k$ be the greatest index such that $\lambda_1\cdot\ldots\cdot\lambda_k<1$. Since $t>1$, it follows from the relation $\Pi(\pmb\lambda)=t$ that $k<m$. Taking into account \eqref{eq:lambda_mu_ast}, we get
\begin{equation}\label{eq:lambda_inequalities}
\begin{aligned}
  \lambda_1\cdot\ldots\cdot\lambda_k<1,\qquad &
  \lambda_m\geq\ldots\geq\lambda_{k+1}\geq1, \\
  \lambda_1\cdot\ldots\cdot\lambda_k\lambda_{k+1}^{m-k}\geq1,\qquad &
  \lambda^\ast_m\leq\ldots\leq\lambda^\ast_{k+1}\leq\pi(\pmb\lambda,\pmb\mu).
\end{aligned}
\end{equation}
Along with the tuple $\pmb\lambda$, let us consider the tuple $\hat{\pmb\lambda}=(\hat\lambda_1,\ldots,\hat\lambda_m)$,
\begin{equation}\label{eq:lambda_mu_hat}
\begin{aligned}
  & \hat\lambda_j=1,
    \qquad\qquad\qquad\qquad\qquad\quad\ \ j=1,\ldots,k, \\
  & \hat\lambda_j=\lambda_j\big(\lambda_1\cdot\ldots\cdot\lambda_k\big)^{1/(m-k)},
    \qquad\ \ j=k+1,\ldots,m.
    \vphantom{1^{\big|}}
\end{aligned}
\end{equation}
Then by \eqref{eq:lambda_mu_for_omega_mult_induction}, \eqref{eq:suppose_P_ast_is_bad}, and \eqref{eq:lambda_inequalities} we have
\[
  \Pi(\hat{\pmb\lambda})=t,\ \
  \Pi(\pmb\mu)=t^{-\gamma},\ \
  \min_{1\leq j\leq m}\hat\lambda_j\geq1,\ \
  \max_{1\leq i\leq n}\mu_i\leq\pi(\hat{\pmb\lambda},\pmb\mu),
\]
i.e.
\begin{equation}\label{eq:hat_is_good}
  \cP(\hat{\pmb\lambda},\pmb\mu)\in\cF'(t,\gamma).
\end{equation}

Let us show that
\begin{equation}\label{eq:bad_is_not_that_bad}
  \cP(\pmb\lambda^\ast,\pmb\mu^\ast)\cap\La^\ast\neq\{\vec 0\}
  \implies
  (d-1)\cP(\hat{\pmb\lambda},\pmb\mu)\cap\La\neq\{\vec 0\}.
\end{equation}
Consider the truncated tuples
\[
  \pmb\lambda_\downarrow=(\lambda_{k+1},\ldots,\lambda_{m}),
  \qquad
  \pmb\lambda^\ast_\downarrow=(\lambda^\ast_{k+1},\ldots,\lambda^\ast_{m}),
  \qquad
  \hat{\pmb\lambda}_\downarrow=(\hat\lambda_{k+1},\ldots,\hat\lambda_m).
\]
Then
\[
  \cP(\hat{\pmb\lambda}_\downarrow,\pmb\mu)^\ast=
  \Bigg\{\,(z_{k+1},\ldots,z_d)\in\R^{d-k} \ \Bigg|
  \begin{array}{l}
    |z_j|\leq c\lambda_j^\ast,\quad\,\ j=k+1,\ldots,m \\
    |z_{m+i}|\leq\mu_i^\ast,\quad i=1,\ldots,n
  \end{array} \Bigg\},
\]
where $c=\big(\lambda_1\cdot\ldots\cdot\lambda_k\big)^{-1/(m-k)}$. Since $c>1$, we have
\begin{equation}\label{eq:ammendment_of_bad}
  \cP(\pmb\lambda_\downarrow^\ast,\pmb\mu^\ast)
  \subset
  \cP(\hat{\pmb\lambda}_\downarrow,\pmb\mu)^\ast.
\end{equation}
Consider also the matrix
\[
  \Theta_\downarrow=
  \begin{pmatrix}
   \theta_{1\,k+1} & \cdots & \theta_{1m} \\
   \vdots & \ddots & \vdots \\
   \theta_{n\,k+1} & \cdots & \theta_{nm}
  \end{pmatrix}
\]
obtained from $\Theta$ by deleting the first $k$ columns, and the lattices
\begin{equation*}
  \La_\downarrow=
  \begin{pmatrix}
    \vec I_{m-k}        & \\
    -\Theta_\downarrow  & \vec I_n
  \end{pmatrix}
  \Z^{d-k},
  \qquad
  \La^\ast_\downarrow=
  \begin{pmatrix}
    \vec I_{m-k} & \tr\Theta_\downarrow \\
                 & \vec I_n
  \end{pmatrix}
  \Z^{d-k}.
\end{equation*}
It can be easily verified that the set
\[
  \Big\{ (0,\ldots,0,z_{k+1},\ldots,z_d)\in\R^d \ \Big|\
                    (z_{k+1},\ldots,z_d)\in\La_\downarrow \Big\}
\]
is a sublattice of $\La$, and the set
\[
  \Big\{ (0,\ldots,0,z_{k+1},\ldots,z_d)\in\R^d \ \Big|\
                    (z_{k+1},\ldots,z_d)\in\La^\ast_\downarrow \Big\}
\]
is the projection of $\La^\ast$ onto the plane of coordinates $z_{k+1},\ldots,z_d$. Thus, we get the implications
\begin{equation}\label{eq:down_and_up}
\begin{aligned}
  \cP(\pmb\lambda^\ast,\pmb\mu^\ast)\cap\La^\ast\neq\{\vec 0\}
  & \implies
  \cP(\pmb\lambda^\ast_\downarrow,\pmb\mu^\ast)\cap\La^\ast_\downarrow\neq\{\vec 0\}, \\
  \cP(\hat{\pmb\lambda}_\downarrow,\pmb\mu)\cap\La_\downarrow\neq\{\vec 0\}
  & \implies
  \cP(\hat{\pmb\lambda},\pmb\mu)\cap\La\neq\{\vec 0\}.
  \vphantom{1^{\big|}}
\end{aligned}
\end{equation}
Finally, applying Mahler's theorem in disguise of Theorem \ref{t:mahler_reformulated} we get the implication
\begin{equation}\label{eq:mahler_downarrowed}
  \cP(\hat{\pmb\lambda}_\downarrow,\pmb\mu)^\ast\cap\La^\ast_\downarrow\neq\{\vec 0\}
  \implies
  (d-k-1)\cP(\hat{\pmb\lambda}_\downarrow,\pmb\mu)\cap\La_\downarrow\neq\{\vec 0\}.
\end{equation}
Gathering up \eqref{eq:ammendment_of_bad}, \eqref{eq:down_and_up}, and \eqref{eq:mahler_downarrowed}, we get the following chain of implications:
\begin{multline*}
  \cP(\pmb\lambda^\ast,\pmb\mu^\ast)\cap\La^\ast\neq\{\vec 0\}
  \implies
  \cP(\pmb\lambda^\ast_\downarrow,\pmb\mu^\ast)\cap\La^\ast_\downarrow\neq\{\vec 0\}
  \implies \\ \implies
  \cP(\hat{\pmb\lambda}_\downarrow,\pmb\mu)^\ast\cap\La^\ast_\downarrow\neq\{\vec 0\}
  \implies
  (d-k-1)\cP(\hat{\pmb\lambda}_\downarrow,\pmb\mu)\cap\La_\downarrow\neq\{\vec 0\}
  \implies \vphantom{\bigg|} \\ \implies
  (d-k-1)\cP(\hat{\pmb\lambda},\pmb\mu)\cap\La\neq\{\vec 0\}
  \implies
  (d-1)\cP(\hat{\pmb\lambda},\pmb\mu)\cap\La\neq\{\vec 0\}.
\end{multline*}
Thus, indeed, \eqref{eq:bad_is_not_that_bad} does hold. Taking into account \eqref{eq:hat_is_good}, we get
\begin{equation}\label{eq:multi_transference_implication_instead_of_induction}
\begin{aligned}
  \exists\cP\in\cG(s,\delta)\backslash\cG'(s,\delta) & :\cP\cap\La^\ast\neq\{\vec 0\}
  \implies \\ & \implies
  \exists\cP\in\cF'(t,\gamma):(d-1)\cP\cap\La\neq\{\vec 0\}.
  \vphantom{1^{\big|}}
\end{aligned}
\end{equation}
Clearly, 
\eqref{eq:multi_transference_implication_instead_of_induction}
and
\eqref{eq:multi_transference_implication_before_induction}
give the desired implication
\eqref{eq:multi_transference_implication_after_induction}.
Thus, in view of \eqref{eq:omega_mult_vs_parallelepipeds}, Theorem \ref{t:german_multiplicative_dysonlike} is proved.

\subsubsection{Refinement of Theorem \ref{t:multi_transference_implication_after_induction}} \label{subsubsec:multi_transference_essence}

We actually prove something stronger than Theorem \ref{t:multi_transference_implication_after_induction} in Section \ref{subsubsec:dimension_reduction}. Implication \eqref{eq:bad_is_not_that_bad} can be naturally generalised to the case of an arbitrary parallelepiped $\cP(\pmb\lambda^\ast,\pmb\mu^\ast)\in\cG(s,\delta)$. First, we note that for $\cP(\pmb\lambda^\ast,\pmb\mu^\ast)\in\cG'(s,\delta)$ we can set $k$ to be equal to zero, as by \eqref{eq:lambda_mu_ast} we have $\lambda_j\geq1$ foe each $j$. In this case the equality $\hat{\pmb\lambda}=\pmb\lambda$ is an analogue of \eqref{eq:lambda_mu_hat}. Second, the definition of $\hat{\pmb\lambda}$ is naturally generalised to the case of arbitrary ordering of the components of $\pmb\lambda$. Let $\lambda_{j_1}\leq\ldots\leq\lambda_{j_m}$. Then, if $\lambda_{j_1}<1$, we set $k$ to be equal to the greatest index such that $\lambda_{j_1}\cdot\ldots\cdot\lambda_{j_k}<1$. If $\lambda_{j_1}\geq1$, we set $k=0$. Finally, we define the tuple $\hat{\pmb\lambda}=(\hat\lambda_1,\ldots,\hat\lambda_m)$ by
\begin{equation}\label{eq:lambda_mu_hat_general_case}
\begin{aligned}
  & \hat\lambda_{j_i}=1,
    \qquad\qquad\qquad\qquad\qquad\qquad\ i=1,\ldots,k, \\
  & \hat\lambda_{j_i}=\lambda_{j_i}\big(\lambda_{j_1}\cdot\ldots\cdot\lambda_{j_k}\big)^{1/(m-k)},
    \qquad\ \ i=k+1,\ldots,m.
    \vphantom{1^{\big|}}
\end{aligned}
\end{equation}
Arguing in the manner of Section \ref{subsubsec:dimension_reduction}, we come to the following refinement of Theorem \ref{t:multi_transference_implication_after_induction}.

\begin{theorem}\label{t:multi_transference_essence}
  Let $\La$, $\La^\ast$ be defined by \eqref{eq:lattices_mult}. Given arbitrary tuples $\pmb\lambda=(\lambda_1,\ldots,\lambda_m)\in\R_+^m$ and $\pmb\mu=(\mu_1,\ldots,\mu_n)\in\R_+^n$, define $\pmb\lambda^\ast$, $\pmb\mu^\ast$ by \eqref{eq:lambda_mu_ast}, and $\hat{\pmb\lambda}$ by \eqref{eq:lambda_mu_hat_general_case}. Then
  \begin{equation}\label{eq:multi_transference_essence}
    \cP(\pmb\lambda^\ast,\pmb\mu^\ast)\cap\La^\ast\neq\{\vec 0\}
    \implies
    (d-1)\cP(\hat{\pmb\lambda},\pmb\mu)\cap\La\neq\{\vec 0\}.
  \end{equation}
  Moreover,
  $
    \pi(\hat{\pmb\lambda},\pmb\mu)=
    \pi(\pmb\lambda,\pmb\mu),
  $
  $
    \displaystyle
    \min_{1\leq j\leq m}\hat\lambda_j\geq1,
  $
  and
  \[
    \,\
    \min_{1\leq i\leq n}\mu^\ast_i\geq1
    \implies
    \max_{1\leq i\leq n}\mu_i\leq\pi(\pmb\lambda,\pmb\mu).
  \]
\end{theorem}

Theorem \ref{t:multi_transference_essence} is the ``core'' of the multiplicative transference principle, same as Theorem \ref{t:mahler_reformulated} is the ``core'' of Khintchine's transference principle.

We note that, under the assumption that $\cP(\pmb\lambda^\ast,\pmb\mu^\ast)\cap\La^\ast\neq\{\vec 0\}$, Theorem \ref{t:mahler_reformulated} guarantees the existence of nonzero points of $\La$ in the parallelepiped $(d-1)\cP(\pmb\lambda,\pmb\mu)$, which is not contained in $(d-1)\cP(\hat{\pmb\lambda},\pmb\mu)$, if $\pmb\lambda$ has components strictly smaller than $1$ as well as ones strictly greater than $1$. It turns out that there is a whole family of pairwise distinct parallelepipeds, each containing a nonzero point of $\La$. The exact formulation of this fact can be found in paper \cite{german_evdokimov} devoted to improving Mahler's theorem.

We note also that the constant $d-1$ in \eqref{eq:multi_transference_essence} can be replaced with a smaller one, which tends to $1$ as $d\to\infty$. For instance, with $d^{1/(2d-2)}$ -- the same constant as in relation \eqref{eq:mahler_improved} strengthening the statement of Theorem \ref{t:mahler_reformulated}. Details can be found in \cite{german_mult}.

\subsection{Multiplicatively badly approximable matrices}\label{subsec:multi_bad}

The multiplicative setting also admits talking about badly approximable matrices.

\begin{definition} \label{def:MBA}
  A matrix $\Theta$ is called \emph{multiplicatively badly approximable} if there is a constant $c>0$ depending only on $\Theta$ such that for each pair $(\vec x,\vec y)\in\Z^m\oplus\Z^n$ with nonzero $\vec x$ we have
  \[
    \Pi'(\vec x)^m\Pi(\Theta\vec x-\vec y)^n\geq c.
  \]
\end{definition}

While the existence of badly approximable matrices can be proved rather easily, the question whether there exist multiplicatively badly approximable matrices is an open problem. Even in the simplest case $n=2$, $m=1$, the statement that multiplicatively badly approximable matrices exist is precisely the negation of Littlewood's conjecture (its formulation can be found in the beginning of Section \ref{sec:mult}).

In 1955 Cassels and Swinnerton-Dyer \cite{cassels_swinnerton_dyer} proved that if $n=2$, $m=1$ and $\Theta$ is multiplicatively badly approximable, then so is $\tr\Theta$. Theorem \ref{t:multi_transference_essence} enables proving the converse statement as well as its analogue for an arbitrary matrix. Indeed, by analogy with the deduction of \eqref{eq:omega_mult_vs_parallelepipeds}, it can be shown that $\Theta$ is multiplicatively badly approximable if and only if there is a constant $c>0$ such that none of the parallelepipeds in the family
\[
  \Big\{\, \cP(\pmb\lambda,\pmb\mu) \ \Big|\
           \pi(\pmb\lambda,\pmb\mu)=c,\
           \min_{1\leq j\leq m}\lambda_j\geq1 \Big\}
\]
contains nonzero points of $\La$ (for the notation, see \eqref{eq:lattices_mult}, \eqref{eq:prallelepipeds_mult}, and \eqref{eq:lambda_mu_product}). Similarly, $\tr\Theta$ is multiplicatively badly approximable if and only if there is a constant $c>0$ such that none of the parallelepipeds in the family
\[
  \Big\{\, \cP(\pmb\lambda,\pmb\mu) \ \Big|\
           \pi(\pmb\lambda,\pmb\mu)=c,\
           \min_{1\leq i\leq n}\mu_i\geq1 \Big\}
\]
contains nonzero points of $\La^\ast$. Since $\pi(\pmb\lambda^\ast,\pmb\mu^\ast)=\pi(\pmb\lambda,\pmb\mu)^{d-1}$, it follows immediately from Theorem \ref{t:multi_transference_essence} that if $\tr\Theta$ is not multiplicatively badly approximable, then neither is $\Theta$. Thus, the following theorem is valid.

\begin{theorem}\label{t:MBA_transference}
  A matrix $\Theta$ is multiplicatively badly approximable if and only if so is $\tr\Theta$.
\end{theorem}

\section{Diophantine approximation with weights}\label{sec:weights}

As before, let us consider a matrix
\[
  \Theta=
  \begin{pmatrix}
    \theta_{11} & \cdots & \theta_{1m} \\
    \vdots & \ddots & \vdots \\
    \theta_{n1} & \cdots & \theta_{nm}
  \end{pmatrix},\qquad
  \theta_{ij}\in\R,\quad m+n\geq3,
\]
and the system of linear equations
\begin{equation*} 
  \Theta\vec x=\vec y
\end{equation*}
with variables $\vec x=(x_1,\ldots,x_m)\in\R^m$, $\vec y=(y_1,\ldots,y_n)\in\R^n$.

In previous Sections we observed two approaches in measuring the ``magnitude'' of $\Theta\vec x-\vec y$: the first one uses the sup-norm, the second one uses the geometric mean of the absolute values of coordinates. There is also a, so to say, intermediate approach -- so called \emph{Diophantine approximation with weights}.

Let us fix \emph{weights} $\pmb\sigma=(\sigma_1,\ldots,\sigma_m)\in\R_{>0}^m$, $\pmb\rho=(\rho_1,\ldots,\rho_n)\in\R_{>0}^n$,
\[
  \sigma_1\geq\ldots\geq\sigma_m,\qquad
  \rho_1\geq\ldots\geq\rho_n,\qquad
  \sum_{j=1}^m\sigma_j=\sum_{i=1}^n\rho_i=1,
\]
and define the \emph{weighted norms} $|\cdot|_{\pmb\sigma}$ and $|\cdot|_{\pmb\rho}$ by
\[
  |\vec x|_{\pmb\sigma}=\max_{1\leq j\leq m}|x_j|^{1/\sigma_j}\qquad\text{ for }\vec x=(x_1,\ldots,x_m),
\]
\[
  |\vec y|_{\pmb\rho}=\max_{1\leq i\leq n}|y_i|^{1/\rho_i}\qquad\,\ \ \text{ for }\vec y=(y_1,\ldots,y_n).\
\]
Consider the system of inequalities
\begin{equation}\label{eq:system_with_weights}
  \begin{cases}
    |\vec x|_{\pmb\sigma}\leq t \\
    |\Theta\vec x-\vec y|_{\pmb\rho}\leq t^{-\gamma}
  \end{cases}.
\end{equation}

\begin{definition}\label{def:weighted_exponents}
  The supremum of real $\gamma$ satisfying the condition that there exist arbitrarily large $t$ such that (resp. for every $t$ large enough) the system \eqref{eq:system_with_weights} admits a nonzero solution $(\vec x,\vec y)\in\Z^m\oplus\Z^n$ is called the \emph{regular} (resp. \emph{uniform}) \emph{weighted Diophantine exponent} of $\Theta$ and is denoted by $\omega_{\pmb\sigma,\pmb\rho}(\Theta)$ (resp. $\hat\omega_{\pmb\sigma,\pmb\rho}(\Theta)$).
\end{definition}

It is easily verified that in the case of ``trivial'' weights, i.e. when all the $\sigma_j$ are equal to $1/m$ and all the $\rho_i$ are equal to $1/n$, we are dealing with ordinary Diophantine exponents, as in this case we have
\[
  \omega_{\pmb\sigma,\pmb\rho}(\Theta)=\frac nm\omega(\Theta)
  \qquad\text{ and }\qquad
  \hat\omega_{\pmb\sigma,\pmb\rho}(\Theta)=\frac nm\hat\omega(\Theta).
\]
Minkowski's convex body theorem applied to the system \eqref{eq:system_with_weights} gives ``trivial'' inequalities
\[
  \omega_{\pmb\sigma,\pmb\rho}(\Theta)
  \geq
  \hat\omega_{\pmb\sigma,\pmb\rho}(\Theta)
  \geq1,
\]
analogous to \eqref{eq:belpha_matrix_trivial_inequalities}, that hold for every choice of $\pmb\sigma,\pmb\rho$.

\subsection{Schmidt's Bad-conjecture}

In the case $m=1$, $n=2$ we can observe a very close relationship between Diophantine approximation with weights and Littlewood's conjecture, which relates to multiplicative Diophantine approximation (the formulation of this conjecture is given in the beginning of Section \ref{sec:mult}). As we noted in Section \ref{subsec:multi_bad}, the negation of this conjecture is equivalent to the statement that for $m=1$, $n=2$ multiplicatively badly approximable matrices exist.

\begin{definition} \label{def:weighted_BA}
  A matrix $\Theta$ is called \emph{badly approximable with weights $\pmb\sigma,\pmb\rho$} if there is a constant $c>0$ depending only on $\Theta$ such that for each pair $(\vec x,\vec y)\in\Z^m\oplus\Z^n$ with nonzero $\vec x$ we have
  \[
    |\vec x|_{\pmb\sigma}|\Theta\vec x-\vec y|_{\pmb\rho}\geq c.
  \]
\end{definition}

In paper \cite{schmidt_luminy_1982} Schmidt noted that Davenport's construction from paper \cite{davenport_1962} can be easily modified to prove that for $m=1$, $n=2$ there exist badly approximable matrices $\Theta$ with arbitrary weights $\pmb\rho=(\rho_1,\rho_2)$ and conjectured that for any two distinct sets of weights $\pmb\rho'$ and $\pmb\rho''$  there exist matrices that are simultaneously badly approximable with weights $\pmb\rho'$ and with weights $\pmb\rho''$. It is easy to see that the existence of a counterexample to this conjecture, i.e the existence of two sets of weights $\pmb\rho'$ and $\pmb\rho''$ such that any matrix $\Theta\in\R^{2\times1}$ is not badly approximable with at least one of them, implies Littlewood's conjecture.

In 2011 Badziahin, Pollington, and Velani published the paper \cite{badziahin_annals_2011}, where they proved Schmidt's conjecture formulated above. Note that not only did they prove the existence of matrices $\Theta\in\R^{2\times1}$ that are badly approximable with weights $\pmb\rho'$ and with weights $\pmb\rho''$, but they also proved that there are pretty many such matrices. More specifically, they proved that for any given $k$ sets of weights $\pmb\rho^{(1)},\ldots,\pmb\rho^{(k)}$ the Hausdorff dimension of the set of matrices that are badly approximable with each of these $k$ sets of weights equals $2$.

\subsection{Transference theorems}\label{subsec:weighted_transference}

The first transference inequality for Diophantine approximation with weights generalising Dyson's inequality \eqref{eq:dyson_transference} was obtained in paper \cite{ghosh_marnat_Pisa_2020}. In this paper the authors showed that we have
\begin{equation}\label{eq:nedo_weighted_Dyson}
  \omega_{\pmb\rho,\pmb\sigma}(\tr\Theta)\geq
  \frac{(m+n-1)\big(\rho_n^{-1}+\sigma_m^{-1}\omega_{\pmb\sigma,\pmb\rho}(\Theta)\big)+\sigma_1^{-1}\big(\omega_{\pmb\sigma,\pmb\rho}(\Theta)-1\big)}
       {(m+n-1)\big(\rho_n^{-1}+\sigma_m^{-1}\omega_{\pmb\sigma,\pmb\rho}(\Theta)\big)-\rho_1^{-1}\big(\omega_{\pmb\sigma,\pmb\rho}(\Theta)-1\big)}\,.
\end{equation}
As it turned out later, inequality \eqref{eq:nedo_weighted_Dyson} is not optimal. It was improved in paper \cite{german_mathmatika_2020}.

\begin{theorem}[O.\,G., 2020]\label{t:weighted_Dyson}
  For every matrix $\Theta\in\R^{n\times m}$ and arbitrary weights $\pmb\sigma,\pmb\rho$ we have
  \begin{equation}\label{eq:weighted_Dyson}
    \omega_{\pmb\rho,\pmb\sigma}(\tr\Theta)\geq
    \frac{\big(\rho_n^{-1}-1\big)+\sigma_m^{-1}\omega_{\pmb\sigma,\pmb\rho}(\Theta)}
         {\rho_n^{-1}+\big(\sigma_m^{-1}-1\big)\omega_{\pmb\sigma,\pmb\rho}(\Theta)}\,.
  \end{equation}
\end{theorem}

In that same paper \cite{german_mathmatika_2020} a transference theorem for uniform weighted exponents is proved. It generalises Theorem \ref{t:my_inequalities}.

\begin{theorem}[O.\,G., 2020]\label{t:weighted_German}
  For every matrix $\Theta\in\R^{n\times m}$, $m+n\geq3$, and arbitrary weights $\pmb\sigma,\pmb\rho$ we have
  \begin{equation}\label{eq:weighted_German}
    \hat\omega_{\pmb\rho,\pmb\sigma}(\tr\Theta)\geq
    \begin{cases}
      \dfrac{1-\sigma_m\hat\omega_{\pmb\sigma,\pmb\rho}(\Theta)^{-1}}{1-\sigma_m}
      \hskip 3mm
      \text{ if }
      \hat\omega_{\pmb\sigma,\pmb\rho}(\Theta)\geq\sigma_m/\rho_n \\
      \hskip 2mm
      \dfrac{1-\rho_n\vphantom{1^{\big|}}}{1-\rho_n\hat\omega_{\pmb\sigma,\pmb\rho}(\Theta)}
      \hskip 6.3mm
      \text{ if }
      \hat\omega_{\pmb\sigma,\pmb\rho}(\Theta)\leq\sigma_m/\rho_n
    \end{cases}.
  \end{equation}
\end{theorem}

As before, if some denominator happens to be equal to zero, we assume the respective expression to take value $+\infty$.

\subsubsection{Embedding into $\R^{m+n}$} \label{subsec:embedding_into_R_d_weighted}

We can prove Theorems \ref{t:weighted_Dyson}, \ref{t:weighted_German} following the scheme we discussed in Section \ref{subsec:embedding_into_R_d}. Let us set, as before,
\[
  d=m+n,
\]
and consider the same the lattices
\begin{equation}\label{eq:lattices_weighted}
  \La=\La(\Theta)=
  \begin{pmatrix}
    \vec I_m & \\
    -\Theta  & \vec I_n
  \end{pmatrix}
  \Z^d,
  \qquad
  \La^\ast=\La^\ast(\Theta)=
  \begin{pmatrix}
    \vec I_m & \tr\Theta \\
    & \vec I_n
  \end{pmatrix}
  \Z^d.
\end{equation}
Instead of parallelepipeds \eqref{eq:t_gamma_family_arbitrary_nm}, \eqref{eq:s_delta_family_arbitrary_nm}, let us consider parallelepipeds
\begin{align}
  \label{eq:t_gamma_family_weighted}
  \cP(t,\gamma) & =\Bigg\{\,\vec z=(z_1,\ldots,z_d)\in\R^d \ \Bigg|
                        \begin{array}{l}
                          |(z_1,\ldots,z_m)|_{\pmb\sigma}\leq t \\
                          |(z_{m+1},\ldots,z_d)|_{\pmb\rho}\leq t^{-\gamma}
                        \end{array} \Bigg\}, \\
  \label{eq:s_delta_family_weighted}
  \cQ(s,\delta) & =\Bigg\{\,\vec z=(z_1,\ldots,z_d)\in\R^d \ \Bigg|
                        \begin{array}{l}
                          |(z_1,\ldots,z_m)|_{\pmb\sigma}\leq s^{-\delta} \\
                          |(z_{m+1},\ldots,z_d)|_{\pmb\rho}\leq s
                        \end{array} \Bigg\}.
\end{align}
Then
\begin{equation*}
\begin{aligned}
  \omega_{\pmb\sigma,\pmb\rho}(\Theta) & =
  \sup\bigg\{ \gamma\geq1 \,\bigg|\ \forall\,t_0\in\R\,\ \exists\,t>t_0:\text{\,we have }\cP(t,\gamma)\cap\La\neq\{\vec 0\} \bigg\}, \\
  \hat\omega_{\pmb\sigma,\pmb\rho}(\Theta) & =
  \sup\bigg\{ \gamma\geq1 \,\bigg|\ \exists\,t_0\in\R:\ \forall\,t>t_0\text{ we have }\cP(t,\gamma)\cap\La\neq\{\vec 0\} \bigg\}, \\
  \omega_{\pmb\rho,\pmb\sigma}(\tr\Theta) & =
  \sup\bigg\{ \delta\geq1\ \bigg|\ \forall\,s_0\in\R\,\ \exists\,s>s_0:\text{\,we have }\cQ(s,\delta)\cap\La^\ast\neq\{\vec 0\} \bigg\}, \\
  \hat\omega_{\pmb\rho,\pmb\sigma}(\tr\Theta) & =
  \sup\bigg\{ \delta\geq1\ \bigg|\ \exists\,s_0\in\R:\ \forall\,s>s_0\text{ we have }\cQ(s,\delta)\cap\La^\ast\neq\{\vec 0\} \bigg\}.
\end{aligned}
\end{equation*}

\subsubsection{Deriving the theorem on regular exponents}

For $s>1$ and $\delta\geq1$ set
\begin{equation}\label{eq:Q_is_contained_in_P_star_weighted}
  t=s^{(\sigma_m^{-1}+(\rho_n^{-1}-1)\delta)/(\sigma_m^{-1}+\rho_n^{-1}-1)},
  \qquad
  \gamma=\frac{\big(\sigma_m^{-1}-1\big)+\rho_n^{-1}\delta}
              {\sigma_m^{-1}+\big(\rho_n^{-1}-1\big)\delta}\,.
\end{equation}
Such a correspondence provides the inequalities
\[
  \begin{aligned}
    s^{-\delta\sigma_j} & \leq t^{-\sigma_j+1-\gamma},\qquad j=1,\ldots,m, \\
    s^{\rho_i} & \leq t^{\gamma\rho_i+1-\gamma},\qquad\,\ i=1,\ldots,n,
  \end{aligned}
\]
which imply that
\begin{equation}\label{eq:Q_subset_P}
  \cQ(s,\delta)\subseteq
  \cP(t,\gamma)^\ast\ ,
\end{equation}
since, according to Definition \ref{def:pseudo_compound} (see Section \ref{subsubsec:pseudocompounds_and_dual_lattices}), the compound of $\cP(t,\gamma)$ is determined by
\begin{equation*}
  \cP(t,\gamma)^\ast=\Bigg\{\,\vec z=(z_1,\ldots,z_d)\in\R^d \ \Bigg|
                             \begin{array}{lr}
                               |z_j|\leq t^{-\sigma_j+1-\gamma}, & j=1,\ldots,m \\
                               |z_{m+i}|\leq t^{\gamma\rho_i+1-\gamma}, & i=1,\ldots,n
                             \end{array} \Bigg\}.
\end{equation*}
Applying again Mahler's theorem in disguise of Theorem \ref{t:mahler_reformulated}, we get
\begin{equation}\label{eq:weighted_dyson_transference_implication}
  \cQ(s,\delta)\cap\La^\ast\neq\{\vec 0\}
  \implies
  (d-1)\cP(t,\gamma)\cap\La\neq\{\vec 0\}.
\end{equation}
Thus,
\begin{equation*}
  \omega_{\pmb\rho,\pmb\sigma}(\tr\Theta)\geq\delta
  \implies
  \omega_{\pmb\sigma,\pmb\rho}(\Theta)\geq\gamma=
  \frac{\big(\sigma_m^{-1}-1\big)+\rho_n^{-1}\delta}
       {\sigma_m^{-1}+\big(\rho_n^{-1}-1\big)\delta}\,,
\end{equation*}
whence
\[
  \omega_{\pmb\sigma,\pmb\rho}(\Theta)\geq
  \frac{\big(\sigma_m^{-1}-1\big)+\rho_n^{-1}\omega_{\pmb\rho,\pmb\sigma}(\tr\Theta)}
       {\sigma_m^{-1}+\big(\rho_n^{-1}-1\big)\omega_{\pmb\rho,\pmb\sigma}(\tr\Theta)}\,.
\]
Swapping the triple $(\pmb\sigma,\pmb\rho,\Theta)$ for the triple $(\pmb\rho,\pmb\sigma,\tr\Theta)$, we get \eqref{eq:weighted_Dyson}.

\subsubsection{Deriving the theorem on uniform exponents}

Let us apply the ``nodes'' and ``leaves'' construction described in Section \ref{subsubsec:nodes_and_leaves}. Instead of parallelepipeds $\cQ_r$ determined by \eqref{eq:Q_r}, let us consider parallelepipeds
\begin{equation*}
  \cQ_r=\Bigg\{\,\vec z=(z_1,\ldots,z_d) \in\R^d \ \Bigg|
                 \begin{array}{l}
                   |(z_1,\ldots,z_m)|_{\pmb\sigma}\leq (hH/r)^{-\alpha} \\
                   |(z_{m+1},\ldots,z_d)|_{\pmb\rho}\leq r
                 \end{array} \Bigg\}.
\end{equation*}
Respectively, we define parallelepipeds $\cP(t,\gamma)$ and $\cQ(s,\delta)$ not by \eqref{eq:t_gamma_family}, \eqref{eq:s_delta_family}, but by \eqref{eq:t_gamma_family_weighted}, \eqref{eq:s_delta_family_weighted}. Let us use Fig. \ref{fig:nodes_and_leaves} with the agreement that now $u$ and $v$ denote $|(z_{m+1},\ldots,z_d)|_{\pmb\rho}$ and $|(z_1,\ldots,z_m)|_{\pmb\sigma}$ respectively.

Upon fixing $s>1$ and $\delta\geq1$, we define $t$ and $\gamma$ by \eqref{eq:Q_is_contained_in_P_star_weighted} and set
\[
  h=s,\qquad
  \beta=\delta,\qquad
  \alpha=
  \begin{cases}
    \dfrac{\sigma_m^{-1}+(\rho_n^{-1}-1)\delta}{\sigma_m^{-1}+\rho_n^{-1}-1}
    \quad\ \text{ if }\ \delta\geq\dfrac{\rho_n(\sigma_m^{-1}-1)}{\sigma_m(\rho_n^{-1}-1)} \\
    \dfrac{(\sigma_m^{-1}+\rho_n^{-1}-1)\delta}{(\sigma_m^{-1}-1)+\rho_n^{-1}\delta}
    \quad\ \text{ if }\ \delta\leq\dfrac{\rho_n(\sigma_m^{-1}-1)}{\sigma_m(\rho_n^{-1}-1)}
    \vphantom{\dfrac{\Big|}{}}
  \end{cases}.
\]
With such a choice of parameters, the quantities $\gamma$ and $\alpha$ are related by
\begin{equation}\label{eq:gamma_via_alpha_weighted}
  \gamma=
  \begin{cases}
    \dfrac{1-\rho_n\alpha^{-1}}
          {1-\rho_n}        \hskip 7mm\text{ if }\alpha\geq\rho_n/\sigma_m \\
    \hskip 2mm
    \dfrac{1-\sigma_m\vphantom{1^{\textstyle|}}}
          {1-\sigma_m\alpha}\hskip 8.2mm\text{ if }\alpha\leq\rho_n/\sigma_m
  \end{cases}.
\end{equation}
Arguing in the manner of Section \ref{subsubsec:proof_of_uniform_transference}, we get the implication
\[
  \hat\omega_{\pmb\rho,\pmb\sigma}(\tr\Theta)\geq\alpha
  \implies
  \hat\omega_{\pmb\sigma,\pmb\rho}(\Theta)\geq\gamma,
\]
whence
\[
  \hat\omega_{\pmb\sigma,\pmb\rho}(\Theta)\geq
  \begin{cases}
    \dfrac{1-\rho_n\hat\omega_{\pmb\rho,\pmb\sigma}(\tr\Theta)^{-1}}
          {1-\rho_n}        \hskip 7mm\text{ if }\hat\omega_{\pmb\rho,\pmb\sigma}(\tr\Theta)\geq\rho_n/\sigma_m \\
    \hskip 2mm
    \dfrac{1-\sigma_m\vphantom{1^{\textstyle|}}}
          {1-\sigma_m\hat\omega_{\pmb\rho,\pmb\sigma}(\tr\Theta)}\hskip 8.2mm\text{ if }\hat\omega_{\pmb\rho,\pmb\sigma}(\tr\Theta)\leq\rho_n/\sigma_m
  \end{cases}.
\]
Swapping the triple $(\pmb\sigma,\pmb\rho,\Theta)$ for the triple $(\pmb\rho,\pmb\sigma,\tr\Theta)$, we come to \eqref{eq:weighted_German}.

\subsection{Marnat's result}

For $m=1$, $n=2$ (i.e. in the case most interesting for Littlewood's conjecture) we have $\sigma_1=1$, $\rho_n=\rho_2$, and $1-\rho_n=\rho_1$, which simplifies the appearance of inequality \eqref{eq:weighted_German} significantly. Swapping $(\pmb\sigma,\pmb\rho,\Theta)$ for $(\pmb\rho,\pmb\sigma,\tr\Theta)$, as usually, gives one more inequality. Besides that, in this case, same as in the problem of classical simultaneous approximation (see inequality \eqref{eq:alpha_leq_1} in Section \ref{subsec:regular_and_uniform_exponents}), we have the inequality $\hat\omega_{\pmb\sigma,\pmb\rho}(\Theta)\leq\rho_1^{-1}$ for the uniform exponent (under the condition that not all the components of $\Theta$ are rational). Therefore, Theorem \ref{t:weighted_German} gives the following statement.

\begin{corollary}\label{cor:weighted_German_m=1_n=2}
  Let $m=1$, $n=2$ and let $\Theta\notin\Q^{2\times1}$. Then
  \begin{equation}\label{eq:weighted_German_m=1_n=2}
    \begin{aligned}
      & \rho_1\hat\omega_{\pmb\sigma,\pmb\rho}(\Theta)+\rho_2\hat\omega_{\pmb\rho,\pmb\sigma}(\tr\Theta)^{-1}\geq1, \\
      & \rho_2\hat\omega_{\pmb\sigma,\pmb\rho}(\Theta)+\rho_1\hat\omega_{\pmb\rho,\pmb\sigma}(\tr\Theta)^{-1}\leq1.
      \vphantom{1^{|}}
    \end{aligned}
  \end{equation}
  If, moreover, $\theta_{11}$ is irrational, then
  \begin{equation}\label{eq:weighted_German_m=1_n=2_extra}
    \rho_1\hat\omega_{\pmb\sigma,\pmb\rho}(\Theta)\leq1.
  \end{equation}
\end{corollary}

In paper \cite{marnat_twisted} Marnat proved that for every positive $b<(3\rho_1)^{-1}$ and every $a$ satisfying the inequalities
\[
  \begin{aligned}
    & \rho_1a+\rho_2b\geq1, \\
    & \rho_2a+\rho_1b\leq1, \\
    & \rho_1a\leq1
  \end{aligned}
\]
there exist continuously many matrices $\Theta\in\R^{2\times1}$ such that $\hat\omega_{\pmb\sigma,\pmb\rho}(\Theta)=a$ and $\hat\omega_{\pmb\sigma,\pmb\rho}(\tr\Theta)=b^{-1}$. It follows that, if $\hat\omega_{\pmb\sigma,\pmb\rho}(\tr\Theta)>3\rho_1$ and $\theta_{11}$ is irrational, inequalities \eqref{eq:weighted_German_m=1_n=2}, \eqref{eq:weighted_German_m=1_n=2_extra} are sharp.

Particularly, this result by Marnat implies that already for $m+n=3$ there is no analogue of Jarn\'{\i}k's inequality in the case of nontrivial weights.

\subsection{Parametric geometry of numbers}\label{subsec:parametric_geometry_of_numbers_weighted}

The approach described in Sections \ref{subsubsec:parametric_geometry_of_numbers}, \ref{subsubsec:parametric_geometry_of_numbers_arbitrary_nm} is also applicable in Diophantine approximation with weights. We define the lattices $\La$ and $\La^\ast$ by \eqref{eq:lattices_weighted} -- same as in Section \ref{subsubsec:parametric_geometry_of_numbers_arbitrary_nm}. As for the subspace of the space of parameters $\cT$ corresponding to the problem of Diophantine approximation with nontrivial weights, unlike the case of trivial weights, we need more than just a single one-dimensional subspace.

Given weights $\pmb\sigma=(\sigma_1,\ldots,\sigma_m)$, $\pmb\rho=(\rho_1,\ldots,\rho_n)$, let us set
\begin{equation*}
\begin{array}{l}
  \vec e_1=\vec e_1(\pmb\sigma,\pmb\rho)=
  (1-d\sigma_1,\ldots,1-d\sigma_m,\underbrace{1,\ldots,1}_{n}), \\
  \vec e_2=\vec e_2(\pmb\sigma,\pmb\rho)=
  (\underbrace{1,\ldots,1}_{m},1-d\rho_n,\ldots,1-d\rho_1).
  \vphantom{\Big|}
\end{array}
\end{equation*}
If every $\sigma_j$ equals $1/m$ and every $\rho_i$ equals $1/n$, then vectors $\vec e_1$ and $\vec e_2$ are proportional. But if the weights are nontrivial, then $\vec e_1$ and $\vec e_2$ span a two-dimensional subspace of $\cT$. For each $\gamma,\delta\in\mathbb{R}$ let us set
\begin{equation*}
\begin{aligned}
  \pmb\mu_{\pmb\sigma,\pmb\rho}(\gamma) & =\gamma\vec e_2-\vec e_1, \\
  \pmb\mu^\ast_{\pmb\sigma,\pmb\rho}(\delta) & =\delta\vec e_1-\vec e_2.
\end{aligned}
\end{equation*}
Every $\pmb\mu_{\pmb\sigma,\pmb\rho}(\gamma)$ and every $\pmb\mu^\ast_{\pmb\sigma,\pmb\rho}(\delta)$ determine one-dimensional subspaces of $\cT$. Respectively, we consider the family of paths $\gT_{\pmb\sigma,\pmb\rho}(\gamma)$ determined by the mappings $s\mapsto s\pmb\mu_{\pmb\sigma,\pmb\rho}(\gamma)$ and the family of paths $\gT^\ast_{\pmb\sigma,\pmb\rho}(\delta)$ determined by the mappings $s\mapsto s\pmb\mu^\ast_{\pmb\sigma,\pmb\rho}(\delta)$. We have lower and upper Schmidt--Summerer exponents of the first and of the second types defined for these paths (see Definition \ref{def:schmidt_psi}). It is shown in paper \cite{german_monatshefte_2022} that weighted Diophantine exponents are related to Schmidt--Summerer exponents by
\begin{equation}\label{eq:weighted_in_terms_of_schmimmerer_inequality}
\begin{array}{l}
  \omega_{\pmb\sigma,\pmb\rho}(\Theta)\geq\gamma\
  \iff\
  \bpsi_1\big(\La,\gT_{\pmb\sigma,\pmb\rho}(\gamma)\big)\leq1-\gamma, \\
  \vphantom{1^{\big\vert}}
  \omega_{\pmb\sigma,\pmb\rho}(\Theta)\leq\gamma\
  \iff\
  \bpsi_1\big(\La,\gT_{\pmb\sigma,\pmb\rho}(\gamma)\big)\geq1-\gamma, \\
  \vphantom{1^{\big\vert}}
  \hat\omega_{\pmb\sigma,\pmb\rho}(\Theta)\geq\gamma\
  \iff\
  \apsi_1\big(\La,\gT_{\pmb\sigma,\pmb\rho}(\gamma)\big)\leq1-\gamma, \\
  \vphantom{1^{\big\vert}}
  \hat\omega_{\pmb\sigma,\pmb\rho}(\Theta)\leq\gamma\
  \iff\
  \apsi_1\big(\La,\gT_{\pmb\sigma,\pmb\rho}(\gamma)\big)\geq1-\gamma
\end{array}
\end{equation}
and, similarly,
\begin{equation}\label{eq:weighted_in_terms_of_schmimmerer_inequality_tr}
\begin{array}{l}
  \omega_{\pmb\rho,\pmb\sigma}(\tr\Theta)\geq\delta\
  \iff\
  \bpsi_1\big(\La^\ast,\gT^\ast_{\pmb\sigma,\pmb\rho}(\delta)\big)\leq1-\delta, \\
  \vphantom{1^{\big\vert}}
  \omega_{\pmb\rho,\pmb\sigma}(\tr\Theta)\leq\delta\
  \iff\
  \bpsi_1\big(\La^\ast,\gT^\ast_{\pmb\sigma,\pmb\rho}(\delta)\big)\geq1-\delta, \\
  \vphantom{1^{\big\vert}}
  \hat\omega_{\pmb\rho,\pmb\sigma}(\tr\Theta)\geq\delta\
  \iff\
  \apsi_1\big(\La^\ast,\gT^\ast_{\pmb\sigma,\pmb\rho}(\delta)\big)\leq1-\delta, \\
  \vphantom{1^{\big\vert}}
  \hat\omega_{\pmb\rho,\pmb\sigma}(\tr\Theta)\leq\delta\
  \iff\
  \apsi_1\big(\La^\ast,\gT^\ast_{\pmb\sigma,\pmb\rho}(\delta)\big)\geq1-\delta.
\end{array}
\end{equation}
In that same paper it is shown that if $\delta\geq1$ and $\gamma$ is related to $\delta$ by the right-hand equality \eqref{eq:Q_is_contained_in_P_star_weighted}, then
\begin{equation}\label{eq:weighted_dyson_schmimmerered}
  \bpsi_1\big(\La^\ast,\gT^\ast_{\pmb\sigma,\pmb\rho}(\delta)\big)\leq 1-\delta
  \implies
  \bpsi_1\big(\La,\gT_{\pmb\sigma,\pmb\rho}(\gamma)\big)\leq 1-\gamma.
\end{equation}
In view of \eqref{eq:weighted_in_terms_of_schmimmerer_inequality} and \eqref{eq:weighted_in_terms_of_schmimmerer_inequality_tr}, statement \eqref{eq:weighted_dyson_schmimmerered} is nothing else but inequality \eqref{eq:weighted_Dyson} written in terms of Schmidt--Summerer exponents. Same as with Khintchine's and Dyson's inequalities, this approach enables splitting inequality \eqref{eq:weighted_Dyson} into a chain of inequalities between intermediate exponents. The respective definitions and formulations can be found in paper \cite{german_monatshefte_2022}.

\section{Diophantine exponents if lattices}

In previous Sections we considered a matrix $\Theta\in\R^{n\times m}$ and dealt with the question how fast the vector $\Theta\vec x-\vec y$ can tend to zero if $\vec x$ and $\vec y$ are assumed to range through integer vectors. In other words, we considered $n$ linear forms with coefficients written in the rows of the matrix
\[
  \begin{pmatrix}
    \Theta  & -\vec I_n
  \end{pmatrix}=
  \begin{pmatrix}
    \theta_{11} & \cdots & \theta_{1m} & -1 & \cdots & 0 \\
    \vdots & \ddots & \vdots & \vdots & \ddots & \vdots \\
    \theta_{n1} & \cdots & \theta_{nm} & 0 & \cdots & -1
  \end{pmatrix}
\]
and studied their values at integer points. We considered several ways to measure the ``magnitude'' of a vector. The first one uses the sup-norm, the second one uses the geometric mean of the absolute values of coordinates, the third one uses weighted norms. But each time we worked with $n$ linear forms in $d=m+n$ variables. I.e. every time the number of linear forms was strictly less than the dimension of the ambient space. This section is devoted to problems concerning $d$-tuples of linear forms in $d$ variables.

Let $L_1,\ldots,L_d$ be linearly independent linear forms in $d$ variables. Consider the lattice
\[
  \La=\Big\{\big(L_1(\vec u),\ldots,L_d(\vec u)\big)\,\Big|\,\vec u\in\Z^d \Big\}.
\]
Using a norm to measure the ``magnitude'' of the elements of $\La$ is not very productive, as such a norm is bounded away from zero at nonzero points of $\La$. But the geometric mean of the absolute values of coordinates leads to very interesting and sometimes very difficult problems. As in Section \ref{sec:mult}, let us set for each $\vec z=(z_1,\ldots,z_d)\in\R^d$
\[
  \Pi(\vec z)=\prod_{\begin{subarray}{c}1\leq i\leq d\end{subarray}}|z_i|^{1/d}.
\]
We remind also that $|\cdot|$ denotes the sup-norm.

\begin{definition} \label{d:lattice_exponent}
  Supremum of real $\gamma$ such that the inequality
  \[
    \Pi(\vec z)\leq |\vec z|^{-\gamma}
  \]
  admits infinitely many solutions in $\vec z\in\La$ is called the \emph{Diophantine exponent} of $\La$ and is denoted as $\omega(\La)$.
\end{definition}

\subsection{Spectrum of lattice exponents}

It follows from Minkowski's convex body theorem that for every lattice $\La$ the ``trivial'' inequality
\[
  \omega(\La)\geq0
\]
holds. It is clear that $\omega(\La)=0$ whenever the functional $\Pi(\vec x)$ is bounded away from zero at nonzero points of $\La$. For instance, this is the case if $\La$ is the lattice of a complete module in a totally real algebraic extension of $\Q$, i.e. if
\begin{equation}\label{eq:algebraic_lattice}
  \La=
  \begin{pmatrix}
    \sigma_1(\omega_1) & \sigma_1(\omega_2) & \cdots & \sigma_1(\omega_d) \\
    \sigma_2(\omega_1) & \sigma_2(\omega_2) & \cdots & \sigma_2(\omega_d) \\
    \vdots             & \vdots             & \ddots & \vdots             \\
    \sigma_d(\omega_1) & \sigma_d(\omega_2) & \cdots & \sigma_d(\omega_d)
  \end{pmatrix}\Z^d,
\end{equation}
where $\omega_1,\ldots,\omega_d$ is a basis of a totally real extension $E$ of $\Q$ of degree $d$ and $\sigma_1,\ldots,\sigma_d$ are the embeddings of $E$ into $\R$. Such lattices are often called \emph{algebraic}. A detailed account on algebraic lattices can be found, for instance, in Borevich and Shafarevich's book \cite{borevich_shafarevich}.

There is a wider class of lattices with $\omega(\La)=0$. Their existence is provided by Schmidt's famous subspace theorem, which was published by Schmidt in 1972 in paper \cite{schmidt_subspace}. Its formulation and proof can also be found in Bombieri and Gubler's book \cite{bombieri_gubler}. It is shown in papers \cite{skriganov_1998}, \cite{german_lattice_transference} that if the coefficients of linear forms $L_1,\ldots,L_d$ are algebraic and for every $k$-tuple $(i_1,\ldots,i_k)$, $1\leq i_1<\ldots<i_k\leq d$, $1\leq k\leq d$, the coefficients of the multivector $L_{i_1}\wedge\ldots\wedge L_{i_k}$ are linearly independent over $\Q$, then by the subspace theorem we have $\omega(\La)=0$.

It was shown in paper \cite{german_lattice_transference} that weakening the condition of linear independence formulated above enables constructing examples for $d\geq3$ of lattices with exponents attaining values
\begin{equation} \label{eq:spectrum_from_schmidt}
  \frac{\,ab\,}{cd}\,,\qquad
  \begin{array}{l}
    a,b,c\in\N, \\
    a+b+c=d.
  \end{array}
\end{equation}
It is natural to conjecture that there are lattices $\La$ with any prescribed nonnegative $\omega(\La)$. For $d=2$ this is easily proved with the help of continued fractions (see Section \ref{subsec:continued_fractions}). For $d\geq3$ this question is still open. It is proved in paper \cite{german_lattice_exponents_spectrum} that the interval
\begin{equation} \label{eq:spectrum_from_german}
  \bigg[3-\frac{d}{(d-1)^2}\,,\,+\infty\bigg]
\end{equation}
is contained in the spectrum of the values of $\omega(\La)$. The question whether there are positive numbers in this spectrum that are not in the interval \eqref{eq:spectrum_from_german} and do not equal any of the numbers \eqref{eq:spectrum_from_schmidt} is still open for $d\geq3$.

\subsection{Lattices with positive norm minimum}

For lattices, the analogue of the property of a number to be badly approximable is the property to have positive \emph{norm minimum}.

\begin{definition}
  Let $\La$ be a full rank lattice in $\R^d$. Its \emph{norm minimum} is defined as
  \[
    N(\La)=\inf_{\vec z\in\La\backslash\{\vec 0\}}\Pi(\vec z)^d.
  \]
\end{definition}
As we said in the previous Section, if $\La$ is an algebraic lattice, then $N(\La)>0$. Cassels and Swinnerton-Dyer showed in paper \cite{cassels_swinnerton_dyer} that Littlewood's conjecture (see the formulation in the beginning of Section \ref{sec:mult}) can be derived from the following assumption, which is an analogue of Oppenheim's conjecture on quadratic forms for decomposable forms of degree $d$. The latter was proved by Margulis in the end of 1980s (see \cite{margulis_FML}).

\begin{cassels}
  Let $d\geq3$ and let $\La$ be a full rank lattice in $\R^d$. Then the condition $N(\La)>0$ is equivalent to the existence of a non-degenerate diagonal operator $D$ such that $D\La$ is an algebraic lattice, i.e. satisfies \eqref{eq:algebraic_lattice}.
\end{cassels}

The main tool in the proof of the fact that Littlewood's conjecture follows from Cassels--Swinnerton-Dyer's conjecture is Mahler's compactness criterion (see \cite{cassels_GN} and \cite{cassels_swinnerton_dyer}).

For $d=2$ the statement of Cassels--Swinnerton-Dyer's conjecture is not true. Indeed, in view of Proposition \ref{prop:CF_gap_estimates} and geometric interpretation of continued fractions described in Section \ref{subsec:continued_fractions}, in the two-dimensional case the lattice
\[
  \La=\Big\{\big(L_1(\vec u),L_2(\vec u)\big)\,\Big|\,\vec u\in\Z^2 \Big\}
\]
has positive norm minimum if and only if the ratio of the coefficients of $L_1$ and the ratio of the coefficients of $L_2$ are badly approximable numbers. Thus, if these two ratios are irrational numbers with bounded partial quotients, but the sequences of these partial quotients are not periodic, then $N(\La)>0$, but there is no diagonal operator $D$ such that $D\La$ is an algebraic lattice.

Due to Dirichlet's theorem on algebraic units, algebraic lattices have rich symmetry groups consisting of diagonal operators. Details can be found, for instance, in books \cite{borevich_shafarevich}, \cite{karpenkov_book} and also in paper \cite{german_tlyustangelov_izvestiya_2021}. This observation generalises to the multidimensional case the fact that quadratic irrationalities have periodic continued fractions. Thus, Cassels--Swinnerton-Dyer's conjecture claims that for $d\geq3$ the inequality $N(\La)>0$ (which, we remind, is an analogue of the property of a number to be badly approximable) implies that $\La$ has a rich group of symmetries consisting of diagonal operators. This statement admits reformulation in terms of multidimensional continued fractions. To be more specific, in terms of \emph{Klein polyhedra}.

\begin{definition}\label{def:klein_polyhedra}
  Let $\La$ be a full rank lattice in $\R^d$. In each orthant $\cO$, let us consider the convex hull
  \[
    \cK(\La,\cO)=\conv\big(\La\cap\cO\backslash\{\vec 0\}\big).
  \]
  Each of the $2^d$ convex hulls thus obtained is called a \emph{Klein polyhedron}.
\end{definition}

In the case $d=2$ Definition \ref{def:klein_polyhedra} differs from Definition \ref{def:klein_polygons} of Klein polygons given in Section \ref{subsubsec:klein_polygons}. The reason is that Klein polygons from Definition \ref{def:klein_polygons} are a geometric interpretation of a continued fraction of a single number. If we take a pair of distinct numbers, then the respective pair of lines divides the plane into four angles and the respective pair of linear forms determines a full rank lattice in $\R^2$. This construction precisely corresponds to Definition \ref{def:klein_polyhedra}, but it interprets geometrically the continued fractions of a pair of numbers. We discuss it in detail in Section \ref{subsec:lattice_exponents_vs_growth}.

We say that a lattice $\La$ is \emph{irrational} if none of its nonzero points have zero coordinates. If a lattice is irrational, then each of the $2^d$ Klein polyhedra is a generalised polyhedron, i.e. a set whose intersection with any given compact polyhedron is itself a compact polyhedron. This fact is proved in paper \cite{moussafir_A_polyhedra}. Particularly, in this case each vertex belongs to a finite number of faces. If, at the same time, the dual lattice $\La^\ast$ is irrational, then, as it was shown in \cite{german_norm_minima_I}, each of the $2^d$ Klein polyhedra has only compact faces. Particularly, each face has a finite number of vertices.

We showed in Section \ref{subsubsec:klein_polygons} that in the two-dimensional case the role of partial quotients is played by edges and vertices of Klein polygons equipped with linear integer invariants such as integer length of an edge and integer angle at a vertex. In the multidimensional case it is natural to consider facets (i.e. $(d-1)$-dimensional faces) and edge stars of vertices of a Klein polyhedron equipped with some appropriate linear integer invariants. In papers \cite{german_norm_minima_I}, \cite{german_norm_minima_II}, \cite{german_bordeaux}, \cite{german_bigushev_2022} \emph{determinants} of facets and edge stars are considered as such invariants.

\begin{definition} \label{d:det_F_and_det_starv}
  Let $\La$ be a full rank lattice in $\R^d$. Let $\La$ and $\La^\ast$ be irrational. Let $\cK$ be one of the $2^d$ Klein polyhedra of $\La$.

  \textup{(i)} Let $F$ be an arbitrary facet of $\cK$ and let $\vec v_1,\ldots,\vec v_k$ be the vertices of $F$. The \emph{determinant} of $F$ is defined as
  \[
    \det F=\sum_{1\leq i_1<\ldots<i_d\leq k}|\det(\vec v_{i_1},\ldots,\vec v_{i_d})|.
  \]

  \textup{(ii)} Let $\vec v$ be an arbitrary vertex of $\cK$ and let it be incident to $k$ edges. Let $\vec r_1,\ldots,\vec r_k$ be primitive vectors of $\La$ parallel to these edges. The \emph{determinant} of the edge star $\starv$ of $\vec v$ is defined as
  \[
    \det\starv=\sum_{1\leq i_1<\ldots<i_d\leq k}|\det(\vec r_{i_1},\ldots,\vec r_{i_d})|.
  \]
\end{definition}

It is clear that if $d=2$ and $\det\La=1$, then the determinants of edges are equal to their integer lengths and the determinants of the edge stars of vertices are equal to the integer angles between the respective edges.

According to Corollary \ref{cor:BA_vs_boundedness_of_partial_quotients} in Section \ref{subsubsec:exponent_vs_growth}, the property of an irrational number to be badly approximable is equivalent to the property to have bounded partial quotients. In papers \cite{german_norm_minima_I}, \cite{german_norm_minima_II}, \cite{german_bordeaux} the following multidimensional generalisation of this statement is proved.

\begin{theorem}[O.\,G., 2006]\label{t:main_normin_II}
  Let $\La$ be an irrational full rank lattice in $\R^d$. Then the following statements are equivalent:

  \textup{(i)} $N(\La)>0$;

  \textup{(ii)} the facets of all the $2^n$ Klein polyhedra of $\La$ have bounded determinants (bounded by a common constant);

  \textup{(iii)} the facets and the edge stars of the vertices of the Klein polyhedron corresponding to $\La$ and the positive orthant have bounded determinants (bounded by a common constant).
\end{theorem}

The aforementioned corollary to Dirichlet's theorem on algebraic units can be specified. If $\La$ is an algebraic lattice in $\R^d$, then there is a group isomorphic to $\Z^{d-1}$ which consists of diagonal operators with positive diagonal elements that preserve the lattice. In this case the boundary of each of the $2^d$ Klein polyhedra equipped with any linear integer invariants that can be of interest has $(d-1)$-periodic combinatorial structure.

Thus, Theorem \ref{t:main_normin_II} provides the following reformulation of Cassels--Swinnerton-Dyer's conjecture.

\begin{cassels}[Reformulated]
  Let $d\geq3$ and let $\La$ be an irrational full rank lattice in $\R^d$. Let $\cK$ be one of the $2^n$ Klein polyhedra of $\La$. Then the following statements are equivalent:

  \textup{(i)} the facets and the edge stars of the vertices of $\cK$ have bounded determinants (bounded by a common constant);

  \textup{(ii)} the combinatorial structure of the boundary of $\cK$ equipped with determinants of facets and determinants of edge stars of vertices is $(d-1)$-periodic.
\end{cassels}

Thus, we can say that Cassels--Swinnerton-Dyer's conjecture claims that for $d\geq3$ the boundedness of multidimensional analogues of partial quotients implies the periodicity of the respective multidimensional continued fraction.

\subsection{Relation to the growth of multidimensional analogues of partial quotients}\label{subsec:lattice_exponents_vs_growth}

If statement \textup{(ii)} or statement \textup{(iii)} of Theorem \ref{t:main_normin_II} does not hold, then $N(\La)=0$, i.e. $\Pi(\vec z)$ attains arbitrarily small values at nonzero $\vec z\in\La$. In the two-dimensional case we can recall Corollary \ref{cor:omega_vs_growth} from Section \ref{subsubsec:exponent_vs_growth}, which relates the Diophantine exponent of a number to the growth of its partial quotients. The correspondence described in Section \ref{subsubsec:geometric_algorithm} enables to formulate a statement equivalent to Corollary \ref{cor:omega_vs_growth} in terms of Diophantine exponents of lattices in dimension $2$. Let us extend Definition \ref{def:klein_polygons} of Klein polygons to the case of two numbers $\theta_1$ and $\theta_2$ by considering the convex hulls of nonzero integer points in the four angles determined by the lines $y=\theta_1x$ and $y=\theta_2x$ (see Fig. \ref{fig:KP_and_CF}).

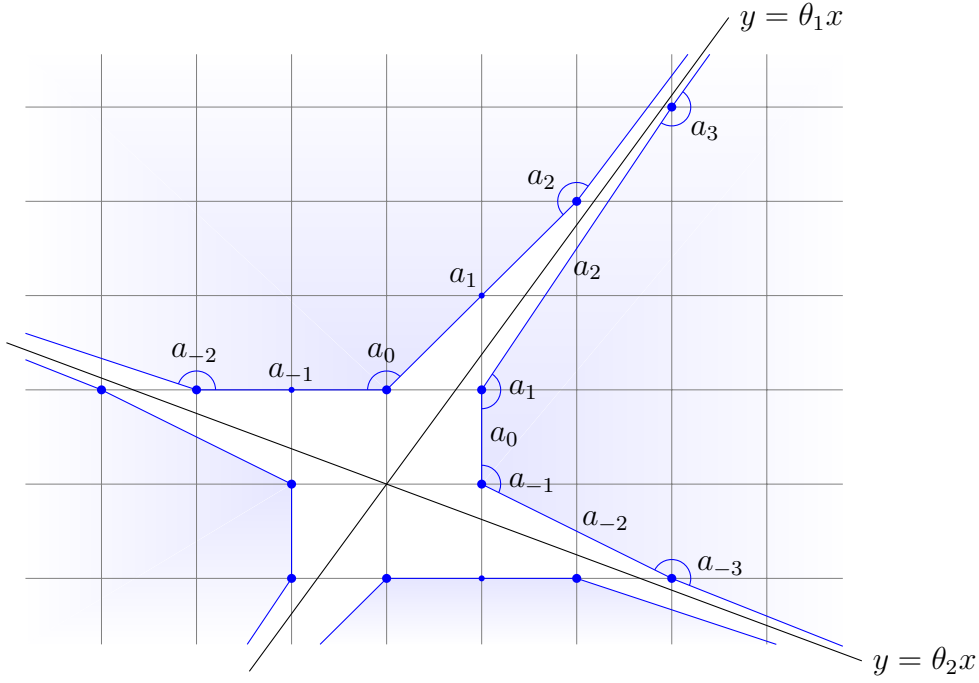
\begin{figure}[h]
  \centering
  \begin{tikzpicture}[scale=1.25]
    \fill[blue!10!,path fading=north]
        (4.8,4.56) -- (3+0.4,4+0.4*7/5) -- (3,4) -- (1,1) -- (1,0) -- (4.8,4.56) -- cycle;
    \fill[blue!10!,path fading=east]
        (4.8,4.56) -- (1,0) -- (3,-1) -- (3+1.8,-1-1.8*2/5) -- cycle;
    \fill[blue!10!,path fading=north]
        (2+1.17,3+1.17*4/3) -- (2,3) -- (0,1) -- (-2-1.8,3+1.17*4/3) -- cycle;
    \fill[blue!10!,path fading=west]
        (-2-1.8,3+1.17*4/3) -- (0,1) -- (-2,1) -- (-2-1.8,1+1.8/3) -- cycle;
    \fill[blue!10!,path fading=south]
        (-0.7,-1.7) -- (0,-1) -- (2,-1) -- (4.1,-1.7) -- cycle;
    \fill[blue!10!,path fading=west]
        (-3.8,-1.7) -- (-1,0) -- (-3,1) -- (-3.8,1+0.8*2/5) -- cycle;
    \fill[blue!10!,path fading=south]
        (-1-0.7*2/3,-1.7) -- (-1,-1) -- (-1,0) -- (-3.8,-1.7) -- cycle;

    \draw[very thin,color=gray,scale=1] (-3.8,-1.7) grid (4.8,4.56);

    \draw[color=black] plot[domain=-13/9:3.6] (\x, {11*\x/8}) node[right]{$y=\theta_1x$};
    \draw[color=black] plot[domain=-4:5] (\x, {-3*\x/8}) node[right]{$y=\theta_2x$};

    \draw[color=blue] (3+0.4,4+0.4*7/5) -- (3,4) -- (1,1) -- (1,0) -- (3,-1) -- (3+1.8,-1-1.8*2/5);
    \draw[color=blue] (2+1.17,3+1.17*4/3) -- (2,3) -- (0,1) -- (-2,1) -- (-2-1.8,1+1.8/3);
    \draw[color=blue] (-0.7,-1.7) -- (0,-1) -- (2,-1) -- (4.1,-1.7);
    \draw[color=blue] (-1-0.7*2/3,-1.7) -- (-1,-1) -- (-1,0) -- (-3,1) -- (-3.8,1+0.8*2/5);

    \node[fill=blue,circle,inner sep=1.2pt] at (3,4) {};
    \node[fill=blue,circle,inner sep=1.2pt] at (1,1) {};
    \node[fill=blue,circle,inner sep=1.2pt] at (1,0) {};
    \node[fill=blue,circle,inner sep=1.2pt] at (3,-1) {};
    \node[fill=blue,circle,inner sep=1.2pt] at (2,3) {};
    \node[fill=blue,circle,inner sep=1.2pt] at (0,1) {};
    \node[fill=blue,circle,inner sep=1.2pt] at (-2,1) {};

    \node[fill=blue,circle,inner sep=1.2pt] at (-1,-1) {};
    \node[fill=blue,circle,inner sep=1.2pt] at (-1,0) {};
    \node[fill=blue,circle,inner sep=1.2pt] at (-3,1) {};
    \node[fill=blue,circle,inner sep=1.2pt] at (0,-1) {};
    \node[fill=blue,circle,inner sep=1.2pt] at (2,-1) {};

    \node[fill=blue,circle,inner sep=0.8pt] at (1,-1) {};
    \node[fill=blue,circle,inner sep=0.8pt] at (-1,1) {};
    \node[fill=blue,circle,inner sep=0.8pt] at (1,2) {};

    \node[right] at (1-0.03,0.5) {$a_0$};
    \node[right] at (2-2/13,2.5-3/13) {$a_2$};

    \node[right] at (2-0.06,-0.4) {$a_{-2}$};

    \node[above left] at (1.08,2-0.02) {$a_1$};

    \node[above] at (-1,0.95) {$a_{-1}$};

    \draw[blue] ([shift=({atan(-1/2)}:0.2)]1,0) arc (atan(-1/2):90:0.2);
    \draw[blue] ([shift=({atan(-2/5)}:0.2)]3,-1) arc (atan(-2/5):90+atan(2):0.2);
    \draw[blue] ([shift=(-90:0.2)]1,1) arc (-90:atan(3/2):0.2);
    \draw[blue] ([shift=({-90-atan(2/3)}:0.2)]3,4) arc (-90-atan(2/3):atan(7/5):0.2);
    \draw[blue] ([shift=({atan(4/3)}:0.2)]2,3) arc (atan(4/3):225:0.2);
    \draw[blue] ([shift=(45:0.2)]0,1) arc (45:180:0.2);
    \draw[blue] ([shift=(0:0.2)]-2,1) arc (0:90+atan(3):0.2);

    \node[right] at (3.15,-0.88) {$a_{-3}$};
    \node[right] at (1.17,0) {$a_{-1}$};
    \node[right] at (1.17,1) {$a_1$};
    \node[right] at (3.08,3.75) {$a_3$};
    \node[left] at (2-0.1,3.23) {$a_2$};
    \node[above] at (-0.05,1.18) {$a_0$};
    \node[above] at (-2.02,1.16) {$a_{-2}$};

  \end{tikzpicture}
  \caption{Extended Klein polygons for}
          {$\theta_1=[a_0;a_1,a_2,\ldots],\ -1/\theta_2=[a_{-1};a_{-2},a_{-3},\ldots]$}
  \label{fig:KP_and_CF}
\end{figure}

In the case of a single number, a partial quotient $a_{k+1}$ equals the integer angle at the vertex $\vec v_k$ of the respective Klein polygon, whereas the denominator $q_k$ differs from $|\vec v_k|$ by a bounded factor. Hence, in the case of two numbers, relation \eqref{eq:omega_vs_growth} from Corollary \ref{cor:omega_vs_growth} provides the equality
\begin{equation} \label{eq:mu_vs_alpha}
  \max\big(\omega(\theta_1),\omega(\theta_2)\big)=
  1+\displaystyle\limsup_{\substack{ \vec w\in \cW,\ |\vec w|>1 \\ |\vec w|\to\infty }}\frac{\log(\det\starw)}{\log|\vec w|}\,,
\end{equation}
where $\cW$ denotes the set of vertices of all the four extended Klein polygons.

Now, let us address the Diophantine exponent of the lattice $\La=A\Z^2$, where
\begin{equation*}
  A=
  \begin{pmatrix}
    \theta_1 & -1 \\
    \theta_2 & -1
  \end{pmatrix}.
\end{equation*}
Denote by $\cV$ the set of vertices of all the four Klein polyhedra of $\La$. Since $\cV=A\cW$, we have the following relations for $\vec v=A(\vec w)$:
\begin{equation}\label{eq:v_asymp_w}
  |\vec v|\asymp|\vec w|,
  \qquad
  \det\starw\asymp\det\starv.
\end{equation}
If $\vec w=(q,p)$, then for $|q|$ large enough
\begin{equation}\label{eq:gamma_vs_1+2gamma}
  \Pi(\vec v)\asymp|\vec v|^{-\gamma}
  \iff
  \min\big(|q\theta_1-p|,|q\theta_2-p|\big)\asymp q^{-1-2\gamma}.
\end{equation}
Further, if all the vertices of all the Klein polyhedra of $\La$ satisfy the condition $\Pi(\vec v)\geq|\vec v|^{-\gamma}$, then so do all nonzero points of $\La$. This leads us to the following equality:
\begin{equation*}
  \omega(\La)=\sup\Big\{\gamma\in\R\ \Big|\,\exists\,\infty\,\vec v\in\cV:\,\Pi(\vec v)\leq|\vec v|^{-\gamma} \Big\}.
\end{equation*}
Hence by \eqref{eq:gamma_vs_1+2gamma}
\begin{equation*}
  \max\big(\omega(\theta_1),\omega(\theta_2)\big)=1+2\omega(\La).
\end{equation*}
Taking into account \eqref{eq:mu_vs_alpha} and \eqref{eq:v_asymp_w}, we get
\begin{equation}\label{eq:omega_vs_starv_2dim}
  \omega(\La)=
  \frac12
  \displaystyle\limsup_{\substack{ \vec v\in \cV,\ |\vec v|>1 \\ |\vec v|\to\infty }}\frac{\log(\det\starv)}{\log|\vec v|}\,.
\end{equation}

Thus, relation \eqref{eq:omega_vs_starv_2dim} is an analogue of equality \eqref{eq:omega_vs_growth} from Corollary \ref{cor:omega_vs_growth} for Diophantine exponents of lattices of rank $2$. It can be generalised to the multidimensional case, but what is currently known is far from an exhaustive generalisation. There is only one respective result. It was obtained in paper \cite{german_bigushev_2022}.

\begin{theorem}[E.\,R.\,Bigushev, O\,G., 2022]\label{t:omega_vs_det_starv}
  Let $\La$ be an irrational full rank lattice in $\R^3$. Let $\cV$ denote the set of vertices of all the eight Klein polyhedra of $\La$. Then
  \begin{equation} \label{eq:omega_vs_det_starv}
    \omega(\La)
    \leq
    \frac23
    \displaystyle\limsup_{\substack{ \vec v\in \cV,\ |\vec v|>1 \\ |\vec v|\to\infty }}\frac{\log(\det\starv)}{\log|\vec v|}\,.
  \end{equation}
\end{theorem}

It is shown in that same paper \cite{german_bigushev_2022} that the main statement upon which the proof of Theorem \ref{t:omega_vs_det_starv} is based cannot be inverted due to the existence of a counterexample. Nevertheless, it may be so that this statement is too local and its locality can be ``relaxed'' a bit. Another way that can probably help to turn inequality \eqref{eq:omega_vs_det_starv} into an equality is to find a more appropriate quantitative characteristic of multidimensional partial quotients than determinants.

\subsection{Transference theorem}

In the case of Diophantine exponents of lattices, there also exists a transference theorem. And same as with the theorems of Khintchine, Dyson, and their generalisations for multiplicative exponents and Diophantine approximation with weights, the simplest way to prove it is to derive it from Mahler's theorem in disguise of Theorem \ref{t:mahler_reformulated}.

As before, we denote by $\La^\ast$ the dual lattice of $\La$.
For $d=2$ the lattice $\La^\ast$ coincides up to homothety with $\La$ rotated by $\pi/2$. Thus, in the two-dimensional case we have the obvious equality $\omega(\La)=\omega(\La^\ast)$. The following theorem is proved in paper \cite{german_lattice_transference}.


\begin{theorem} \label{t:lattice_transference}
  If one of the exponents $\omega(\La)$, $\omega(\La^\ast)$ equals zero, then so does the other one. If they are both nonzero, then
  \begin{equation} \label{eq:lattice_transference}
    (d-1)^{-2}
    \leq
    \frac{1+\omega(\La)^{-1}}{1+\omega(\La^\ast)^{-1}}
    \leq
    (d-1)^2.
  \end{equation}
\end{theorem}

Let us show how to derive Theorem \ref{t:lattice_transference} from Theorem \ref{t:mahler_reformulated}.

For each tuple $\pmb\lambda=(\lambda_1,\ldots,\lambda_d)\in\R_+^d$, let us define the parallelepiped $\cP(\pmb\lambda)$ by
\begin{equation}\label{eq:prallelepipeds_lattice_exp}
  \cP(\pmb\lambda)=\Big\{\,\vec z=(z_1,\ldots,z_d)\in\R^d \ \Big|\ |z_i|\leq\lambda_i,\ i=1,\ldots,d \Big\}.
\end{equation}
Let us assign to each tuple $\pmb\lambda$ the tuple $\pmb\lambda^\ast=(\lambda_1^\ast,\ldots,\lambda_d^\ast)$,
\begin{equation}\label{eq:lambda_ast_lattice_exp}
  \lambda_i^\ast=\lambda_i^{-1}\Pi(\pmb\lambda)^d,
  \qquad i=1,\ldots,d.
\end{equation}
Then $\cP(\pmb\lambda)^\ast=\cP(\pmb\lambda^\ast)$, i.e. $\cP(\pmb\lambda^\ast)$ is the pseudocompound of $\cP(\pmb\lambda)$ (see Definition \ref{def:pseudo_compound} in Section \ref{subsubsec:pseudocompounds_and_dual_lattices}). By Theorem \ref{t:mahler_reformulated}
\begin{equation}\label{eq:lattice_transference_implication}
  \cP(\pmb\lambda^\ast)\cap\La^\ast\neq\{\vec 0\}
  \implies
  (d-1)\cP(\pmb\lambda)\cap\La\neq\{\vec 0\}.
\end{equation}
Next, for each $\delta\geq0$, let us set
\[
  \gamma=\frac{\delta}{(d-1)^2+d(d-2)\delta}\,.
\]
If $\Pi(\pmb\lambda^\ast)=|\pmb\lambda^\ast|^{-\delta}$, then
\[
  \Pi(\pmb\lambda)=
  \Pi(\pmb\lambda^\ast)^{1/(d-1)}=
  |\pmb\lambda^\ast|^{-\delta/(d-1)}
\]
and
\[
  |\pmb\lambda|\leq
  \frac{\Pi(\pmb\lambda)^d}{\displaystyle\min_{1\leq i\leq d}|\lambda_i|^{d-1}}=
  \frac{|\pmb\lambda^\ast|^{d-1}}{\Pi(\pmb\lambda)^{d(d-2)}}=
  |\pmb\lambda^\ast|^{d-1+\frac{d(d-2)}{d-1}\delta}.
\]
Therefore,
\begin{equation}\label{eq:lattice_transference_delta_gamma_implication}
  \Pi(\pmb\lambda^\ast)=|\pmb\lambda^\ast|^{-\delta}
  \implies
  \Pi(\pmb\lambda)\leq|\pmb\lambda|^{-\gamma}
\end{equation}
By \eqref{eq:lattice_transference_delta_gamma_implication} and \eqref{eq:lattice_transference_implication} we have
\[
  \omega(\La^\ast)\geq\delta
  \implies
  \omega(\La)\geq\gamma=\frac{\delta}{(d-1)^2+d(d-2)\delta}\,,
\]
whence
\[
  \omega(\La)\geq\frac{\omega(\La^\ast)}{(d-1)^2+d(d-2)\omega(\La^\ast)}
\]
or, equivalently,
\[
  \frac{1+\omega(\La)^{-1}}{1+\omega(\La^\ast)^{-1}}
  \leq
  (d-1)^2.
\]
Since $\La$ and $\La^\ast$ can be swapped, we get both inequalities \eqref{eq:lattice_transference}.

\subsection{Parametric geometry of numbers}

In Section \ref{subsubsec:parametric_geometry_of_numbers} we described an approach developed by Schmidt and Summerer. A fundamental role in this approach is played by the space of parameters
\[
  \cT=\Big\{ \pmb\tau=(\tau_1,\ldots,\tau_d)\in\R^d\, \Big|\,\tau_1+\ldots+\tau_d=0 \Big\}.
\]
We saw in Sections \ref{subsubsec:parametric_geometry_of_numbers}, \ref{subsubsec:parametric_geometry_of_numbers_arbitrary_nm} that the problem of approximating zero with the values of several linear forms requires considering one-dimensional subspaces of $\cT$, along which $\pmb\tau$ is to tend to infinity. As for Diophantine approximation with nontrivial weights, we saw in Section \ref{subsec:parametric_geometry_of_numbers_weighted} that it requires working with two-dimensional subspaces of $\cT$. In the context of problems related to Diophantine exponents of lattices, we need to work with the whole $\cT$.

Let us set for each $\pmb\tau\in\cT$
\[
  \vert\pmb\tau\vert_+=\max_{1\leq i\leq d}\tau_i.
\]
Every norm in $\mathbb{R}^d$ induces a norm in $\cT$. For instance, the sup-norm $\vert\cdot\vert$. But the functional $\vert\cdot\vert_+$ is not a norm for $d\geq3$, as the respective ``unit balls'' are simplices, which are obviously not symmetric about the origin. Nevertheless, the functional $\vert\cdot\vert_+$ plays a very important role, as it is the image of the sup-norm under the logarithmic mapping: if $\vec z=(z_1,\ldots,z_d)$, $z_i>0$, $i=1,\ldots,d$, and $\vec z_{\log}=(\log z_1,\ldots,\log z_d)$, then
\[
  \log\vert\vec z\vert=\vert\vec z_{\log}\vert_+\,.
\]
It is clear that $\vert\cdot\vert_+$ generates a \emph{monotone exhaustion} of $\cT$, i.e. $\cT=\bigcup_{\lambda>0}\cS(\lambda)$, where
\[
  \cS(\lambda)=\Big\{ \pmb\tau\in\cT\, \Big\vert\,f(\pmb\tau)\leq\lambda \Big\}\qquad\text{ for }\lambda>0,
\]
each set $\cS(\lambda)$ is compact, and $\cS(\lambda')$ is contained in the (relative) interior of $\cS(\lambda'')$ for $\lambda'<\lambda''$. Particularly, $\vert\pmb\tau\vert_+\to+\infty$ is equivalent to $\vert\pmb\tau\vert\to+\infty$.

Diophantine exponents of lattices require the following modification of Definition \ref{def:schmidt_psi}.

\begin{definition}\label{def:psi_lattice}
  Given a lattice $\La$ and $k\in\{1,\ldots,d\}$, the quantities
  \[
    \bpsi_k(\La)=
    \liminf_{\substack{\vert\pmb\tau\vert\to\infty \\ \pmb\tau\in\cT}}
    \frac{L_k(\La,\pmb\tau)}{\vert\pmb\tau\vert_+}
    \qquad\text{ and }\qquad
    \apsi_k(\La)=
    \limsup_{\substack{\vert\pmb\tau\vert\to\infty \\ \pmb\tau\in\cT}}
    \frac{L_k(\La,\pmb\tau)}{\vert\pmb\tau\vert_+}
  \]
  are respectively called the \emph{$k$-th lower} and \emph{upper Schmidt--Summerer exponents of the first type}.
\end{definition}

\begin{definition}\label{def:Psi_lattice}
  Given a lattice $\La$ and $k\in\{1,\ldots,d\}$, the quantities
  \[
    \bPsi_k(\La)=
    \liminf_{\substack{\vert\pmb\tau\vert\to\infty \\ \pmb\tau\in\cT}}
    \frac{S_k(\La,\pmb\tau)}{\vert\pmb\tau\vert_+}
    \qquad\text{ and }\qquad
    \aPsi_k(\La)=
    \limsup_{\substack{\vert\pmb\tau\vert\to\infty \\ \pmb\tau\in\cT}}
    \frac{S_k(\La,\pmb\tau)}{\vert\pmb\tau\vert_+}
  \]
  are respectively called the \emph{$k$-th lower} and \emph{upper Schmidt--Summerer exponents of the second type}.
\end{definition}

It was shown in paper \cite{german_monatshefte_2022} that the Diophantine exponent of $\La$ is related to $\bpsi_1(\La)$ by
\[
  \big(1+\omega(\La)\big)\big(1+\bpsi_1(\La)\big)=1.
\]
This relation is equivalent to
\[
  1+\omega(\La)^{-1}=-\bpsi_1(\La)^{-1}.
\]
Thus, inequalities \eqref{eq:lattice_transference} can be rewritten as follows:
\[
  (d-1)^{-2}
  \leq
  \frac{\bpsi_1(\La^\ast)}{\bpsi_1(\La)}
  \leq
  (d-1)^2.
\]
And same as with Khintchine's and Dyson's inequalities, as well as with transference inequalities for Diophantine approximation with weights, this approach enables to split inequality \eqref{eq:lattice_transference} into a chain of inequalities between intermediate exponents. The respective definitions and formulations can be found in paper \cite{german_monatshefte_2022}.

\subsection{Uniform exponent and Mordell's constant}

The attentive reader may have noticed that in all the previous Sections we discussed two types of Diophantine exponents -- \emph{regular} ones and \emph{uniform} ones. But in the case of Diophantine exponents of lattices, we have been discussing only regular exponents so far, not mentioning their uniform analogue even once. There is a reason for this, and we discuss this reason in the current Section. The thing is that, same as with the problem of approximating a real number by rationals (see Section \ref{subsec:uniform_triviality}), the uniform exponent of a lattice is ``trivial'' -- it equals either infinity, or zero. It is natural to consider as a uniform analogue of $\omega(\La)$ the quantity $\hat\omega(\La)$ equal to the supremum of real $\gamma$ such that for every $t$ large enough and every $\pmb\lambda\in\R_+^d$ satisfying the relations
\[
  |\pmb\lambda|=t,
  \qquad
  \Pi(\pmb\lambda)=t^{-\gamma}
\]
the parallelepiped $\cP(\pmb\lambda)$ determined by \eqref{eq:prallelepipeds_lattice_exp} contains a nonzero point of $\La$.

\subsubsection{Mordell's constant}

In 1937, in his paper \cite{mordell_LMS_1937}, Mordell asked whether there exists a constant $c$ depending only on dimension $d$ such that for every lattice $\La$ with determinant $1$ there is a tuple $\pmb\lambda=(\lambda_1,\ldots,\lambda_d)\in\R_+^d$ satisfying the following two conditions:
\[
  \prod_{i=1}^{d}\lambda_i=c
  \qquad\text{ and }\qquad
  \cP(\pmb\lambda)\cap\La=\{\vec 0\},
\]
where $\cP(\pmb\lambda)$ is the parallelepiped determined by \eqref{eq:prallelepipeds_lattice_exp}. In that same year, Siegel proved in a letter to Mordell that the answer to this question is positive. A slightly different proof was proposed by Davenport in \cite{davenport_AA_1936}. Davenport proved the following statement.

\begin{theorem}[H.\,Davenport, 1937]\label{t:davenport_1937}
  Let $\La$ be a full rank lattice in $\R^d$ with determinant $1$. Consider an arbitrary $d$-tuple $\pmb\lambda=(\lambda_1,\ldots,\lambda_d)\in\R_+^d$ such that
  \[
    \prod_{i=1}^{d}\lambda_i=1.
  \]
  Let $\mu_k=\mu_k\big(\cP(\pmb\lambda),\La\big)$, $k=1,\ldots,d$, denote the $k$-th successive minimum of the parallelepiped $\cP(\pmb\lambda)$ w.r.t. $\La$. Then there is a positive constant $c$ depending only on $d$ and a permutation $k_1,\ldots,k_d$ of the indices $1,\ldots,d$ such that the interior of $\cP(\pmb\lambda')$, where $\pmb\lambda'=(\lambda'_1,\ldots,\lambda'_d)$,
  \[
    \lambda'_i=(d!\cdot c)^{1/d}\mu_{k_i}\lambda_i,
    \qquad
    i=1,\ldots,d,
  \]
  contains no nonzero points of $\La$.
\end{theorem}

It follows from Minkowski's second theorem (see Theorem \ref{t:minkowski_minima} in Section \ref{subsubsec:parametric_geometry_of_numbers}) that the tuple $\pmb\lambda'$ from Theorem \ref{t:davenport_1937} satisfies
\[
  \prod_{i=1}^{d}\lambda'_i\geq c.
\]
Thus, Theorem \ref{t:davenport_1937}, indeed, gives a positive answer to Mordell's question. Respectively, if we define for each lattice $\La$ in $\R^d$ with determinant $1$ its \emph{Mordell's constant} by
\[
  \kappa(\La)=
  \sup_{\substack{\pmb\lambda=(\lambda_1,\ldots,\lambda_d)\in\R_+^d \\ \cP(\pmb\lambda)\cap\La=\{\vec 0\}}}
  \prod_{i=1}^{d}\lambda_i,
\]
then the quantity
\[
  \kappa_d=\inf_{\La:\,\det\La=1}\kappa(\La)
\]
is positive for each $d$. The best known estimate
\[
  \kappa_d\geq
  d^{-d/2}
\]
belongs to Shapira and Weiss \cite{shapira_weiss_bound_for_mordell}. Their result confirmed the conjecture made by Ramharter \cite{ramharter_2000} that $\limsup_{d\to\infty}\kappa_d^{1/d\log d}>0$.

\subsubsection{Triviality of the uniform exponent}

Theorem \ref{t:davenport_1937} has local nature. It enables therefore to find ``empty'' parallelepipeds of fixed volume arbitrarily ``far''. Let us assume for simplicity that the first coordinate axis
\[
  \Big\{ \vec z=(z_1,\ldots,z_d)\in\R^d\ \Big|\ z_i=0,\ i=2,\ldots,d \Big\}
\]
is not a subset of any subspace $\cL$ of dimension $k<d$ such that $\cL\cap\La$ is a lattice of rank $k$. Then, for every $\e>0$, the cylinder
\[
  \Big\{ \vec z=(z_1,\ldots,z_d)\in\R^d\ \Big|\ |z_i|<\e,\ i=2,\ldots,d \Big\}
\]
contains $d$ linearly independent points of $\La$. Thus, for every such $\e$, there is $\lambda\in\R_+$ and a tuple $\pmb\lambda=(\lambda_1,\ldots,\lambda_d)\in\R_+^d$ such that 
\[
  \lambda_1=\lambda,
  \qquad
  \lambda_i=\lambda^{-1/(d-1)},
  \qquad
  i=2,\ldots,d,
\]
and
\[
  \mu_d\big(\cP(\pmb\lambda),\La\big)\lambda^{-1/(d-1)}<\e.
\]
By Theorem \ref{t:davenport_1937} the parallelepiped $\mu_d\big(\cP(\pmb\lambda),\La\big)\cP(\pmb\lambda)$ contains a parallelepiped $\cP(\pmb\lambda')$ of fixed volume that does not contain nonzero points of $\La$. It immediately follows that $\hat\omega(\La)=0$.

\end{document}